\documentclass[10pt, twoside]{amsart}
\usepackage{amsmath,mathtools}
\usepackage{amsmath, amsthm, amssymb, amsfonts, enumerate}
\usepackage[colorlinks=true,linkcolor=blue,citecolor=blue]{hyperref}
\usepackage{dsfont}
\usepackage{color}
\usepackage{geometry}
\usepackage{todonotes}
\usepackage{epstopdf}
\usepackage{bbm}
\usepackage{geometry}
\usepackage[utf8]{inputenc}
\usepackage{bm}
\usepackage{amsfonts}
\usepackage{amsfonts}
\usepackage{textcomp}
\usepackage{amssymb}
\usepackage{float}
\usepackage{tikz}
\usepackage{epsfig}
\usepackage{amsmath}
\usepackage[english]{babel}
\usepackage{a4}
\usepackage{enumerate}
\usepackage{tcolorbox}
\usepackage{soul}
\newcommand{\stkout}[1]{\ifmmode\text{\sout{\ensuremath{#1}}}\else\sout{#1}\fi}
\newtheorem{theorem}{Theorem}[section]

\newtheorem{remark}[theorem]{Remark}
\newtheorem{hypothesis}[theorem]{Assumption}
\newtheorem{lemma}[theorem]{Lemma}
\newtheorem{prop}[theorem]{Proposition}
\newtheorem{corollary}[theorem]{Corollary}
\newtheorem{definition}[theorem]{Definition}

\newtheorem{example}[theorem]{Example}

\newcommand{\cp}[2]{\langle#1,#2\rangle}
\newcommand{\ds}{\displaystyle}
\newcommand{\tr}{\mathop{\mathrm{Tr}}\nolimits}

\def \R{\mathbb{R}}

\def\erre{\mathbb{R}}

\definecolor{red}{rgb}{1.0,0.0,0.0}
\def\red#1{{\textcolor{red}{#1}}}
\definecolor{blu}{rgb}{0.0,0.0,1.0}
\def\blu#1{{\textcolor{blu}{#1}}}
\definecolor{gre}{rgb}{0.03,0.50,0.03}

\usepackage[T1]{fontenc}

\title[Optimal control of stochastic delay differential equations: Optimal feedback controls.]{Optimal control of stochastic delay differential equations: Optimal feedback controls}

\author[de Feo]{Filippo de Feo}
\address{F.~de Feo: Department of Economics and Finance,  Luiss Guido Carli University, Rome, Italy, Department of Mathematics, Politecnico di Milano, Milan, Italy, and Institut für Mathematik, Technische Universität Berlin, Berlin, Germany}
\email{\href{mailto:defeo@math.tu-berlin.de}{defeo@math.tu-berlin.de}}

\author[\'{S}wi\k{e}ch]{Andrzej \'{S}wi\k{e}ch}
\address{A. \'{S}wi\k{e}ch: School of Mathematics, Georgia Institute of Technology, Atlanta, GA 30332, USA}
\email{\href{mailto:swiech@math.gatech.edu}{swiech@math.gatech.edu}}

\numberwithin{equation}{section}

\begin{document}

\begin{abstract}
In this manuscript, we study optimal control problems for stochastic delay differential equations using the dynamic programming approach in Hilbert spaces via viscosity solutions of the associated Hamilton-Jacobi-Bellman equations. We show how to use the partial $C^{1,\alpha}$-regularity of the value function established in \cite{defeo_federico_swiech} to obtain optimal feedback controls. The main result of the paper is a verification theorem which provides a sufficient condition for optimality using the value function. We then discuss its applicability to the construction of optimal feedback controls. We provide an economic application of our results to stochastic optimal advertising problems.
\end{abstract}

\maketitle

%49K27: Optimality conditions for problems in abstract spaces
%49K45: Optimality conditions for problems involving randomness
%49L12:	Hamilton-Jacobi equations in optimal control and differential games
%49L20: Dynamic programming in optimal control and differential games
%49L25:	Viscosity solutions to Hamilton-Jacobi equations in optimal control and differential games
%49N35:	Optimal feedback synthesis
%60H15: SPDEs (aspects of stochastic analysis)
%93E20: Optimal stochastic control
%34K50 Stochastic functional-differential equations


\section{Introduction}

In this manuscript, we study optimal control problems for stochastic delay differential equations (SDDE) using the dynamic programming approach in Hilbert spaces and viscosity solutions of Hamilton-Jacobi-Bellman (HJB) equations \cite{fgs_book}. We consider the problem for which the state equation is a stochastic delay differential equation (SDDE) in $\mathbb{R}^n$ of the form
\begin{small}
\begin{equation*}
\begin{cases}
dy(t) = \ds b_0 \left ( y(t),\int_{-d}^0 a_1(\xi)y(t+\xi)\,d\xi ,u(t) \right)
         dt 
\ds  + \sigma_0 \left (y(t),\int_{-d}^0 a_2(\xi)y(t+\xi)\,d\xi \right)\, dW(t),\quad t\geq 0,\\
y(0)=x_0, \quad y(\xi)=x_1(\xi)\; \quad \forall \xi\in[-d,0),
\end{cases}
\end{equation*}
\end{small}
where $x_0 \in \mathbb R^n, x_1 \in L^2([-d,0];\mathbb R^n)$ and  $u(\cdot)$ is a suitable control process with values in $U \subset \mathbb R^p$. The goal is to minimize, over all admissible controls $u(\cdot)$, a cost functional
$$
J(x_0,x_1 ; u(\cdot))=\mathbb{E}\left[\int_0^{\infty} e^{-\rho t} l(y(t), u(t)) d t\right].
$$
Following \cite{defeo_federico_swiech}, 
%a similar optimal control problem for a stochastic delay differential equation, with controls also in the diffusion coefficients,
we rewrite the problem as an optimal control problem without delay for an abstract stochastic differential equation of the form 
\begin{equation}\label{eq:abstract_SDE_intro}
dY(t) = [\tilde A Y(t)+\tilde b(Y(t),u(t))]dt + \sigma(Y(t))\,dW(t), \quad t \geq 0 , \quad 
Y(0) = x:=(x_0,x_1) \in X,
\end{equation}
in the infinite dimensional Hilbert space $$X=\mathbb R^n\times L^2([-d,0];\mathbb R^n).$$ 
 In this setup, we have $$Y(t)=(y(t),y(t+\xi)_{\xi \in [-d,0]}),$$ so that $\mathbb R^n$ (the ``present'' space) is the space where we keep track of the current state of the controlled random variable $y(t)$ and $L^2([-d,0];\mathbb R^n)$ (the ``past'' space) is the space where we keep track of the relevant past part $y(t+\xi)_{\xi \in [-d,0]}$ via an unbounded maximal  dissipative operator $\tilde A$, which is  the generator of the so-called delay semigroup\footnote{to be precise, $\tilde A$ is a suitable bounded perturbation of the standard generator $A$ of the delay semigroup, see \cite{defeo_federico_swiech}}, and $\tilde b, \sigma$ are appropriate coefficients on $X$ with zero $L^2$-component (see Section \ref{sec:prelim}). We can then investigate this new equivalent problem using the dynamic programming approach  and study the value function $V$ of the problem and the associated  HJB equation, which is a partial differential equation in $X$ of the form
 \begin{equation*}
%\left\{\begin{array}{l}
\rho v(x) - \cp{\tilde A x}{Dv(x)} + H(x,Dv(x))   -  \frac{1}{2} \tr \left [ \sigma \left ( x\right) \sigma \left ( x \right) ^*  D^2v(x) \right]=0,
\quad x \in X,
%0 \leq t \leq T\\[8pt]
%v(T) = \varphi
%end{array}\right.
\end{equation*}
where the Hamiltonian $H \colon X \times X 	\to \mathbb R$ is given by
\begin{small}
\begin{align*}
H(x,p):= - x_0 \cdot p_0 + \sup_{u\in U} \Bigg \{ - b_0\left ( x_0,\int_{-d}^0 a_1(\xi)x_1(\xi)\,d\xi ,u \right) \cdot p_0
 - l(x_0,u)  \Bigg \}.
 \end{align*}
 \end{small}
Notice the presence of  the unbounded operator $\tilde A$ in the HJB equation.
It was proved in \cite{defeo_federico_swiech} that the value function $V$ is the unique viscosity solution of the HJB equation  in the sense of the definition presented in \cite{fgs_book}. Moreover it was proved in \cite{defeo_federico_swiech} that $V$ is such that for every $x_1\in L^2([-d,0];\mathbb R^n)$, $V(\cdot,x_1)\in C^{1,\alpha}_{\rm loc}(\mathbb R^n)$; hence, the derivative of $V$ with respect to the variable $x_0,$ denoted by $D_{x_0} V$, is well-defined. The procedure of rewriting the optimal control problems for SDDE in a Hilbert space as well as the main results of \cite{defeo_federico_swiech} regarding the value function and the HJB equation are recalled here in Section \ref{sec:prelim}. 
  
  The goal of this paper is to use the HJB equation and the partial regularity result for the value function to explore how they can be helpful in the construction of optimal feedback controls. It is standard to construct  an optimal feedback control from the HJB equation if the value function is smooth (see \cite{FS, yong_zhou} or \cite{fgs_book} in infinite dimension). Here, since the diffusion does not depend on the controls, the partial $C^{1,\alpha}$-regularity result allows to construct a candidate for an optimal feedback map. However the value function is not regular enough, we cannot even write It\^o's formula for $V$, some terms in the equation are not well defined, and thus we cannot follow a standard argument. Instead we employ an approximation procedure involving several layers of approximations. We work under minimal regularity assumptions on the coefficients of the problem, however we assume that the value function is so called $|\cdot|_{-1}$ semiconvex (see Section \ref{sec:approximation via inf-convolutions}; see also Section \ref{sec:prelim} for the definition of the weak norm $|\cdot|_{-1}$). We first use inf-convolutions: this allows us to obtain functions $\tilde V_\epsilon \in C^{1,1}(X_{-1})$, approximating $V$,  which are viscosity supersolutions of  perturbed infinite-dimensional HJB equations  (see Section \ref{sec:approximation via inf-convolutions}). Here $X_{-1}$ is the Hilbert space obtained as the completion of $X$ under the norm $|\cdot|_{-1}$. Then, we adopt approximations based on limits of partial convolutions, which were originally introduced in \cite{lions-infdim1} for equations with bounded terms, and which we adapt here to equations containing unbounded operators. This allows us to construct more regular functions $\tilde V_\epsilon^\eta,$ approximating $V,$ which are viscosity and pointwise supersolutions of further perturbed infinite-dimensional HJB equations (see Section \ref{sec:lions_approx}). The functions $V_\epsilon^\eta $ are in $  C^{1,1}(X_{-1})$ and admit a sort of weak Gateaux second order derivative in $X_{-1},$ so that  a non-smooth Dynkin's formula can be proved for them (see Section \ref{sec:dynkin_formula}). We can then work with the approximating functions and infinite-dimensional HJB equations and the candidate optimal feedback map to show passing to the limit that the optimal feedback map indeed allows to define an optimal feedback control.  The main result of the paper is a verification theorem (Theorem \ref{th:verification}), which provides a sufficient condition for optimality using the value function $V$. Such theorem is stated, under suitable assumptions, in the following form:  let $x \in X$ and $u^*(\cdot)$ be an admissible control and denote by $Y^*(t)=(Y_0^*(t),Y_1^*(t))$ the solution of \eqref{eq:abstract_SDE_intro} with $u(\cdot)=u^*(\cdot)$. Assume that, $\mathbb P$-a.s., for a.e. $t \geq 0,$ it holds
  $$u^*(t)\in {\rm argmax}_{u\in U} \Bigg \{ - b_0\left ( Y_0^*(t),\int_{-d}^0 a_1(\xi)Y_1^*(t)(\xi)\,d\xi ,u \right) \cdot D_{x_0}V(Y^*(t))  - l(Y_0^*(t),u)  \Bigg \}.$$
 Then the couple $(Y^*(\cdot),u^*(\cdot))$ is optimal.
We refer to Theorem \ref{th:verification}, for the precise statement.

The verification theorem naturally leads us to the construction of optimal feedback controls  in the following way: define the multivalued map $\Psi \colon X \to  \mathcal  P (U)$ by
  \begin{align*}
   \Psi(x):&={\rm argmax}_{u \in U}   \Bigg \{ - b_0\left ( x_0,\int_{-d}^0 a_1(\xi)x_1(\xi)\,d\xi ,u \right) \cdot D_{x_0}V(x) 
 - l(x_0,u)  \Bigg \} 
  \end{align*}
and assume that $\Psi$ has a measurable selection $\psi$ such that the closed loop equation
  \begin{equation*}
dY(t) = [\tilde A Y(t)+\tilde b(Y(t),\psi(Y(t)) ] dt + \sigma(Y(t)) \,dW(t), \quad Y(0) = x \in X,
\end{equation*}
admits a weak solution $Y^\psi(t)$ in some generalized reference probability space. Then, setting $u^\psi(\cdot) :=  \psi(Y^\psi (\cdot)),$ we have that the pair $(u^\psi(\cdot),Y^\psi(\cdot))$ is optimal. We refer to Corollary \ref{cor:closed_loop} for the precise statement of such result. We also show  that the value functions in the weak formulation of the optimal control problem using reference probability spaces and the ones using the so called generalized reference probability spaces (see Section \ref{sec:formul} for definitions) are the same.

In the final part of the paper, we provide an application of our results to a stochastic optimal advertising problem with delays (e.g. see \cite[Section 7]{defeo_federico_swiech}, \cite{gozzi_marinelli_2004}). In particular, we consider a controlled $1$-dimensional SDDE for the dynamics of the stock of advertising goodwill, denoted by $y(t)$, of the form
\begin{equation*}
\begin{cases}
dy(t) = \left[ a_0 y(t)+\int_{-d}^0 a_1(\xi)y(t+\xi)\,d\xi +c_0 u(t) \right]
         dt   + \sigma_0 \, dW(t),\\
y(0)=x_0, \quad y(\xi)=x_1(\xi)\; \quad \forall \xi\in[-d,0),
\end{cases}
\end{equation*}
where the control process $u(s)$ models the intensity of advertising spending. In this case, thanks to Girsanov theorem, we are able to solve the corresponding closed loop equation and construct optimal feedback controls (see Section \ref{sec:example}). We remark that the current paper as well as \cite{defeo_federico_swiech} only consider the case where the delay occurs in the state variable.

Stochastic optimal control problems have also been studied by means of maximum principle: we refer, e.g., to \cite{Chen_wu_2010,Chen_wu_2011,guatteri_masiero_2021,guatteri_masiero_2023,meng_shi_2021,oksendal_sulem,oksendal_sulem_zhang}\footnote{In \cite{oksendal_sulem_zhang} processes with jumps are also considered.} and the references therein. Indeed, optimal controls can be constructed using such approach, e.g. see \cite{oksendal_sulem}, where the delay kernel has a special  structure, which allows  to prove a finite dimensional It\^o's formula and a stochastic maximum principle in which the adjoint processes follow standard backward stochastic differential equations (BSDE) in $\mathbb R^n$, and 
 \cite{Chen_wu_2010,Chen_wu_2011,meng_shi_2021,oksendal_sulem_zhang}, where, instead, the adjoint processes follow anticipated BSDE (introduced in \cite{peng_yang}).
 
However, here we focus on the dynamic programming method. With this method, stochastic optimal control problems with delays can be approached in different ways, depending on the problem at hand.
%(see also a detailed description of existing research in optimal control problems for SDDE  via dynamic programming/HJB techniques in \cite{defeo_federico_swiech}).
If the delay kernels $a_1,a_{2}$ have a special structure (similarly to \cite{oksendal_sulem}), the HJB equation (which is intrinsically infinite-dimensional) can be  reduced to a finite-dimensional one, see, e.g.,  \cite{larssen2}. However, this is not the case in general and other approaches are needed to tackle the problem. For the approach using path-dependent viscosity solutions on spaces of continuous paths see, e.g., \cite{bayraktar_keller_2018, bayraktar_keller_2022, cosso_federico_gozzi_rosestolato_touzi, ekren_2014, ekren_2016, ekren_2016b, ren_touzi_zhang} and the references therein. We refer to \cite{fuhrman_tessitore_masiero_2010} for an infinite dimensional approach on Banach spaces of continuous functions. 
 
When the initial path is a function in $L^2([-d,0];\mathbb R^n)$, deterministic and stochastic optimal control problems for delay differential equations have been studied using the dynamic programming approach in Hilbert spaces by means of mild solutions, mild solutions in $L^2$ spaces and BSDE in 
\cite{fgs_book,gozzi_masiero_2017, gozzi_masiero_2017_b, gozzi_masiero_2022, masiero_2022, gozzi_marinelli_2004}.
Classical explicit solutions were employed in \cite{fabbri_gozzi_2008, bambi_fabbri_gozzi (2012), bambi_digirolami_federico_gozzi (2017), biffis_gozzi_prosdocini_2020, biagini_gozzi_zanella_2022, djehiche_gozzi_zanco_zanella_2022}. Viscosity solutions were first  used to deal with deterministic optimal control problems in \cite{goldys_1, goldys_2, tacconi, carlier_tahraoui_2010}. Paper \cite{defeo_federico_swiech} was the first to study optimal control problems for SDDE using the notion of the so-called $B$-continuous viscosity solution in a Hilbert space. $B$-continuous viscosity solutions in Hilbert spaces were also used in \cite{rosestolato} for Kolmogorov equations related to SDDEs and in \cite{deFeo_2023}, where stochastic optimal control problems with delays including delays in the control were studied.  Connections between  path-dependent viscosity solutions and $B$-continuous viscosity solutions are shown in \cite{ren_rosestolato}. We finally refer to \cite{zhou_2018, zhou_2019, zhou_2021} for other  approaches using appropriately defined viscosity solutions in spaces of right-continuous functions and continuous functions.

In the dynamic programming approach, classical verification theorems for stochastic optimal control problems assume smoothness of the value function and use the associated HJB equation to obtain sufficient and necessary conditions for optimality. For finite dimensional problems such results can be found in \cite{FS, yong_zhou} and for problems in an infinite dimensional Hilbert space corresponding formulations are in \cite{fgs_book}, Sections 2.5.1 and 2.5.2. When the value function is not $C^2$ verification theorems become complicated. For finite dimensional problems a viscosity solution version of the verification theorem  is in 
\cite[Chapter 5]{yong_zhou} and full proofs are in \cite{gozzi_swiech_zhou_2005} and \cite{gozzi_swiech_zhou_2010}. It is very rare for infinite dimensional problems that the value function is regular enough to apply the smooth verification theorem. Nevertheless some results exist. For deterministic problems using viscosity solution framework, we refer for instance to \cite[Chapter 6]{LY}, \cite{can-fr, fgs_2008}. In the stochastic case versions of the result from \cite{yong_zhou} appeared recently in \cite{stannat_wessels_2021, ChenLu}. Hence, to the best of our knowledge, Theorem \ref{th:verification} is the first verification theorem in the context of viscosity solutions of HJB equations for a stochastic optimal control problem with delays.
Verification theorems and optimal synthesis results using other frameworks are discussed in \cite{fgs_book}: for mild solutions in Section 4.8; for solutions in $L^2$ spaces in Section 5.5, see also \cite{federico_gozzi_2018}; for solutions using backward stochastic differential equations in Sections 6.5, 6.6 and 6.10. Such approaches were applied to problems with delays. Our approach here, based on viscosity solutions, allows us to work under different assumptions. We  refer to Remark \ref{rem:comparison_feedback_with_mild_and_bsde} for a comparison of our results with these works.
Viscosity solutions handle second order HJB equations, which may be fully nonlinear and degenerate,  and where good regularity results for solutions are hard to get. Hence there are very few results about construction of optimal feedback controls. Optimal feedback controls for deterministic optimal control problems coming from controlled differential delay equations were constructed in \cite{goldys_2, tacconi}.
Recently optimal feedback controls were constructed in \cite{mayorga_swiech_2022} for a special class of stochastic optimal control problems with bounded evolution in a Hilbert space coming from a mean field control problem, for which the HJB equation was semilinear, had bounded terms and  the value function was $C^{1,1}$ in the state variable. Moreover, simultaneously to the current paper, optimal feedback controls were also constructed in \cite{defeo_swiech_wessels} for optimal control problems driven by more general stochastic differential equations in Hilbert spaces with unbounded operators, adapting the technique of \cite{mayorga_swiech_2022}. In \cite{defeo_swiech_wessels} the crucial ingredients are that the value function is $C^{1,1}$ in the state variable and there exists a Lipschitz selection function maximizing the Hamiltonian. The technique used there avoids the use of It\^o's formula, employed here. Applications to problems with delays are also discussed. However, we remark that the assumptions in \cite{defeo_swiech_wessels} are stronger than the ones used here.
% and may hold only in specific situations.} 
The results of our paper and \cite{defeo_swiech_wessels} seem to be the first on optimal synthesis for optimal control problems for stochastic differential equations in Hilbert spaces with unbounded operators using viscosity solutions.

The plan of the manuscript is the following. In Section \ref{sec:formul} we introduce the optimal control problems for SDDE and the main assumptions. In Section \ref{sec:prelim} we recall the results from \cite{defeo_federico_swiech} which are the basis for the current paper. Section
\ref{sec:approximation via inf-convolutions} deals with the first approximation of the value function $V$ by inf-convolutions. It is proved there that the inf-convolution of the value function is a viscosity supersolution of a perturbed HJB equation. In Section \ref{sec:dynkin_formula} we introduce a modified class of functions $\mathcal D$ from \cite{lions-infdim1} and show that they satisfy Dynkin's formula.
In Section \ref{sec:lions_approx} we further perturb the inf-convolutions of $V$ by partial convolutions to obtain functions from the class $\mathcal D$ and which are viscosity supersolutions of another perturbed HJB equations. Section \ref{sec:verthm-optfeed}
contains the proof of the verification theorem and construction of an optimal feedback control. Finally, in Section \ref{sec:example}, we present an application to stochastic optimal advertising. In Appendix \ref{sec:comparison} we prove a comparison theorem for SDDE.

\section{The optimal control problem: Setup and assumptions}\label{sec:formul}
We denote by  $M^{m \times n}$ the space of real valued $m \times n$-matrices and we denote by $|\cdot|$ the Euclidean norm in $\mathbb{R}^{n}$ as well as the norm of elements of $M^{m \times n}$ regarded as linear operators from $\mathbb{R}^{m}$ to  $\mathbb{R}^{n}$. We will write $x \cdot y$ for the inner product in $\mathbb R^n$. We consider 
the standard Lebesgue space $L^2:=L^2([-d,0];\mathbb{R}^n)$ of square integrable functions from  $[-d,0]$ to $\erre^{n}$. We denote by  
$\langle \cdot,\cdot\rangle_{L^{2}}$ the inner product in $L^2$ and by $|\cdot|_{L^{2}}$ the norm. We also consider the standard Sobolev space $W^{1,2}:=W^{1,2}([-d,0]; \mathbb{R}^n)$ of functions $f \in L^{2}$ having weak derivative $f' \in L^{2}$, endowed with the inner product $\langle f,g\rangle_{W^{1,2}}:= \langle f,g\rangle_{L^{2}}+\langle f',g'\rangle_{L^{2}}$ and norm $|f|_{W^{1,2}}:=(|f|^2_{L^{2}}+|f'|^2_{L^{2}})^{\frac{1}{2}}$.
% which render it a Hilbert space. 
%It is well known that the space $W^{1,2}$ can be identified with the space of absolutely continuous functions from $[-d,0]$ to $\R^{n}$.

We use the setup of \cite{defeo_federico_swiech}. We say that $\tau=(\Omega,\mathcal{F},(\mathcal{F}_t)_{t\geq 0},\mathbb{P}, W)$ is a generalized reference probability space if $(\Omega,\mathcal{F},\mathbb{P})$ is a complete probability space, $(\mathcal{F}_t)_{t\geq 0}$ is a filtration satisfying the usual conditions, i.e. it is right-continuous and complete, and $W=(W(t))_{t\ge 0}$ is a standard $\mathbb{R}^q$-valued $(\mathcal{F}_t)_{t\geq 0}$-Wiener process  (see \cite[Definition 1.100]{fgs_book}).
A generalized reference probability space $\tau=(\Omega,\mathcal{F},(\mathcal{F}_t)_{t\geq 0},
\mathbb{P}, W)$ is called a reference probability space if in addition $W(0)=0$ and $(\mathcal{F}_t)_{t\geq 0}$ is the augmented  filtration generated by $W$ (see \cite[Definition 2.7]{fgs_book}).
We consider the following controlled stochastic differential delay equation (SDDE) 
\begin{small}
\begin{equation}
\label{eq:SDDE}
\begin{cases}
dy(t) = \ds b_0 \left ( y(t),\int_{-d}^0 a_1(\xi)y(t+\xi)\,d\xi ,u(t) \right)
         dt 
\ds  + \sigma_0 \left (y(t),\int_{-d}^0 a_2(\xi)y(t+\xi)\,d\xi \right)\, dW(t),\quad t>0,\\
y(0)=x_0, \quad y(\xi)=x_1(\xi)\; \quad \forall \xi\in[-d,0),
\end{cases}
\end{equation}
\end{small}
where $d>0$ is the maximum delay and:
\begin{enumerate}[(i)]
\item $x_0 \in \mathbb{R}^n$, $x_1 \in L^2([-d,0]; \mathbb{R}^n)$ are the initial conditions;
\item $b_0 \colon \mathbb{R}^n \times \mathbb{R}^h \times U \to \mathbb{R}^n$, $\sigma_0 \colon \mathbb{R}^n \times \mathbb{R}^h  \to M^{n\times q}$;
\item $a_{i}:[-d,0]\to M^{h \times n}$ for $i=1,2$
and if $a_{i}^{j}$ is the $j$-th row of $a_i(\cdot)$, for $j=1,...,h$, then  $a_i^{j}\in W^{1,2}$ and $a_i^{j}(-d)=0$;
\item $u(\cdot): \Omega\times [0,+\infty )\to U$ is a suitable control process.
\end{enumerate}
The precise assumptions on $b_0,\sigma_0$ will be given later.

%\begin{remark}
%The condition $a_i(-d)=0$, $i=1,2$, is technical (see \cite[Remark 4.2]{defeo_federico_swiech}), yet not too restrictive in applications: indeed,  %one can always take a slightly larger $d$, i.e. $\tilde d= d+ \varepsilon$, and extend $a_{i}$ to an absolute continuous  $M^{h \times n}$-valued %function over $[-\tilde{d},0]$ in such a way that $a_i(-\tilde d)=0$.\end{remark}

We consider the following infinite horizon optimal control problem. Given $x=(x_0,x_1)\in \mathbb{R}^{n}\times L^2$, we define a cost functional of the form
\begin{equation}
\label{eq:obj-origbis}
  J(x;u(\cdot)) =
\mathbb E\left[\int_0^{\infty} e^{-\rho t}
l(y^{x,u}(t),u(t)) dt \right],
\end{equation}
where $\rho>0$ is the discount factor, 
$l \colon \mathbb{R}^n \times U \to \mathbb{R} $ is the running cost and $U \subset \mathbb{R}^{p}$. It is convenient to consider the stochastic optimal control problem in the weak formulation  (see \cite{fgs_book,yong_zhou}). First, we define the class of admissible control processes in the weak formulation  over reference probability spaces. For every  reference probability space $\tau$, the set of control processes $\mathcal{U}_\tau$ is defined by
\begin{equation}
\label{eq:controlsrps}
\mathcal{U}_\tau=\{u(\cdot): \Omega\times [0,+\infty )\to U: \	 u(\cdot) \ \mbox{is} \ \mathcal{F}_t\mbox{-progressively measurable} \}
\end{equation}
and we define the class of admissible controls  in the weak formulation over reference probability spaces to be
\begin{equation}
\label{eq:controlsrps1}
\mathcal{U}=\bigcup \left \{\mathcal{U}_\tau: \tau \textit{ is a reference probability space} \right \}.
\end{equation}
In von Neumann–Bernays–Gödel set theory (NBG)\footnote{Recall that NBG is a conservative extension of Zermelo-Fraenkel set theory (with the Axiom of Choice) \cite{mendelson}}, the collection of all reference probability spaces $\tau$ is not a set, but it is a proper class\footnote{To see this, e.g., fix a reference probability space $\tau=(\Omega,\mathcal{F},(\mathcal{F}_t)_{t\geq 0},\mathbb{P}, W)$; for every set $S$, define the reference probability space 
$\tau_S=(\Omega \times S,\mathcal{F} \otimes \mathcal S,(\mathcal{F}_t \otimes \mathcal S)_{t\geq 0},\mathbb{P} \otimes \delta_S, \tilde W)$, where $\mathcal S=\{\emptyset, S\}$, $\delta_S$ is the  probability measure on $\mathcal S$ such that $ \delta_S(\emptyset)=0, \delta_S(S)=1$, and $\tilde W(t)(\tilde \omega)=W(t)(\omega)$ for all $\tilde \omega=(\omega,s) \in \Omega \times S$. Thus, $\{\tau_S: S \textit{ is a set}\}$ is in one-to-one correspondence with the proper class of all sets, so it is a proper class; since  $\{\tau_S: S \textit{ is a set}\}$ is a subclass of the class of all reference probability spaces, the latter is a proper class.}
and hence it follows that  $\left \{\mathcal{U}_\tau: \tau \right.$ is a reference probability space$\left. \right \}$ is a proper class; in turn, $\mathcal U$ is also a proper class.
%This is a standard setup of a stochastic optimal control problem (see \cite{yong_zhou, fgs_book}) used to apply the dynamic programming approach. 
We remark (see e.g. \cite{fgs_book}, Section 2.3.2) that, under the assumptions below,
\begin{equation}\label{eq:equivrefprsp}
\inf_{u(\cdot)\in \mathcal{ U_\tau}} J(x;u(\cdot))=\inf_{u(\cdot)\in \mathcal{U}} J(x;u(\cdot)),
\end{equation}
for every reference probability space $\tau$  so the optimal control problem is in fact independent of the choice of a reference probability space.

We will also consider the optimal control problem using generalized reference probability spaces.
For every generalized reference probability space $\tau$, the set of control processes is defined as in \eqref{eq:controlsrps} and is denoted by $ \mathcal{\overline  U}_\tau$. We define the class of admissible control processes in the weak formulation (over generalized reference probability spaces) by 
\[
\mathcal{ \overline  U}=\bigcup \left \{\mathcal{\overline U}_\tau: \tau \textit{ is a generealized reference probability space} \right \}
\] 
(similarly, $\mathcal{ \overline  U}$ is a proper class). The goal is to minimize  $J(x;u(\cdot))$ over all $u(\cdot)\in \mathcal{\overline U}$.
%The goal is to minimize  $J(x;u(\cdot))$ over all $u(\cdot)\in \mathcal{\overline U}$. We remark (see e.g. \cite{fgs_book}, Section 2.3.2) that, under the assumptions below,
%\begin{equation}\label{eq:equivrefprsp}
%\inf_{u(\cdot)\in \mathcal{ U_\tau}} J(x;u(\cdot))=\inf_{u(\cdot)\in \mathcal{U}} J(x;u(\cdot)),
%\end{equation}
%for every reference probability space $\tau$  so the optimal control problem is in fact independent of the choice of a reference probability space. 
%and that we will prove in Section \ref{sec:verthm-optfeed} that 
%\begin{equation*}
%\inf_{u(\cdot)\in \mathcal{  U}} J(x;u(\cdot))=\inf_{u(\cdot)\in \mathcal{ \overline U}} J(x;u(\cdot)),
%\end{equation*}
% so the optimal control problem is in fact independent of the choice of a (generalized) reference probability space.}

\medskip
We will assume the following conditions. 
\begin{hypothesis}\label{hp:state} The functions 
$b_0,\sigma_0$ are continuous and such that there exists $C>0$ such that, for every $x,x_1,x_2 \in \mathbb R^n,z,z_1,z_2 \in \mathbb{R}^{m}$ and every $u \in U$,  
\begin{align*}
& |b_0(x,z,u)| \leq C(1+|x|+|z|), \\
& |\sigma_0(x,z)| \leq C(1+|x|+|z|),\\
& |b_0(x_2,z_2,u)-b_0(x_1,z_1,u)| \leq C (|x_2-x_1|+|z_2-z_1|), \\
& |\sigma_0(x_2,z_2)-\sigma_0(x_1,z_1)| \leq C(|x_2-x_1|+|z_2-z_1|). 
\end{align*}
\end{hypothesis}
Under Assumption  \ref{hp:state}, by \cite[Theorem IX.2.1]{revuz}, for each initial datum $x:=(x_0,x_1)\in \mathbb{R}^{n}\times L^2([-d,0];\mathbb R^n)$ and each control $u(\cdot)\in\ \mathcal{\overline U}$,  there exists a unique (up to indistinguishability) strong solution to \eqref{eq:SDDE} and this solution admits a version with continuous paths that we denote by  $y^{x;u}$. The proof that the assumptions of \cite[Theorem IX.2.1]{revuz} are satisfied can be found in \cite[Proposition 2.5]{tankov_federico}.

\begin{hypothesis}\label{hp:cost} $l \colon \mathbb{R}^n \times U \to \mathbb{R} $ is continuous and is such that the following hold.

\begin{itemize}
\item[(i)]
There exist constants $K>0$, such that
\begin{equation}\label{eq:growth_l}
|l(z,u)| \leq K(1+|z|) \ \ \ \forall y\in  \mathbb{R}^n, \ \forall u \in U.
\end{equation}
\item[(ii)] There exists $C>0$ such that
\begin{equation}\label{eq:regularity_l}
|l(z,u)-l(\tilde z,u)| \leq C |z-\tilde z| \quad \forall z,\tilde z \in \mathbb{R}^n, u \in U.
\end{equation}
\end{itemize}
\end{hypothesis}
\begin{flushleft}
In order for the cost functional to be well defined and continuous we will later assume (see Assumption \ref{hp:discount}) that the discount factor $\rho>0$ is sufficiently large. 
\end{flushleft}

Throughout the paper we will write $C>0,\omega, \omega_R$ to indicate, respectively, a constant, a modulus continuity, and a local modulus of continuity, which may change from place to place if the precise dependence on other data is not important.
\section{Preliminary results}\label{sec:prelim}
\subsection{The equivalent infinite dimensional Markovian
  representation}\label{subsec:infinite_dimensional_framework}
The optimal control problem we study is not Markovian due to the delay. As in \cite{defeo_federico_swiech} in order to regain Markovianity and approach the problem by dynamic programming, following a well-known procedure, see \cite[Part II, Chapter 4]{delfour} for deterministic delay equations and \cite{choj78}, \cite{DZ14}, \cite{goldys_1} for the stochastic case, we reformulate the state equation by lifting it to an infinite-dimensional space.

We define 
$
X := \mathbb{R}^n \times L^2
$.
An element $x\in X$ is 
a couple $x= (x_0,x_1)$, where $x_0 \in \mathbb{R}^n$, $x_{1}\in L^{2}$; sometimes, we will write $x=\begin{bmatrix}
x_0\\ x_1
\end{bmatrix}.$ 
The space $X$ is a Hilbert space when endowed
with the inner product 
\begin{align*}
\langle x,z\rangle_{X}& := x_0\cdot  z_0 + \langle x_1,z_1 \rangle_{L^2}= x_0  z_0 + \int_{-d}^0  x_1(\xi)\cdot z_1(\xi) \,d\xi, \ \ \ x=(x_{0},x_{1}), \ z=(z_{0},z_{1})\in X.
\end{align*}
The  induced norm, denoted by $|\cdot|_X$, is then
$$
|x|_{X} = \left( |x_0|^2 + \int_{-d}^0 |x_1(\xi)|_{L^{2}}^2\,d\xi\right)^{1/2}, \ \ \ x=(x_0,x_1) \in X.
$$
For $R>0$, we  denote  $$B_R:=\{x \in X: |x|_{X} < R\}, \ \ \ B_R^0:=\{x_0 \in \mathbb R^n: |x_0| < R\}, \ \ \ B_R^1:=\{x_1 \in L^2[-d,0]: |x_1|_{L^{2}} < R\},$$ 
to be the open balls of radius $R$ in $X$, $\mathbb R^n,$ and $L^2$, respectively. 
We denote by 
$\mathcal{L}(X)$ the space of bounded linear operators from $X$ to $X$, endowed with the operator norm 
$$|T|_{\mathcal{L}(X)}=\sup_{|x|_{X}=1} |T x|_{X}.$$
An operator $T \in \mathcal{L}(X)$ can be seen as 
$$
Tx=\begin{bmatrix}
T_{00} & T_{01}\\
T_{10} & T_{11}
\end{bmatrix}\begin{bmatrix}
x_0\\
x_1
\end{bmatrix}, \quad x=(x_0,x_1) \in X,
$$
where $T_{00} \colon \mathbb R^n \to \mathbb R^n$, $T_{01} \colon   L^2 \to \mathbb R^n$, $T_{10} \colon \mathbb R^n \to L^2$, $T_{00} \colon L^2 \to L^2$ are bounded linear operators.
Moreover, given two separable Hilbert spaces $(H, \langle \cdot , \cdot \rangle_H), (K, \langle \cdot , \cdot \rangle_K)$, we denote by $\mathcal{L}_1(H,K)$ the space of trace-class operators endowed with the norm 
$$|T|_{\mathcal{L}_1(H,K)}=\inf \left \{\sum_{i \in \mathbb N} |a_i|_H |b_i|_{K}: Tx=\sum_{i \in \mathbb N} b_i  \langle a_i,x \rangle_H , a_i \in H, b_i \in  K , \forall i \in \mathbb N \right \}.$$
We also denote by  $\mathcal{L}_2(H,K)$ the space of Hilbert-Schmidt operators from $H$ to $K$ endowed with the norm
$$|T|_{\mathcal{L}_2(H,K)}=(\operatorname{Tr}(T^*T))^{1/2}.$$ When $H=K$ we simply write $\mathcal{L}_1(H)$, $\mathcal{L}_2(H)$.
We denote by $S(H)$ the space of self-adjoint operators in $\mathcal{L}(H)$. If $Y,Z\in S(H)$, we write $Y\geq Z$ if $\langle Yx,x \rangle\leq \langle Zx,x \rangle$ for every $x \in H$.

We now recall from \cite[Section 3]{defeo_federico_swiech}) how to rewrite the state equation \eqref{eq:SDDE} in the space $X$. In order to be consistent with \cite{defeo_federico_swiech} we use the same notation as in \cite[Section 3]{defeo_federico_swiech}.  We define the linear unbounded operator $\tilde A \colon D(\tilde A) \subset X \to X$  by
\begin{equation}\label{eq:tildeA}
\tilde A x= \begin{bmatrix}
-x_0\\
x_1'
\end{bmatrix}, \quad D(\tilde A)= \left\{ x=(x_0,x_1) \in X: x_1 \in W^{1,2}, \ x_1(0)=x_0\right\}.
\end{equation}
\begin{prop}(\cite[Proposition 3.2]{defeo_federico_swiech})\label{prop:Amaxdiss}
The operator $\tilde{A}$ defined in \eqref{eq:tildeA} is maximal dissipative.
\end{prop}
Hence the operator $\tilde A$ is the generator of the so called delay semigroup $e^{t \tilde A}$ which is a strongly continuous semigroup of contractions on $X$.\\
%
%\begin{proof}
%(SF: IS IT NOT AS IN FEDERICO-GOLDYS-GOZZI I?) Let
%$$
%x \in \mathcal{D}=\left\{ x=(x_0,x_1)  \in X:
%       x_1 \in W^{1,2}([-d,0];\mathbb{R}^n), \ x_1(-d)=0\right\}
%$$
%Then for every  $y=(y_0,y_1) \in \mathcal{D}(A)$ since $y_1(0)=y_0$ we have:
%$$
%\langle A y, x\rangle=\int_{-d}^0 y_1^{\prime}(\xi) \cdot x_1(\xi) d \xi=y_0 x_1(0)-\int_{-d}^0 y_1(\xi) \cdot x_1^{\prime}(\xi) d \xi
%$$
%Thus $y \mapsto \langle A y, x\rangle$ is continuous on $\mathcal{D}(A)$ so that it follows that $x \in D\left(A^*\right)$ and
%$$
%A^* x=\begin{bmatrix}
%x_1(0)\\
%-x_1^{\prime}
%\end{bmatrix}
%$$
%To complete the proof we need to show that $\mathcal{D}\left(A^*\right)=\mathcal{D}$. To this purpose compute first the adjoint of the semigroup (see the complete calculation in \cite[Appendix A.1]{tacconi}):
%$$
%\left (e^{At} \right)^*x=\begin{bmatrix}
%x_0+\int_{(-t) \vee (-d)}^0 x_1(\xi)  d \xi\\
% x_1(\cdot-t) I_{[-d, 0]}(\cdot-t)
%\end{bmatrix}, \quad x=(x_0,x_1)\in X
%$$
%Moreover $\mathcal{D}$ is dense in $X$ and it is easy to check that $\left (e^{At} \right)^* \mathcal{D} \subset \mathcal{D}$ for any $t \geq 0$. Hence by \cite[p.8, Theorem 1.9]{davies} $\mathcal{D}$ is a core for $D \left(A^*\right)$, i.e. it is dense in $D\left(A^*\right)$ endowed with the graph norm. Finally it is easy to show that $\mathcal{D}$ is closed in the graph norm of $A^*$. Then we must have $\mathcal{D}\left(A^*\right)=\mathcal{D}$.
%\end{proof}
We  define $\tilde b \colon X \times U \to X$ (with a small abuse of notation for $\tilde b_0(x,u)$) by
\begin{align*}
\tilde b(x,u) & = \begin{bmatrix}
\tilde b_0 \left ( x,u \right)\\
0
\end{bmatrix}
=
\begin{bmatrix}
b_0 \left ( x_0,\int_{-d}^0 a_1(\xi)x_1(\xi)\,d\xi ,u \right)+x_0\\
0
\end{bmatrix}, \quad x=(x_0,x_1) \in X,  \ u \in U
\end{align*}
and  $\sigma \colon X \to \mathcal{L}(\mathbb{R}^q,X)$ (with a small abuse of notation for $\sigma_0(x)$) by 
$$\sigma(x)w=\begin{bmatrix}
\sigma_0 ( x) w\\
0
\end{bmatrix}=
\begin{bmatrix}
\sigma_0 \left ( x_0,\int_{-d}^0 a_2(\xi)x_1(\xi)\,d\xi  \right)w\\
0
\end{bmatrix},\quad x=(x_0,x_1) \in X, \ w \in \mathbb{R}^q.
$$

We point out that we write $\tilde A$, $\tilde b$ to emphasize that these are translated versions of $A$, $b$. We also want to be consistent with the notation used in \cite[Section 3]{defeo_federico_swiech}). 
%Indeed in \cite[Section 3]{defeo_federico_swiech}) $\tilde A$, $\tilde b$ are obtained through some bounded perturbations (of some  $A$, $b$) %while $\sigma$ is not. The reader will find the details in \cite[Section 3]{defeo_federico_swiech}).}

Given $x\in X$ and a control process $u(\cdot)\in\mathcal{\overline U}$, we consider  the following infinite-dimensional stochastic differential equation
\begin{equation}
\label{eq:abstract_dissipative_operator}
dY(t) = [\tilde A Y(t)+\tilde b(Y(t),u(t))]dt + \sigma(Y(t))\,dW(t) \quad \forall t \geq 0 , \quad 
Y(0) = x.
\end{equation}
As in \cite[Section 3]{defeo_federico_swiech} there exists a unique mild solution to \eqref{eq:abstract_dissipative_operator}, that is an $X$-valued progressively measurable stochastic process $Y$ satisfying
\begin{align*}
Y(t)=e^{\tilde A t}x+ \int_0^t e^{\tilde A(t-s)}\tilde b(Y(s),u(s)) ds+ \int_0^t e^{\tilde A(t-s)}\sigma(Y(s)) dW(s), \ \ \ \forall t\geq 0.
\end{align*}
The infinite dimensional stochastic differential equation \eqref{eq:abstract_dissipative_operator} is linked to \eqref{eq:SDDE} by the following result, see   \cite[Theorem 3.4]{tankov_federico} (cf. also the original result in the linear case \cite{choj78}).
\begin{prop}\label{prop:equiv} Let  Assumption \ref{hp:state} hold. Given $x\in X$ and $u(\cdot)\in\blu{\mathcal{\overline  U}}$,
let $y^{x,u}$ be the unique strong (in the probabilistic sense) solution to \eqref{eq:SDDE} and let $Y^{x,u}$ be the unique mild solution to \eqref{eq:abstract_dissipative_operator}. Then  
$$Y^{x,u}(t)=(y^{x,u}(t),y^{x,u}(t+\cdot)|_{[-d,0]}), \ \ \ \forall t\geq 0.$$\end{prop}
  
Proposition \eqref{prop:equiv} provides a Markovian
reformulation of the optimal control problem in the Hilbert space $X$.
 Indeed, the cost functional \eqref{eq:obj-origbis} can be rewritten in $X$ as
\begin{equation}
\label{ex1jbisbis}
J(x;u(\cdot)) =
\mathbb E \left[ \int_0^{\infty} e^{-\rho t}[L(Y^{x,u}(t),u(t))]dt \right],
\end{equation}
where $L : X \times U\to\erre$ is defined by
$$
L(x,u) := l(x_0,u), \quad x=(x_0,x_1) \in X, \ u \in U.
$$
The value function $V \colon X \rightarrow \mathbb R$ for the optimal control problem in the reference probability space formulation is defined by 
$$
V(x) := \inf_{u(\cdot)\in\mathcal{U}} J(x;u(\cdot)).
$$
We also define the value function $\overline V \colon X \rightarrow \mathbb R$ for the optimal control problem in the generalized reference probability space formulation  
$$
\overline V(x) := \inf_{u(\cdot)\in\mathcal{\tilde U}} J(x;u(\cdot)).
$$
We will later see in Proposition \ref{prop:V_bar_equal_V} that, under proper conditions, $\overline V=V.$ 
\subsection{Operator \texorpdfstring{$B$}{B} and space \texorpdfstring{$X_{-1}$}.}\label{sec:operator_B}
In this subsection, following \cite[Section 3]{defeo_federico_swiech}, we introduce the operator $B$ and the so-called weak $B$-condition for $\tilde{A}$. 

First note that, as in \cite[Section 3]{defeo_federico_swiech}, the adjoint operator $\tilde{A}^{*}:D(\tilde A^*)\subset X\to X$ is given by
\begin{align*}
 \tilde A^* x = \begin{bmatrix}
 x_1(0)-x_0\\
 -x_1'
\end{bmatrix}, \quad
 D(\tilde A^*) = \left\{ x=(x_0,x_1)  \in X:
       x_1 \in W^{1,2}([-d,0];\mathbb{R}^n), \ x_1(-d)=0\right\}.
\end{align*}
\begin{definition}\label{def:weakB}(See \cite[Definition 3.9]{fgs_book})
Let $B \in \mathcal{L}(X)$. We say that $\tilde{A}$ satisfies the weak $B$-condition if the following hold:
\begin{enumerate}[(i)]
\item $B$ is strictly positive, i.e. $\langle Bx,x \rangle_{X} >0$ for every $x \neq 0$;
\item $B$ is self-adjoint;
\item $\tilde{A}^* B \in \mathcal{L}(X)$;
\item There exists $C_0 \geq 0$ such that $\langle \tilde{A}^* Bx,x \rangle_{X} \leq C_0 \langle Bx,x \rangle_{X}, $ $\forall x \in X.$
\end{enumerate}
\end{definition}

Let $\tilde{A}^{-1}$ be the inverse of the operator $\tilde{A}$. As in \cite[Section 3]{defeo_federico_swiech} its explicit expression is given by
\begin{equation}\label{eq:tilde_A_inverse}
 \tilde{A}^{-1}x=\left(-x_0, -x_0 -\int_\cdot^0 x_1(\xi)d\xi\right), \quad \forall x=(x_0,x_1)\in X.
\end{equation}
Notice that $\tilde A^{-1} \in \mathcal{L}(X)$. Moreover, since $\tilde A^{-1}$ is continuous as an operator from $X$ to $W^{1,2}$, and the embedding  $W^{1,2} \hookrightarrow L^{2}$ is compact, $\tilde A^{-1}: L^2\to L^2$ is compact.
Define now
\begin{equation}\label{def:B}
B:=(\tilde{A}^{-1})^*\tilde{A}^{-1}=(\tilde{A}^*)^{-1}\tilde{A}^{-1}\in\mathcal{L}(X).
\end{equation}
$B$ is compact by the compactness of $\tilde A^{-1}$. 
\begin{prop}( \cite[Proposition 3.4]{defeo_federico_swiech})\label{prop:weak_b}
Let $B$ be defined by \eqref{def:B}. Then $\tilde{A}$ satisfies the weak $B$-condition with $C_0=0$.
\end{prop}
Observe that if we write
\begin{align}\label{eq:representation_B}
B x= 
  \begin{bmatrix}
    B_{00} & B_{01}  \\
   B_{10} & B_{11}
  \end{bmatrix} 
   \begin{bmatrix}
   x_0 \\
  x_1
  \end{bmatrix}, \ \ \ \ x=(x_0,x_1)\in X, 
\end{align}
by the strict positivity of $B$, $B_{00}\in M^{n\times n}$ is  strictly positive   and $B_{11}$ is strictly positive as an operator from $L^2$ to $L^2$. Moreover, since $B$ is strictly positive and self-adjoint, the operator $B^{1/2}\in \mathcal{L}(X)$ is well defined, self-adjoint and strictly positive. We introduce the $|\cdot |_{-1}$-norm on $X$ by
\begin{align}\label{eq:properties_norm_-1}
|x|_{-1}^2&=\langle B^{1/2}x, B^{1/2}x\rangle_{X}= \langle Bx,x\rangle_{X}=\langle (\tilde{A}^{-1})^*\tilde{A}^{-1}x,x\rangle_{X}=\langle \tilde{A}^{-1}x,\tilde{A}^{-1}x\rangle_{X}= |\tilde A^{-1}x|^{2}_{X} \quad \forall x \in X. 
\end{align}
We define 
%Moreover, since $B$ is strictly positive, compact, and self-adjoint, we can define the strictly positive, self-adjoint, and  compact operator $B^{\alpha}$ for every $\alpha \in \mathbb{R}$ (SF: COME?). 
%
%For $\alpha\in \R$, let us define the inner product in $X$:
%\begin{equation}\label{Bscalar}
%\langle x,y \rangle_{-\alpha}=\langle B^{\alpha / 2} x, B^{\alpha / 2} y\rangle_{X}, \ \ \ \ x,y\in X.
%\end{equation}
%We define a family of  spaces $\left\{X_{\alpha}\right\}_{\alpha \in \mathbb{R}}$  as follows:
%\begin{enumerate}[(i)]
%\item If $\alpha\geq 0$, then  $X_{\alpha}:=\mathcal{R}\left(B^{\alpha / 2}\right)$;
%\item If $\alpha<0$, then
%
$$X_{-1}:= \ \mbox{  the completion of $X$ under} \ |\cdot|_{-1},$$
which is a Hilbert space endowed with the inner product
$$
\langle x,y\rangle_{-1}:=\langle B^{1/2}x,B^{1/2}y\rangle_{X}= \langle Bx,y\rangle_{X}= \langle \tilde{A}^{-1}x,\tilde A^{-1}y\rangle_{X}.
$$
Notice that $|x|_{-1}\leq |\tilde A^{-1}|_{\mathcal L(X)}|x|_X$; in particular,  we have $(X,|\cdot|) \hookrightarrow (X_{-1},|\cdot|_{-1})$. 
Moreover,
strict positivity of $B$ ensures that the operator $B^{1 /2}$  can be extended to an isometry 
$B^{1 /2} \colon (X_{- 1},|\cdot|_{-1}) \to (X,|\cdot|_{X}).$

%\end{enumerate}
%
%
%We have the following classical result, whose proof can be found in (SF: CITAZIONE PIU' PRECISA) \cite{LY}:
%
%\begin{prop}\label{prop:spaces_X_alpha} 
%There exists a family of Hilbert spaces $\left\{X_{\alpha}\right\}_{\alpha \in \mathbb{R}}$ such that:
% $$\quad X_{0}=X, \quad X_{\alpha} \hookrightarrow X_{\beta}, \quad \forall \alpha \geq \beta,$$
% $$\quad\left(X_{-\alpha}\right)^{*} := \textit{dual of} \ X_{-\alpha}=X_{\alpha}, \quad \forall \alpha \in \mathbb{R},$$
%$$X_{\alpha}=\mathcal{R}\left(B^{\alpha / 2}\right), \quad \alpha \geq 0,$$ endowed by the scalar product $$ \langle x,y \rangle_{\alpha} =\langle B^{-\alpha / 2} x, B^{-\alpha / 2} y\rangle_{X}$$
%and such that
%$$X_{-\alpha}= \textit{the completion of X under the norm induced by the scalar product } $$ 
%$$\langle x,y \rangle_{-\alpha}=\langle B^{\alpha / 2} x, B^{\alpha / 2} y\rangle_{X}$$
%\end{prop}
%In particular we have
%$$ X_{2} \hookrightarrow X_{1} \hookrightarrow X \hookrightarrow X_{-1} \hookrightarrow X_{-2}$$
By \eqref{eq:properties_norm_-1} and  an application of \cite[Proposition B.1]{DZ14},  we have  
$\mbox{Range}(B^{1/2})=\mbox{Range}((\tilde A^{-1})^*)$. Since $\mbox{Range}((\tilde A^{-1})^*)=D(\tilde A^*)$, we have 
\begin{align}\label{eq:R(B^1/2)=D(A^*)}
\mbox{Range}\big(B^{1/2}\big)=D(\tilde A^*).
\end{align}
By \eqref{eq:R(B^1/2)=D(A^*)}, the operator $\tilde A^* B^{1/2}$ is well defined on the whole space $X$. Moreover, since $\tilde A^*$ is closed and $B^{1/2}\in\mathcal{L}(X)$, $\tilde A^* B^{1/2}$ is a closed operator. Thus, by the closed graph theorem, we have 
\begin{equation}\label{eq:A*B^1/2}
\tilde A^* B^{1/2} \in \mathcal{L}(X).
\end{equation}
By \eqref{eq:tilde_A_inverse}, we immediately notice that
\begin{equation}\label{eq:|x_0|_leq_|x|}
|x_0|\leq |x|_{-1}, \quad \forall x=(x_0,x_1)\in X.
\end{equation}
%\begin{remark}
%Note that in our case \eqref{eq:|x_0|_leq_|x|} is trivially satisfied. This seems in contradiction with \cite{tacconi} where this inequality is false and $|x_0|$ can be controlled by $|x|_{-1}$ only on the Sobolev space $W_{d,0}^{2,2}[-d,0]$, see \cite[Remark 5.4]{tacconi}. However recall that in \cite{tacconi} $|x|_{-1}$ is defined in a slightly different way than here, indeed it is defined as $|x|_{-1}=|A^{-1}x|$ where $A$ is not the same $A$ as here but (unless a translation in $x_0$) it is its adjoint. A similar fact can be found in [de Feo, Fed, Sw].
%\end{remark}
Since $B$ is a compact, self-adjoint and strictly positive operator on $X$, by the spectral theorem $B$ admits a set  of  eigenvalues $\{\lambda_i \}_{i \in \mathbb{N}} \subset (0, +\infty)$ such that $\lambda_{i}\to 0^{+}$ and a corresponding set $\{f_i \}_{i \in \mathbb{N}} \subset X$  of  eigenvectors  forming an orthonormal basis of $X$. By taking $\{e_i \}_{i \in \mathbb{N}}$ defined by $e_i:=\frac{1 } {\sqrt \lambda_i} f_i$, we then get   an orthonormal basis of $X_{-1}$. 
We set
$X^N := \mbox{span} \{f_1,...f_N\}= \mbox{span} \{e_1,...e_N\}$ for  $N\geq 1$, and let 
$P_N \colon X \to X$
 be the orthogonal projection onto $X_N$ and 
$Q_N := I - P_N$. 
Since $\{e_i \}_{i \in \mathbb{N}}$ is an orthogonal basis of $X_{-1}$, the projections $P_N,Q_N$ extend to orthogonal projections in $X_{-1}$ and we will use the same symbols to denote them.
% Moreover we claim that $P_N$ is an equivalent projection also in the $X$-topology i.e. if you consider $\bar P_N \colon X \to X$ the %orthogonal projection on $X_N$ in the $X$-topology  then $P_N x=\bar P_N x$ for every $x \in X$. Indeed let $x \in X$ and decompose it %as $x = \sum_{i \in \mathbb{ N}} x_i f_i$ for some $\{x_i\}_i \subset \mathbb{R}$, then
% \begin{align*}
%& \bar  P_N x= \sum_{i \in \mathbb N} x_i \bar P_N f_i= \sum_{i=1}^N x_i f_i\\
%& P_N x =  \sum_{i \in \mathbb N} x_i P_N f_i = \sum_{i \in \mathbb N} x_i \sqrt{\lambda_i} P_N e_i = \sum_{i=1}^N  x_i \sqrt{\lambda_i} %P_N e_i =\sum_{i=1}^N x_i f_i
% \end{align*}
 %so we have the claim.\\
%Then we will see $P_N, Q_N$ also as operators on $X$, i.e. $P_N, Q_N \colon X \to X$, which are bounded self-adjoint, nilpotent operator %being a projection operator on $X$.\\
We notice that 
\begin{equation}\label{eq:BP_BQ}
B P_N =P_N B, \quad B Q_N =Q_N B.
\end{equation}
%Indeed as $P_N$ is a projection operator in $X_{-1}$ we have for every $x,y \in X$
%\begin{align*}
%\langle B P_N x,y \rangle = \langle P_N x,y \rangle_{-1} = \langle x,P_N y \rangle_{-1} =\langle x,B P_N y \rangle
%\end{align*}
%so that $BP_N = (BP_N)^*=P_N^*B^*=P_N B$ and then also $B Q_N =Q_NB$ follows.\\
Therefore, since $|BQ_N|_{\mathcal L(X)}=|Q_NB|_{\mathcal L(X)}$ and $B$ is compact, we get
\begin{equation}\label{eq:norm_BQ_N_to_zero}
\lim_{N \to \infty }|BQ_N|_{\mathcal L(X)} =0.
\end{equation}
%Indeed we have:
%$$|BQ_N|_{\mathcal L(X)}=|(BQ_N)^*|_{\mathcal L(X)}  =|Q_N^*B^*|_{\mathcal L(X)} = |Q_N B|_{\mathcal L(X)}  $$
%Now since $Q_N x \to 0$ for every $x$ and $B$ is compact we have 
%$|Q_N B|_{\mathcal L(X)}  \to 0$  and we have the claim.

\subsection{Estimates for the state equation and the value function}\label{subsec:estimates}
In this subsection we recall from \cite{defeo_federico_swiech} estimates for solutions of the state equation, the cost functional and the value function. These results will be needed in the paper.

\begin{lemma}\label{lemma:regularity_coefficients_theory}
Let Assumptions \ref{hp:state} and \ref{hp:cost} hold.
There exists $C>0$ such that  the following hold true for every $x,y \in X, u \in U$:
\begin{small}
\begin{align}
& |\tilde b(x, u)-\tilde b(y, u)|_X \leq C|x-y|_{-1} \label{eq:b_lip},  \\
& \langle \tilde b(x, u)-\tilde b(y, u), B(x-y)\rangle_{X} \leq C|x-y|_{-1}^{2},  \label{eq:b_B_property}  \\
& |\tilde b(x, u)| \leq C(1+|x|_{X}) , \label{eq:b_sublinear}  \\
%& |\sigma(y,u)-\sigma(x,u)|_{\mathcal{L}_{2}(\mathbb R^q,X)}\leq C |x-y|,\label{eq:G_lipschitz} \\
& |\sigma(y)-\sigma(x)|_{\mathcal{L}_{2}(\mathbb R^q,X)} \leq C |x-y|_{-1}, \label{eq:G_lipschitz_norm_B} \\
& |\sigma(x)|_{\mathcal{L}_{2}(\mathbb R^q,X)} \leq C (1+|x|_{X}), \label{eq:G_bounded}  \\
%&|L(x, u)-L(y, u)| \leq C|x-y|, \label{eq:L_unif_cont}   \\
&|L(x, u)-L(y, u)| \leq C|x-y|_{-1}, \label{eq:L_unif_cont_norm_B}  \\
& |L(x, u)| \leq C\left(1+|x|_X\right).   \label{eq:L_growth} 
\end{align}
\end{small}
Moreover, 
\begin{equation}\label{eq:trace_goes_to_zero}
\lim_{N \to \infty} \sup_{u \in U} \operatorname{Tr}\left[\sigma(x) \sigma(x)^{*} B Q_{N}\right]=0, \ \ \ \forall x\in X.
\end{equation}
\end{lemma}
\begin{proof}
The statement of the lemma follows from \cite[Lemma 4.1]{defeo_federico_swiech}.
\end{proof}
Set 
$$\rho_0 =\max \left \{C+\frac{ C^2}{2}, C+\frac{ C^2|B|_{\mathcal{L}(X)}}{2} \right \} ,$$ where $C$ is the constant from \eqref{eq:b_sublinear} and \eqref{eq:G_bounded}.

\begin{prop}\label{prop:expect_Y(s)} (\cite[Proposition 3.24]{fgs_book})
Let Assumption \ref{hp:state} hold and let $\lambda>\rho_0$.
Let $Y(t)$ be the mild solution of \eqref{eq:abstract_dissipative_operator} with initial datum $x \in X$ and control $u(\cdot) \in \mathcal U$. Then there exists $C_\lambda>0$ such that
$$
\mathbb{E}\left[|Y(t)|_X\right] \leq C_\lambda\left(1+|x|_X\right)e^{\lambda t}, \quad \forall  t \geq 0.
$$
\end{prop}
We remark that only $\lambda>C+\frac{ C^2}{2}$ is needed in Proposition \ref{prop:expect_Y(s)}. The second restriction for $\rho_0$ is necessary to obtain Proposition \ref{prop:V_continuous_norm_-1}. We need the following assumption.
\begin{hypothesis}\label{hp:discount}
$\rho >\rho_{0}.$
\end{hypothesis}

\begin{prop}(\cite[Proposition 4.4]{defeo_federico_swiech}.)\label{prop:growth_V}
Let Assumptions \ref{hp:state}, \ref{hp:cost},  and \ref{hp:discount} hold. There exists $\bar C>0$ such that
$$
|J(x;u(\cdot))|\leq \bar C(1+|x|_X) \quad \forall x \in X,  \ \forall u(\cdot) \in \mathcal{U}.
$$
Hence, 
$$|V(x)|\leq \bar C(1+|x|_X), \ \ \ \forall x \in X.$$
\end{prop}

We now recall the notion of $B$-continuity (see \cite[Definition 3.4]{fgs_book}).
\begin{definition}\label{def:B-semicontinuity}
 Let $B \in\mathcal{L}(X)$ be a strictly positive self-adjoint operator. 
 A function $u: X \rightarrow \mathbb{R}$ is said to be $B$-upper semicontinuous (respectively, $B$-lower semicontinuous) if, for any sequence $\left\{ x_{n}\right\}_{n \in \mathbb{N}}\subset X$ such that $ x_{n} \rightharpoonup x \in X$ and $B x_{n} \rightarrow B x$ as $n \rightarrow \infty$, we have
$$
\limsup _{n \rightarrow \infty} u\left( x_{n}\right) \leq u(x) \ \ \ \mbox{
(respectively, }
 \ \liminf _{n \rightarrow \infty} u\left( x_{n}\right) \geq u(x)).$$ 
  A function $u: X \rightarrow \mathbb{R}$ is said to be $B$-continuous if it is both $B$-upper semicontinuous and $B$-lower semicontinuous.
\end{definition}
We remark that, since the operator $B$ defined in \eqref{def:B} is compact,  in our case $B$-upper/lower semicontinuity is equivalent to the weak sequential upper/lower semicontinuity, respectively. The next proposition is proved in  \cite[Example 6.2]{defeo_federico_swiech}.

\begin{prop}\label{prop:V_continuous_norm_-1}
Let Assumptions \ref{hp:state}, \ref{hp:cost},  and \ref{hp:discount} hold.
There exists $K>0$ such that
\begin{equation}\label{Va-1}
|V(x)-V(y)| \leq K|x-y|_{-1} , \ \ \ \forall x,y \in X.
\end{equation}
Hence $V$ is $B$-continuous and thus weakly sequentially continuous. 
\end{prop}
We observe that the function $V$ extends to a function defined on $X_{-1}$ satisfying \eqref{Va-1} for all $x,y\in X_{-1}$.

\subsection{HJB equation: Viscosity solutions}\label{subsec:viscosity}
In this subsection we recall the  characterization of $V$ as the unique $B$-continuous viscosity solution to the associated HJB equation that was obtained in \cite{defeo_federico_swiech}. 

%Now we state the dynamic programming principle, which is a fundamental tool for the dynamic programming approach to optimal control problems. The result is standard and its proof can be found in \cite[Theorem 2.31]{fgs_book}.
%\begin{prop}
%For every $t>0$, $x \in X$ we have:
%\begin{equation*}
%V(x)=\inf_{u \in \mathcal{U}} \mathbb{E} \left [ \int_0^t e^{- \rho s} L(Y^{x,u}(s),u(s))ds + e^{-\rho t} V(Y^{x,u}(t)) \right]
%\end{equation*}
%\end{prop}
Given $v \in C^1(X)$, we  denote by $D v(x)$ its Fr\'echet derivative at $x \in X$ and we write
\begin{align*}
Dv(x)=
  \begin{bmatrix}
    D_{x_0}v(x) \\
    D_{x_1}v(x)
  \end{bmatrix},
\end{align*}
where $D_{x_0}v(x), D_{x_1}v(x)$ are the partial Fr\'echet derivatives.
For $v \in C^2(X)$, we denote by $D^2 v(x)$ its second order Fr\'echet derivative at $x \in X$ which we will often write as
\[
D^2v(x)=
  \begin{bmatrix}
  D^2_{x_0^2}v(x) & D^2_{x_0x_1}v(x)  \\
    D^2_{x_1x_0}v(x)  & D^2_{x_1^2}v(x) 
  \end{bmatrix}.
\]
We define the Hamiltonian function  
 $H:X\times X\times S(X)\to \R$ by
\begin{align}\label{eq:hamiltonian_synthesis}
H(x,p,Z) & = H(x,p)   -  \frac{1}{2} \tr \left [ \sigma \left ( x\right) \sigma \left ( x \right) ^*  Z \right]
\nonumber \\
 &=\tilde H \left (x,p_0 \right) -  \frac{1}{2} \tr \left [ \sigma_0 \left ( x_0,\int_{-d}^0 a_2(\xi)x_1(\xi)\,d\xi \right)\sigma_0 \left ( x_0,\int_{-d}^0 a_2(\xi)x_1(\xi)\,d\xi\right)^T  Z_{00} \right]
 \nonumber  \\
 &=:\tilde H \left (x,p_0 ,Z_{00} \right),
\end{align}
where 
\begin{align}\label{eq:hamiltonian_synthesis_no_sigma}
H(x,p)&:=\sup_{u\in U} \Bigg \{ - \tilde b\left ( x ,u \right) \cdot p 
 - L(x,u)  \Bigg \}=\sup_{u\in U} \Bigg \{ - \tilde b_0\left ( x ,u \right) \cdot p_0  
 - l(x_0,u)  \Bigg \}  \nonumber \\
&= - x_0 \cdot p_0 + \sup_{u\in U} \Bigg \{ - b_0\left ( x_0,\int_{-d}^0 a_1(\xi)x_1(\xi)\,d\xi ,u \right) \cdot p_0
 - l(x_0,u)  \Bigg \}= :\tilde H \left (x,p_0 \right).
 \end{align}
 
The Hamiltonian $H(x,p,Z)$ satisfies the following properties.
\begin{lemma}\label{lemma:properties_H} (\cite[Theorem 3.75]{fgs_book})
Let Assumptions \ref{hp:state} and \ref{hp:cost} hold.
\begin{enumerate}[(i)]
\item $H$ is uniformly continuous on bounded subsets of $X \times X \times S(X)$. 
\item  For every $x,p \in X$ and every $Y,Z \in S(X)$ such that $Z \leq Y$, we have
\begin{equation}\label{eq:H_decreasing}
H(x,p,Y)\leq H(x,p,Z). 
\end{equation} 
\item For every $x,p \in X$ and every  $R>0$, we have 
\begin{align}\label{eq:H_lambda_BQ_N}
\lim_{N \to \infty} \sup \Big \{   |H(x,p,Z+\lambda BQ_N)-H(x,p,Z)|: \  |Z_{00}|\leq R, \ |\lambda| \leq R \Big \}=0.
\end{align}
\item There exists a constant $C>0$ such that
\begin{align}\label{eq:H_norm_-1}
 H \left (z,\frac{B(z-y)}{\varepsilon},Z \right )-H \left (y,\frac{B(z-y)}{\varepsilon},Y \right )  \geq -C\left (|z-y|_{-1}\left(1+\frac{|z-y|_{-1}}{\varepsilon}\right) \right) 
 \end{align}
for every $\varepsilon>0$, $y,z \in X$, and $Y,Z \in \mathcal{S}(X)$ satisfying
$$Y=P_N Y P_N \quad Z=P_N Z P_N$$
and 
\begin{align*}
\frac{3}{\varepsilon}\left(\begin{array}{cc}B P_{N} & 0 \\ 0 & B P_{N}\end{array}\right) \leq\left(\begin{array}{cc}Y & 0 \\ 0 & -Z\end{array}\right) \leq \frac{3}{\varepsilon}\left(\begin{array}{cc}B P_{N} & -B P_{N} \\ -B P_{N} & B P_{N}\end{array}\right).
\end{align*}
\item
 If $C>0$ is the constant in  \eqref{eq:b_sublinear} and \eqref{eq:G_bounded}, then, for every $x \in X, p, q \in X,  Y,Z \in \mathcal{S}(X)$,
\begin{align}\label{eq:Hamiltonian_local_lip}
| H(x, p+q,Y+Z)-H(x, p, Y)| \leq C\left(1+|x|_X\right)|q_0|_X+\frac{1}{2}C^2\left(1+|x|_X\right)^{2}|Z_{00}|.
\end{align}

\end{enumerate}
\end{lemma}

% while in \cite[Theorem 3.75]{fgs_book} $P_N$ is defined as the orthogonal projection in $X_{-1}$ onto $X^N = \mbox{Span} (e_1,...e_N)$ where $\{ e_i\}$ is any orthonormal basis in $X_{-1}$ made of elements of $X$. Then $P_N$ seen as bounded operator on $X$  in general is not an orthogonal projection onto $X^N$  because $e_i$ there may not be eigenvectors of $B$.

The Hamilton-Jacobi-Bellman (HJB) equation associated with the optimal control problem is the infinite dimensional PDE
\begin{equation}
\label{eq:HJB}
%\left\{\begin{array}{l}
\rho v(x) - \cp{\tilde A x}{Dv(x)} + H(x,Dv(x),D^2v(x))=0,
\quad x \in X.
%0 \leq t \leq T\\[8pt]
%v(T) = \varphi,
%end{array}\right.
\end{equation}
We recall the definition of $B$-continuous viscosity solution from \cite{fgs_book}.
\begin{definition}\label{def:test_functions}
\begin{itemize}
\item[(i)]  $\phi \colon X \to \mathbb R$ is a regular test function if
\begin{small}
\begin{align*}
\phi \in \Phi := \{ \phi \in C^2(X): \phi \textit{ is weakly sequentially lower semicontinuous and }
\\
D\phi , D^2\phi , \tilde A^*  D\phi \textit{ are uniformly continuous on }X\};
\end{align*}
\end{small}
\item[(ii)] $g \colon X \to \mathbb R$  is a radial test function if
\begin{align*}
g \in \mathcal G:= \{ g \in C^2(X): g(x)=g_0(|x|_X) \textit{ for some } g_0 \in C^2([0,\infty)) \textit{  non-decreasing}, g_{0}'(0)=0 \}.
\end{align*}
\end{itemize}
\end{definition}
%\begin{flushleft}

We remark that the set $\Phi$ is large enough to contain functions used in the proof of the comparison principle. In particular here, we notice that if $\tilde \phi \in C^2({X_{-1}})$ with $D_{-1} \tilde \phi,D_{-1} ^2 \tilde \phi$ being uniformly continuous, then its restriction $\phi$ to $X$ is in  $\Phi$. Indeed, it is easy to see that $\phi \in C^2(X)$ with
$$D\phi(x)=BD_{-1}\tilde \phi(x),\quad \forall x \in X, $$
so that $D\phi$ is uniformly continuous on $X$;  $\tilde A^* D\phi=\tilde A^* B^{1/2}B^{1/2} D_{-1}\tilde \phi$, so that, thanks to \eqref{eq:A*B^1/2}, $\tilde A^* D\phi$ is uniformly continuous on $X$; finally,  $D^2 \phi(x)=B  D_{-1}^2\tilde \phi(x)$, $x \in X$, so that $D^2\phi$ is uniformly continuous on $X$.
Note also that, if $g\in\mathcal{G}$, we have 
%\end{flushleft}
\begin{align}\label{eq:gradient_radial}
D g(x)=\left\{\begin{array}{l}
g_0^{\prime}(|x|_{X}) \frac{x}{|x|_{X}}, \quad \ \ \ \mbox{if} \  x \neq 0, \\
0,  \quad\quad  \ \ \ \ \ \ \ \ \ \ \ \ \,\,\,  \ \mbox{if} \ x=0.
\end{array}\right.
\end{align}

We say that a function is locally bounded if it is bounded on bounded subsets of $X$.
\begin{definition}\label{def:viscosity_solution}
\begin{enumerate}[(i)]
\item A locally bounded weakly sequentially upper semicontinuous function $v:X\to\R$ is a viscosity subsolution of \eqref{eq:HJB} if, whenever $v-\phi-g$ has a local maximum at $x \in X$ for $\phi \in \Phi, g \in \mathcal G$, then 
\begin{equation*}
\rho v(x) - \cp{ x}{\tilde A^* D\phi(x)}_{X} + H(x,D\phi(x)+Dg(x) ,D^2\phi(x)+D^2g(x))\leq 0.
\end{equation*}
\item 
A locally bounded weakly sequentially lower semicontinuous function $v:X\to\R$ is a viscosity supersolution of \eqref{eq:HJB} if, whenever $v+\phi+g$ has a local minimum at $x \in X$ for $\phi \in \Phi$, $g \in \mathcal G$, then 
\begin{equation*}
\rho v(x) + \cp{ x}{\tilde A^* D\phi(x)}_{X} + H(x,-D\phi(x)-Dg(x) ,-D^2\phi(x)-D^2g(x))\geq 0.
\end{equation*}
\item 
A viscosity solution of \eqref{eq:HJB} is a function $v:X\to\R$ which is both a viscosity subsolution and a viscosity supersolution of \eqref{eq:HJB}.
\end{enumerate}
\end{definition}
Define $\mathcal{S}:=\{u \colon X \to \mathbb{R}:  \exists k\geq 0 \ \mbox{satisfying \eqref{eq:k_set_uniqueness} and } \tilde C\geq 0 \,\mbox{such that}\,  |u(x)|\leq  \tilde C (1+|x|_X^k)\},$
where
\begin{equation}\label{eq:k_set_uniqueness}
\begin{cases}
k<\frac{\rho}{C+\frac{1}{2} C^{2}}, \ \ \ \ \ \ \ \ \ \ \ \ \ \ \ \quad \mbox{ if } \ \frac{\rho}{C+\frac{1}{2} C^{2}} \leq 2, \\
C k+\frac{1}{2} C^{2} k(k-1)<\rho,  \quad \mbox{ if } \frac{\rho}{C+\frac{1}{2} C^{2}}>2, 
\end{cases}
\end{equation}
and $C$ is the constant appearing in \eqref{eq:b_sublinear} and \eqref{eq:G_bounded}.

It was proved in \cite{defeo_federico_swiech} that $V$ is the unique viscosity solution of \eqref{eq:HJB} in $\mathcal{S}$.

\begin{theorem}(\cite[Theorem 5.4]{defeo_federico_swiech})\label{th:existence_uniqueness_viscosity_infinite}
Let Assumptions \ref{hp:state}, \ref{hp:cost}, and \ref{hp:discount} hold.
The value function $V$ is the unique viscosity solution of \eqref{eq:HJB} in the set $\mathcal{S}$.
\end{theorem}
\subsection{Partial regularity of \texorpdfstring{$V$}{V}}\label{subsec:regularity}
In this subsection we recall the partial regularity result for $V$ with respect to the $x_0$-variable which was obtained in \cite{defeo_federico_swiech}.

\begin{hypothesis}\label{hp:uniform_ellipticity}
For every $R>0$  there exists $\lambda_{R} >0$ such that
$$\sigma_0(x) \sigma_0(x)^T\geq \lambda_{R} I, \quad \forall x\,\ \mbox{such that}\, \ |x|_{X} \leq R.$$
\end{hypothesis}
For every $\bar x_1 \in L^2$ we define 
$$V^{\bar x_1}(x_0):=V(x_0,\bar x_1), \quad \forall x_0 \in \mathbb R^n.$$

\begin{theorem}(\cite[Theorem 6.5]{defeo_federico_swiech})\label{th:C1alpha}
Let Assumptions \ref{hp:state}, \ref{hp:cost}, \ref{hp:discount},  and \ref{hp:uniform_ellipticity} hold. For every $p>n$ and  every fixed $\bar x_1 \in L^2$, we have $V^{\bar x_1}\in W^{2,p}_{\rm loc}(\mathbb R^n)$; thus,  by Sobolev embedding, $V^{\bar x_1}\in C^{1,\alpha}_{\rm loc}({\mathbb{R}}^n)$ for all $0<\alpha<1$.
Moreover, for every $R>0$, there exists $C_{R}>0$ such that 
$$|V^{\bar x_1}|_{W^{2,p}(B_{R})}\leq C_R, \ \ \ \ \forall \bar x_{1}\ \mbox{such that}\, \ |\bar x_1|_{L^{2}}\leq R.$$ 
Finally, $D_{x_0}V$ is continuous with respect to the $|\cdot |_{-1}$ norm on bounded sets of $X$. In particular, $D_{x_0}V$ is continuous in $X$.
\end{theorem}
%\begin{remark}\label{rem:regularity_for_any_viscosity_sol}
%Notice that the same statement (with the same proof) holds for any viscosity solution $w$ of the  HJB equation \eqref{eq:HJB} which is Lipschitz %with respect $|\cdot |_{-1}$.
%\end{remark}

\section{Approximations by inf-convolutions}\label{sec:approximation via inf-convolutions}
In this section we begin the process of approximating the value function $V$ by more regular functions. The first step is to use an appropriately defined inf-convolution $V_\epsilon$ of $V$ and prove that $V_\epsilon$ is a viscosity super-solution of a perturbed HJB equation. To do this we need one more assumption about $V$.

We extend $V$ to the function $\tilde V:X_{-1}\to\mathbb{R}$ which then also satisfies \eqref{Va-1} for all $x,y\in X_{-1}$.

\begin{hypothesis}\label{hp:convexity}
$V$ is $|\cdot|_{-1}$-semiconvex \blu{on $X$,} i.e. there exists $C\geq 0$ (called a semiconvexity constant) such that $V(x)+C|x|^2_{-1}$ is convex \blu{on $X$}.
\end{hypothesis}
The $|\cdot|_{-1}$-semiconvexity is equivalent to the requirement that there exists $C\geq 0$ such that
\begin{equation}\label{eq:-1semiconvex}
\lambda V(x) +(1-\lambda) V(y)- V(\lambda x + (1-\lambda)y) \geq  -C  \lambda (1-\lambda)|x|^2_{-1}\quad \forall \lambda \in [0,1], x,y \in X.
\end{equation}
If $V$ satisfies \eqref{eq:-1semiconvex} then, due to the continuity of $\tilde V$ in $X_{-1}$, the function $\tilde V$ is semiconvex on $X_{-1}$ and it satisfies \eqref{eq:-1semiconvex} for all 
$x,y\in X_{-1}$.  We say that a function $f$ is $|\cdot|_{-1}$-semiconcave if there is $C\geq 0$ such that $f(x)-C|x|^2_{-1}$ is concave.

It is rather well known that the semiconcavity of $V$ can be obtained under standard hypotheses on the data, e.g. see \cite{yong_zhou}, \cite{defeo_swiech_wessels} for the finite and the infinite-dimensional cases respectively.  Here instead we require the $|\cdot|_{-1}$-semiconvexity of $V$. We provide two examples where $V$ is convex in the spirit of \cite{goldys_1,goldys_2}, where the authors prove the concavity of $V$ for a maximization problem (which corresponds to the convexity for a minimization problem considered here). See also \cite{defeo_swiech_wessels}  for other results on $|\cdot|_{-1}$-semiconvexity of $V$ in the finite horizon case.

We remark that in the approximation procedure of Sections \ref{sec:approximation via inf-convolutions} and \ref{sec:lions_approx}, $V$ can be replaced by any viscosity supersolution of \eqref{eq:HJB} which satisfies \eqref{Va-1}, the regularity properties of Theorem \ref{th:C1alpha} and Assumption \ref{hp:convexity}.

\begin{example}\label{ex:convex1}
We will use the notation introduced in Appendix \ref{sec:comparison}.
Let Assumptions \ref{hp:state} and \ref{hp:cost} hold. Let $U$ be convex, $a_1(\xi)\geq 0$ for every $\xi \in [-d,0]$, let each component of $b_0(x,z,u)$ be jointly concave and let \eqref{eq:assumption_b0_comparison_monotonone}, \eqref{eq:assumption_b0_comparison} hold. Assume that $\sigma_0(x,z,u)=\sigma_0 \in M^{n \times q}$ and $l(x,u)$ is jointly convex and non-increasing with respect to $x \in \mathbb{R}^n$ for each fixed $u \in U$ (here the inequality $x \leq y$ for $x,y \in \mathbb R^n$ is understood component-wise as explained in Appendix \ref{sec:comparison}).  We show that under these hypotheses $V$ is convex.

Indeed, let $x, \bar x \in X$, $\lambda \in [0,1]$, $\epsilon>0$ and $u^\epsilon(\cdot), \bar u^\epsilon(\cdot)$ be $\epsilon$-optimal controls for the initial conditions $x, \bar x$ respectively. By \eqref{eq:equivrefprsp} we can assume that the control processes are defined on the same reference probability space. Denote by $x(t), \bar x(t)$ the solutions of \eqref{eq:SDDE} with initial state $x, \bar x$ and controls $u^\epsilon(\cdot), \bar u^\epsilon(\cdot)$ respectively. Moreover, set $x_\lambda=\lambda x + (1-\lambda) \bar x$, $u^\epsilon_\lambda(\cdot)=\lambda u^\epsilon(\cdot) + (1-\lambda) \bar u^\epsilon(\cdot)$ and let $x(t;x_\lambda, u^\epsilon_\lambda(\cdot))$ be the solution of \eqref{eq:SDDE} with the initial state $x_\lambda$ and control $u^\epsilon_\lambda(\cdot)$. Finally, set $x_\lambda(t)=\lambda x(t)+ (1-\lambda) \bar x(t)$. Note that since \eqref{eq:SDDE} is not linear $x(t,x_\lambda, u^\epsilon_\lambda(\cdot)) \neq x_\lambda(t)$ in general. First, since each component of $b_0(x,y,u)$ is jointly concave,  we have
\begin{small}
\begin{align*}
dx_\lambda(t)& = \lambda dx(t)+(1-\lambda) d\bar x(t)\\
&=\left [\lambda b_0 \left ( x(t),\int_{-d}^0 a_1(\xi)x(t+\xi)\,d\xi ,u^\epsilon(t) \right) +(1-\lambda) b_0 \left ( \bar x(t),\int_{-d}^0 a_1(\xi)\bar x(t+\xi)\,d\xi ,\bar u^\epsilon(t) \right)  \right] dt+ \sigma_0 dW(t)\\
& \leq b_0 \left ( x_\lambda(t),\int_{-d}^0 a_1(\xi)x_\lambda(t+\xi)\,d\xi ,u^\epsilon_\lambda(t) \right)  dt+ \sigma_0 dW(t).
\end{align*}
\end{small}
Regarding $x(t;x_\lambda, u_\lambda^\epsilon(\cdot))$, of course we have
\begin{align*}
d x(t;x_\lambda, u_\lambda^\epsilon(\cdot))  = b_0 \left ( x(t;x_\lambda, u_\lambda^\epsilon(\cdot)),\int_{-d}^0 a_1(\xi)x(t+\xi;x_\lambda, u_\lambda^\epsilon(\cdot))d\xi ,u^\epsilon_\lambda(t) \right)  dt+ \sigma_0 dW(t).
\end{align*}
Thus, by the comparison lemma, Lemma \ref{lemma:comparison_sdde}, we have
\begin{equation}\label{eq:x_lambda leq x}
x_\lambda(t) \leq x(t;x_\lambda, u_\lambda^\epsilon(\cdot)) \quad \forall t \geq 0.
\end{equation}
Finally, by \eqref{eq:x_lambda leq x}, the fact that $l(\cdot,u)$ is non-increasing, the joint convexity of $l(x,u)$ and since $u^\epsilon(\cdot), \bar u^\epsilon(\cdot)$ are $\epsilon$-optimal controls for the initial states $x, \bar x$ respectively, we have
\begin{align*}
V(x_\lambda) \leq J(x_\lambda; u^\epsilon_\lambda)& = \mathbb E\left[ \int_0^{+\infty} e^{-\rho t}l(x(t;x_\lambda, u_\lambda^\epsilon(\cdot)),u^\epsilon_\lambda(t))dt \right] \leq  \mathbb E\left[ \int_0^{+\infty} e^{-\rho t}l(x_\lambda(t),u_\lambda^\epsilon(t))dt \right]\\
& \leq   \mathbb E\left[ \int_0^{+\infty} e^{-\rho t}[\lambda l(x(t),u^\epsilon(t))+(1-\lambda) l(\bar x(t),\bar u^\epsilon(t))]dt \right] \\
& = \lambda J(x;u^\epsilon(\cdot)) +(1-\lambda) J(\bar x; \bar u^\epsilon(\cdot)) \leq \lambda V(x)+ (1-\lambda) V(\bar x) + \epsilon
\end{align*}
so that by letting  $\epsilon \to 0$ we obtain the convexity of $V$.
\end{example}

\begin{example}\label{ex:convex2} Assume that $U$ is convex, $b_0(x,y,u)$ is linear in $x,y,u$, $\sigma_0(x,y,u)=\sigma_0$ and $l(x,u)$ is jointly convex (here we do not require that $l(x,u)$ is non-increasing in $x$ for a fixed $u$). Then $V$ is convex.

As in the previous example let $x, \bar x \in X$, $\lambda \in [0,1]$, $\epsilon>0$ and $u^\epsilon(\cdot), \bar u^\epsilon(\cdot)$ be 
$\epsilon$-optimal controls for the initial conditions $x, \bar x$ respectively (defined on the same reference probability space). We use $x(t), \bar x(t), x_\lambda, x_\lambda(t),u_\lambda^\epsilon(\cdot)$, $x(t;x_\lambda, u_\lambda^\epsilon(\cdot))$ to denote the same objects as in Example \ref{ex:convex1}. Note that since \eqref{eq:SDDE} is now linear, 
\begin{equation}\label{eq:x=x_lambda}
x(t;x_\lambda, u_\lambda^\epsilon(\cdot))=x_\lambda(t):=\lambda x(t)+ (1-\lambda) \bar x(t)
\end{equation}
By \eqref{eq:x=x_lambda} we have 
\begin{align*}
V(x_\lambda) \leq J(x_\lambda; u_\lambda^\epsilon(\cdot))&= \mathbb E\left[ \int_0^{+\infty} e^{-\rho t}l(x(t;x_\lambda, u^\epsilon_\lambda(\cdot)),u^\epsilon_\lambda(t))dt \right]  = \mathbb E \left[ \int_0^{+\infty} e^{-\rho t}l(x_\lambda(t),u_\lambda^\epsilon(t))dt \right]
\end{align*}
and hence, proceeding as in the previous example, we obtain the convexity of $V.$
\end{example}
Let $\epsilon>0$. We define by $\tilde V_{\epsilon}$ the inf-convolution of $\tilde V$,
$$\tilde V_{\epsilon}(x):=\inf_{y \in X_{-1}}\left [ \tilde V(y)+\frac{1}{2\epsilon}|x-y|^2_{-1}\right]=\inf_{y \in X}\left [V(y)+\frac{1}{2\epsilon}|x-y|^2_{-1}\right]=-\sup_{y \in X}\left [- V(y)-\frac{1}{2\epsilon}|x-y|^2_{-1}\right].$$
The function $\tilde V_{\epsilon}$ restricted to $X$ will be denoted by $V_{\epsilon}$. We have the following result.

\begin{lemma}\label{lem:c11}
Let Assumptions \ref{hp:state}, \ref{hp:cost}, \ref{hp:discount} and \ref{hp:convexity} hold.
\begin{enumerate}[(i)]
\item $\tilde V_\epsilon$ satisfies \eqref{Va-1}, it is semiconcave in $X_{-1}$ and, if $\epsilon$ is small enough, it is semiconvex in $X_{-1}$ with a semiconvexity constant $C$ independent of $\epsilon$.
\item $\tilde V_\epsilon \in C^{1,1}(X_{-1})$. It follows that $V_\epsilon \in C^{1,1}(X)$ with
\begin{align}\label{eq:gradient_w_gradient_w_-1}
DV_\epsilon(x)=B D_{-1} \tilde V_\epsilon(x) \in D(\tilde A^*)
\end{align} 
 for every $x \in X$, where $D_{-1}\tilde V_\epsilon$ denotes the Frechet derivative of $\tilde V_\epsilon$ in $X_{-1}$. Moreover, for every $x_1 \in L^2[-d,0]$, $D^2_{x_0}V_\epsilon(x)$ exists for a.e. $x_0 \in \mathbb R^n$ with $x=(x_0,x_1)$.
\item We have 
\begin{align}\label{eq:convergence_inf_convolution}
\lim_{\epsilon \to 0} V_\epsilon  = V \quad \textit{uniformly on}\,\,X,\quad \lim_{\epsilon \to 0} D_{x_0}V_\epsilon  = D_{x_0}V \quad \textit{uniformly on } B_R \subset X \quad \forall R>0.
\end{align} 
\end{enumerate}
\end{lemma}
\begin{proof}$ $
\begin{enumerate}[(i)]
\item This is a standard result, e.g. see \cite{lasry}. 
\item It follows from \cite{lasry} that $\tilde V_\epsilon \in C^{1,1}(X_{-1})$. It is then easy to see  that $ V_\epsilon \in C^{1,1}(X)$ with
\[
DV_\epsilon(x)=B D_{-1} \tilde V_\epsilon(x)
\]
 for every $x \in X$. Since $D_{-1} \tilde V_\epsilon(\bar x)  \in X_{-1}$, $B^{1/2}D_{-1} \tilde V_\epsilon(\bar x) \in X$ so that 
 $D V_\epsilon( x)=BD_{-1} \tilde V_\epsilon(  x)=B^{1/2} B^{1/2}D_{-1} \tilde V_\epsilon( x)\in R(B^{1/2})= D(\tilde A^*)$, where the last equality follows by \eqref{eq:R(B^1/2)=D(A^*)}.
 Finally note that by Alexandrov's theorem, if $\bar x_1 \in L^2([-d,0],\mathbb{R}^n)$ is fixed, $D^2_{x_0}V_\epsilon(x_0,\bar x_1)$ exists for a.e. $x_0 \in \mathbb{R}^n$.
 \item The uniform convergence of $V_\epsilon$ follows by standard theory since $V$ is Lipschitz with respect to the $| \cdot|_{-1}$ norm. 
 %We prove the convergence of $D_{x_0}V_\epsilon.$  
 %With a similar proof of \cite[remark prior to (D.8) p. 862]{fgs_book}, since $w_\epsilon \in \mathcal C^{1,1}(X)$, for every $x \in X$ we have 
%\begin{equation}\label{eq:proof_lemma_convergence_convolution_Dw_eps}
%Dw_\epsilon(x)=\frac{1}\epsilon B(x- y^\epsilon)
%\end{equation}
 % where $y^\epsilon=y^\epsilon(x) \in X$ is the unique point such that the infimum in the definition of $w_\epsilon$ is realized, i.e. 
 %$w_\epsilon(x)=w( y^\epsilon)+\frac{1}{2\epsilon}|x- y^\epsilon|^2_{-1}$.
 
%Recall that by Remark \ref{rem:regularity_for_any_viscosity_sol} $D_{x_0} w$ exists and it is continuous. Then, since $y^\epsilon$ is a minimizer
% of $y \mapsto w( y)+\frac{1}{2\epsilon}|x- y|^2_{-1}$, by differentiating this with respect to $x_0$ and equating to 0 at $y=y^\epsilon$, we have:
%$$D_{x_0}w(y^\epsilon)=\frac{1}\epsilon (B(x- y^\epsilon))_0=D_{x_0}w_\epsilon(x),$$
%where the last equality follows from \eqref{eq:proof_lemma_convergence_convolution_Dw_eps}. Note now that $y^\epsilon =y^\epsilon(x)
%  \xrightarrow{\epsilon \to 0} x$ for every $x \in X$. Hence, taking into account also the continuity of $D_{x_0} w$, we have:
%$$|D_{x_0}w_\epsilon(x)-D_{x_0}w(x) |=|D_{x_0}w(y^\epsilon)-D_{x_0}w(x)| \xrightarrow{\epsilon \to 0} 0 \quad \forall x \in X.$$

We now show that $D_{x_0}V_\epsilon(x) \xrightarrow{\epsilon \to 0} D_{x_0}V(x)$ uniformly on $B_R \subset X$ for every $R>0.$ We will assume without loss of generality that $V$ and $V_\epsilon$ are convex.
Fix $R>0$ and by contradiction assume that $D_{x_0}V_\epsilon(x)$ does not converge uniformly to $D_{x_0}V(x)$ on $B_R$. Then there exist $c>0$, $\{(x_0^\epsilon,x_1^\epsilon)\} \subset  B_R$, such that if we set $p_0^\epsilon :=D_{x_0}V_\epsilon(x_0^\epsilon,x_1^\epsilon)$, $\bar p_0^\epsilon :=D_{x_0}V(x_0^\epsilon,x_1^\epsilon)$, we have $|p_0^\epsilon-\bar p_0^\epsilon|\geq c.$ Let $q_0^\epsilon$ be such that $|q_0^\epsilon|=1$ and $|p_0^\epsilon-\bar p_0^\epsilon|=(p_0^\epsilon-\bar p_0^\epsilon)\cdot q_0^\epsilon$. Denote $a_\epsilon:=|V_\epsilon-V|_{L^\infty(B_{R+1})}$ and let $0 \leq t \leq 1$. Since $V_\epsilon$ is convex, we have
\[
V(x_0^\epsilon+t q_0^\epsilon,x_1^\epsilon) \geq V_\epsilon(x_0^\epsilon+t q_0^\epsilon,x_1^\epsilon)  - a_\epsilon 
\geq 
 V_\epsilon(x_0^\epsilon,x_1^\epsilon)  + t p_0^\epsilon \cdot q_0^\epsilon   -a_\epsilon \geq V(x_0^\epsilon,x_1^\epsilon)  
 + t p_0^\epsilon \cdot q_0^\epsilon   - 2 a_\epsilon.
\]
On the other hand, as $V^{\bar x_1} \in C^{1,\alpha}_{\rm loc}(\mathbb R^n)$ and has locally uniform $C^{1,\alpha}$-norm for  $\bar x_1 \in L^2,$  $|\bar x_1| \leq R$, there exists $C=C_R>0$ such that
$$V(x_0^\epsilon+t q_0^\epsilon,x_1^\epsilon) \leq V(x_0^\epsilon,x_1^\epsilon) + t \bar p_0^\epsilon \cdot q_0^\epsilon+ Ct^{1+\alpha}. $$
Putting together these two inequalities  we obtain
\[
ct \leq t  (p_0^\epsilon-\bar p_0^\epsilon) \cdot q_0^\epsilon \leq Ct^{1+\alpha} +2a_\epsilon.
\]
Letting $\epsilon \to 0$ (note that $a_\epsilon \to 0$ since $V_\epsilon \to V$ uniformly), we thus have
$$ ct\leq Ct^{1+\alpha}$$
which is impossible for small $t$ as $\alpha>0$.
\end{enumerate}
\end{proof}
 
We now prove the main result of this section.
\begin{prop}\label{prop:inf_conv_subsolution_perturbed}
Let Assumptions \ref{hp:state}, \ref{hp:cost}, \ref{hp:discount} and \ref{hp:convexity} hold.
There exists $\gamma(\epsilon) \geq 0$, $\gamma(\epsilon) \to 0$ as $\epsilon \to 0$ such that $V_\epsilon$ is a viscosity supersolution of 
\begin{equation}\label{pertHJB-gamma-eps}
\rho V_\epsilon(x)-\langle  \tilde A^*DV_\epsilon(x),x \rangle_X+\tilde H \left (x,D_{x_0}V_\epsilon(x) ,D^2_{x_0^2}V_\epsilon(x) \right) =-\gamma(\epsilon), \quad x \in X.
\end{equation}
In fact the viscosity supersolution property of $V_\epsilon$ holds in the following stronger sense: if $x \in X$ is a local minimum of $V_\epsilon + \phi + g$ for test functions $\phi \in \Phi, g \in \mathcal G$, then $Dg(x)\in D(\tilde A^*)$ and
\begin{align*}
 \rho V_\epsilon(x)& +  \langle x, \tilde A^*(D \phi ( x) + D g(\bar x)) \rangle_X + H \left ( x, -D \phi (x) - D g(x), -D^2 \phi (x) - D^2 g(x) \right) \geq  -\gamma(\epsilon).
 \end{align*}
\end{prop}
\begin{proof}
 \textbf{Step 1.} 
 
We fix $\epsilon>0$. Let $x_0 \in X$ be a local minimum of $V_\epsilon + \phi + g$. Notice that here we are not using the notation $x=(x_0,x_1)\in X$ as in the rest of the paper. We remark that $\phi, g$ may depend on $\epsilon$ and $x_0$.  Assume without loss of generality that the minimum is global and strict and $V_\epsilon+\phi+g \to \infty$ as $|x|\to \infty$ (see \cite{fgs_book}, Lemma 3.37). We can also assume that 
\begin{align}\label{eq:w_eps_+phi+g(x_0)=0}
V_\epsilon(x_0)+\phi(x_0)+g(x_0)=0
\end{align}
 and $\phi$ is bounded. 
For every $\delta>0$, we can then find $R>0$ such that for every $x,y \in X$ with $|x|_X,|y|_X \geq R-1$,
\begin{align}\label{eq:proof_inf_convolution_approx_bound}
-V(y)-\frac{1}{2 \epsilon}|x-y|_{-1}^2-\delta |y|_X^2-\phi(x)-g(x)  \leq -V(x_0)-\delta |x_0|^2-\phi(x_0)-g(x_0)-1.
 \end{align}
\textbf{Step 2.} We introduce the function 
$$\bar \Phi \colon X \times X \to \mathbb R, \quad \bar\Phi(x,y) := -V(y)-\frac{1}{2 \epsilon}|x-y|_{-1}^2-\delta |y|_X^2-\phi(x)-g(x).
$$
We prove that $\bar \Phi$ achieves a maximum $\bar x, \bar y$ over $X \times X$ and we show the following limit properties of $\bar x, \bar y$:
\begin{align}
&\lim_{\delta \to 0}\bar x=x_0, \label{eq:bar_x_=_x_0}\\
& \lim_{\delta \to 0}\delta |\bar y|_X^2 =0,  \label{eq:delta_bary_squared_to_zero}\\
&\limsup_{\delta \to 0} |\bar x-\bar y|_{-1}\leq  2K \epsilon. \label{eq:proof_barx-bary_1}
\end{align}

We observe that $\bar \Phi$ is weakly sequentially upper semicontinuous on $X\times X$. Indeed, since $V$ is Lipschitz continuous with respect to the $|\cdot|_{-1}$ norm and $B$ is compact, it is weakly sequentially continuous and the same holds for $|\cdot|_{-1}^2$. Then the weak sequential upper semicontinuity of $-\bar \Phi$ follows as $\phi \in \Phi$ is a regular test function and $g \in \mathcal G$ is a radial test function which is weakly sequentially lower semicontinuous.

Therefore, we can find $\bar x, \bar y \in B_R$ such that
\begin{align*}
\bar \Phi(\bar x, \bar y)=\sup_{|x|_X,|y|_X \leq R} \bar \Phi(x,y)=\sup_{x,y \in X} \bar \Phi(x,y),
\end{align*} 
where the last equality follows by \eqref{eq:proof_inf_convolution_approx_bound}.
We now set
$$
\begin{gathered}
m_{\epsilon, \delta}=\sup_{x,y \in X}  \left\{-V(y)-\frac{1}{2 \epsilon}|x-y|_{-1}^2-\delta |y|_X^2-\phi(x)-g(x)\right\} =\bar \Phi(\bar x,\bar y) 
\end{gathered}
$$
and note that 
\begin{align}\label{eq:lim_phi(bar_x,bar_y)}
\lim_{\delta \to 0} m_{\epsilon, \delta}& = \lim_{R \to \infty}  \sup_{x \in X}\sup_{y \in B_R}  \left\{-V(y)-\frac{1}{2 \epsilon}|x-y|_{-1}^2-\phi(x)-g(x) \right\} \nonumber \\
& = \sup_{x \in X}  \left\{-V_\epsilon(x)-\phi(x)-g(x) \right\} = -V_\epsilon(x_0)-\phi(x_0)-g(x_0)=0.
\end{align}
Moreover by the definition of $V_\epsilon$ we have 
\begin{align*}
m_{\epsilon, \delta} & =\bar \Phi(\bar x, \bar y)  \leq -V(\bar y)-\frac{1}{2 \epsilon}|\bar x- \bar y|_{-1}^2-\phi(\bar x)-g(\bar x)\leq -V_\epsilon(\bar x)-\phi(\bar x)-g(\bar x).
\end{align*}
 Letting $\delta \to 0$ and using \eqref{eq:lim_phi(bar_x,bar_y)} we have
$$0=-V_\epsilon(x_0)-\phi(x_0)-g(x_0) \leq \lim_{\delta \to 0} -V_\epsilon(\bar x)-\phi(\bar x)-g(\bar x),$$
so that as $x_0$ is a strict maximum for $-V_\epsilon-\phi-g$, we must have \eqref{eq:bar_x_=_x_0}.

Moreover, since $m_{\epsilon, \delta}=-V(\bar y)-\frac{1}{2 \epsilon}|\bar x-\bar y|_{-1}^2-\delta |\bar y|_X^2-\phi(\bar x)-g(\bar x)$, so that $m_{\epsilon, \delta}+ \frac{\delta }{2} |\bar y|_X^2=-V(\bar y)-\frac{1}{2 \epsilon}|\bar x-\bar y|_{-1}^2-\frac{\delta}{2} |\bar y|_X^2-\phi(\bar x)-g(\bar x) \leq m_{\epsilon, \delta/2}$, letting $\delta \to 0$ we obtain \eqref{eq:delta_bary_squared_to_zero}.

By \eqref{eq:lim_phi(bar_x,bar_y)} we have
\begin{align*}
V(\bar y)+\frac{1}{2 \epsilon}|\bar x-\bar y|_{-1}^2+\delta |\bar y|_X^2+\phi(\bar x)+g(\bar x) = - \Phi(\bar x, \bar y) \leq \omega (\delta, \epsilon),
\end{align*}
for some local modulus $\omega(\cdot,\epsilon)$ (depending on $\phi, g, x_0$).
Since $V_\epsilon+\phi+g$ has a minimum at $x_0$ then $0 \leq V_\epsilon+\phi+g$ so that $-V(\bar x) \leq -V_\epsilon(\bar x) \leq \phi(\bar x)+ g(\bar x)$. Then, inserting this inequality in the previous one and since $\delta>0$, we have 
$$V(\bar y)-V(\bar x)+\frac{1}{2 \epsilon}|\bar x-\bar y|_{-1}^2 \leq  \omega (\delta, \epsilon)$$
so that, by \eqref{Va-1},
$$-K|\bar x-\bar y|_{-1}+\frac{1}{2 \epsilon}|\bar x-\bar y|_{-1}^2 \leq  \omega (\delta, \epsilon)$$
and thus 
$$\limsup_{\delta \to 0} \left [-K|\bar x-\bar y|_{-1}+\frac{1}{2 \epsilon}|\bar x-\bar y|_{-1}^2 \right] \leq  0$$
for every $\epsilon>0$.
Therefore, for every $\epsilon>0$, we have \eqref{eq:proof_barx-bary_1}. 

\textbf{Step 3.} In this step, using the fact that $V$ is a viscosity  supersolution of the HJB equation, we prove that
\begin{align}\label{eq:proof_supersolution_ineq}
0 \geq -\rho V(\bar y)+\frac{1}{\epsilon} \langle \bar y,\tilde A^* B(\bar x- \bar y)  \rangle_X - H \left (\bar y,\frac{1}{\epsilon}B(\bar x- \bar y)-2 \delta \bar y,Y_N - \frac{2}{\epsilon}BQ_N-2 \delta I \right),
\end{align}
\begin{align}\label{eq:proof_gradient_equal_zero}
-D \phi (\bar x)  -Dg( \bar x) =\frac{1}{\epsilon}B(\bar x - \bar y), \quad \quad 
-D^2 \phi (\bar x) -D^2g(\bar x)\leq \frac{2}{\epsilon}BQ_N + X_N,
\end{align}
and 
\begin{align}\label{eq:gradient_g}
Dg( \bar x)=\frac{h_0'(|\bar x|_X)} {|\bar x|_X }\bar x  \in D(\tilde A^*), \quad \textit{in particular}\,\,\,\bar x \in D(\tilde A^*).
\end{align}
Here, $X_N, Y_N \in \mathcal S(X)$ are such that $X_N=P_N X_N P_N , Y_N=P_N Y_N P_N$  and they satisfy \eqref{eq:maximum_principle_2_derivatives} below.

Define
\begin{align*}
& u_1(x)=-\phi(x)-g(x)-\frac{\left\langle B Q_{N}(\bar{x}-\bar{y}), x\right\rangle_X}{\epsilon}-\frac{\left|Q_{N}(x-\bar{x})\right|_{-1}^{2}}{\epsilon} +\frac{\left|Q_{N}(\bar{x}-\bar{y})\right|_{-1}^{2}}{2 \epsilon} \\
& v_1(y)=V(y)+\delta|y|_X^2-\frac{\left\langle B Q_{N}(\bar{x}-\bar{y}), y\right\rangle_X}{\epsilon}+\frac{\left|Q_{N}(y-\bar{y})\right|_{-1}^{2}}{\epsilon}
\end{align*}
and notice that 
$$
|x-y|_{-1}^{2}=\left|P_{N}(x-y)\right|_{-1}^{2}+\left|Q_{N}(x-y)\right|_{-1}^{2}
$$
and
$$
\begin{aligned}
\left|Q_{N}(x-y)\right|_{-1}^{2} \leq 2\left\langle B Q_{N}(\bar{x}-\bar{y}), x-y\right\rangle_X &+2\left|Q_{N}(x-\bar{x})\right|_{-1}^{2} +2\left|Q_{N}(y-\bar{y})\right|_{-1}^{2}-\left|Q_{N}(\bar{x}-\bar{y})\right|_{-1}^{2}
\end{aligned}
$$
with equality at $\bar x, \bar y$. Thus, since $(\bar x, \bar y)$ is a global maximum for $\bar \Phi$,
$$u_{1}(x)-v_{1}(y)-\frac{1}{2 \epsilon}\left|P_{N}(x-y)\right|_{-1}^{2}$$
has a strict global maximum over $X \times X$ at $(\bar x,\bar y)$.\\

Denote $\bar x_N=P_N \bar x, \bar y_N=P_N \bar y$. By \cite[Corollary 3.28]{fgs_book} there exist test functions $\{ \varphi_{k}\},\{ \psi_{k} \}_k \subset \Phi$ and points $\{(x_{k}, y_{k})\}_k \subset X \times X$ such that $u_{1}(x)-\varphi_{k}(x)$ has a maximum at $x_{k}$, $v_{1}(y)-\psi_{k}(y)$ has a minimum at $y_{k}$ and such that 
$$\left(x_{k}, D \varphi_{k}\left(x_{k}\right), D^{2} \varphi_{k}\left(x_{k}\right)\right) \stackrel{k \rightarrow \infty}{\longrightarrow}\left(\bar{x},  \frac{B\left(\bar{x}_{N}-\bar{y}_{N}\right)}{\epsilon}, X_{N}\right)$$
in $X \times X_{2} \times \mathcal{L}\left(X_{-1}, X_{1}\right)$
$$\left(y_{k},  D \psi_{k}\left(y_{k}\right), D^{2} \psi_{k}\left(y_{k}\right)\right) \stackrel{k \rightarrow \infty}{\longrightarrow}\left(\bar{y}, \frac{B\left(\bar{x}_{N}-\bar{y}_{N}\right)}{\epsilon}, Y_{N}\right)$$
in $X \times X_{2} \times \mathcal{L}\left(X_{-1}, X_{1}\right)$,
for some $X_N ,Y_N \in \mathcal S(X)$ satisfying  $X_N=P_N X_N P_N, Y_N=P_N Y_N P_N$ and 
\begin{align}\label{eq:maximum_principle_2_derivatives}
\frac{3}{\epsilon}\left(\begin{array}{cc}B P_{N} & 0 \\ 0 & B P_{N}\end{array}\right) \leq\left(\begin{array}{cc}X_N & 0 \\ 0 & -Y_N\end{array}\right) \leq \frac{3}{\epsilon}\left(\begin{array}{cc}B P_{N} & -B P_{N} \\ -B P_{N} & B P_{N}\end{array}\right).
\end{align}
Since $v_{1}(y)-\psi_{k}(y)$ has a minimum at $y_{k}$, by defining
\begin{align*}
&\bar \phi(y)=-\psi_k(y)-\frac{\left\langle B Q_{N}(\bar{x}-\bar{y}), y\right\rangle_X}{\epsilon}+\frac{\left|Q_{N}(y-\bar{y})\right|_{-1}^{2}}{\epsilon}, \quad \quad  \bar g(y)=\delta |y|_X^2,
\end{align*}
$V+\bar \phi + \bar g$ has a minimum at $y_k$. Since $V$ is a viscosity supersolution of the HJB equation, we have
\begin{align*}
\rho V(y_k)+\langle y_k, \tilde A^*D \bar \phi (y_k) \rangle_X + H(y_k,-D \bar \phi (y_k)-D \bar g (y_k),-D^2 \bar \phi (y_k)-D^2 \bar g (y_k)) \geq 0.
\end{align*}
Note that $D \left|Q_{N}(y-\bar{y})\right|_{-1}^{2}=2B Q_N (y-\bar{y})$. Therefore,
\begin{align*}
&D \bar \phi(y_k) =-D \psi_k(y_k)-\frac{1}{\epsilon}B Q_N(\bar x- \bar y)+ 2BQ_N (y_k-\bar y)\\
& \quad \xrightarrow[]{k \to \infty}  -\frac{1}{\epsilon}B(\bar x_N-\bar y_N)-\frac{1}{\epsilon}B Q_N(\bar x- \bar y)
= -\frac{1}{\epsilon}BP_N(\bar x -\bar y)-\frac{1}{\epsilon}B Q_N(\bar x- \bar y)=  -\frac{1}{\epsilon}B(\bar x-\bar y)\\
\end{align*}
and
$D^2 \bar \phi(y_k) =-D^2 \psi_k(y_k)+\frac{2}{\epsilon}BQ_N \xrightarrow[]{k \to \infty}  -Y_N + \frac{2}{\epsilon} BQ_N,$
so that letting $k \to \infty$ we have \eqref{eq:proof_supersolution_ineq}. 

Since $u_{1}(x)-\varphi_{k}(x)$ has a maximum at $x_k$, we have 
\begin{align*}
D(u_{1}(x_k)-\varphi_{k}(x_k))& =-D \phi (x_k) -Dg(x_k) -\frac{1}{\epsilon}BQ_N(\bar x - \bar y)-2 BQ_N(x_k-\bar x) -D\varphi_k(x_k) = 0
\end{align*}
and
\begin{align*}
D^2(u_{1}(x_k)-\varphi_{k}(x_k))& =-D^2 \phi (x_k) -D^2g(x_k)-\frac{2}{\epsilon}BQ_N-D^2\varphi_k(x_k) \leq 0,
\end{align*}
so that, letting $k \to \infty$ above, we obtain \eqref{eq:proof_gradient_equal_zero}.
%\begin{align*}
%-D \phi (\bar x)  -Dg( \bar x) -\frac{1}{\epsilon}BQ_N(\bar x - \bar y)-\frac{1}{\epsilon} B (\bar x_N-\bar y_N)=0
%\end{align*}
%i.e. as $\bar x_N=P_N \bar x, \bar y_N=P_N \bar y$
We now note that \eqref{eq:proof_gradient_equal_zero} and \eqref{eq:gradient_radial} imply \eqref{eq:gradient_g}.

\textbf{Step 4.} In this last step, we conclude the proof of the proposition.

By \eqref{eq:proof_supersolution_ineq}, \eqref{eq:proof_gradient_equal_zero}, \eqref{eq:gradient_g}, the structure conditions
\eqref{eq:H_decreasing}, \eqref{eq:H_lambda_BQ_N}, \eqref{eq:H_norm_-1}, \eqref{eq:Hamiltonian_local_lip}, the weak $B$-condition with $C_0=0$ and \eqref{Va-1}, we have
\begin{small}
\begin{align*}
 \rho& V_\epsilon(\bar x)+ \langle \bar x, \tilde A^*(D \phi (\bar x) + D g(\bar x)) \rangle_X + H \left ( \bar x, -D \phi (\bar x) - D g(\bar x), -D^2 \phi (\bar x) - D^2 g(\bar x) \right) \\
 & =\rho V_\epsilon(\bar x) -\frac{1}{\epsilon} \langle \bar x, \tilde A^*B(\bar x - \bar y)  \rangle_X + H \left ( \bar x, \frac{1}{\epsilon}  B(\bar x - \bar y), -D^2 \phi (\bar x) - D^2 g(\bar x) \right) \\
&  \geq \rho V_\epsilon(\bar x) -\frac{1}{\epsilon} \langle \bar x, \tilde A^*B(\bar x - \bar y)  \rangle_X + H \left ( \bar x, \frac{1}{\epsilon}  B(\bar x - \bar y), -D^2 \phi (\bar x) - D^2 g(\bar x) \right) \\
& \quad - \rho V(\bar y)+\frac{1}{\epsilon} \langle \bar y,\tilde A^* B(\bar x- \bar y)  \rangle_X - H \left (\bar y,\frac{1}{\epsilon}B(\bar x- \bar y)-2 \delta \bar y,Y_N - \frac{2}{\epsilon}BQ_N-2 \delta I \right) \\
&  \geq \rho V_\epsilon(\bar x) -\frac{1}{\epsilon} \langle \bar x-\bar y, \tilde A^*B(\bar x - \bar y)  \rangle_X + H \left ( \bar x, \frac{1}{\epsilon}  B(\bar x - \bar y),  \frac{2}{\epsilon}BQ_N + X_N \right) \\
&  \quad -\rho V(\bar y)- H \left (\bar y,\frac{1}{\epsilon}B(\bar x- \bar y),Y_N - \frac{2}{\epsilon}BQ_N \right) -\delta C (1+|\bar y|_X^2)\\
&  = \rho V_\epsilon(\bar x) -\frac{1}{\epsilon} \langle \bar x-\bar y, \tilde A^*B(\bar x - \bar y)  \rangle_X + H \left ( \bar x, \frac{1}{\epsilon}  B(\bar x - \bar y),  X_N \right) \\
&  \quad -\rho V(\bar y)- H \left (\bar y,\frac{1}{\epsilon}B(\bar x- \bar y),Y_N \right) -\omega_{\epsilon}(1/N) -\delta C (1+|\bar y|_X^2)\\
& \geq \rho V_\epsilon(\bar x)-\rho V(\bar y) -\frac{1}{\epsilon} \langle \bar x-\bar y, \tilde A^*B(\bar x - \bar y)  \rangle_X  -C\left (|\bar x- \bar y|_{-1}\left(1+\frac{|\bar x- \bar y|_{-1}}{\epsilon}\right) \right)
-\omega_{\epsilon}(1/N) -\delta C  (1+|\bar y|_X^2)\\
& \geq \rho V_\epsilon(\bar x)-\rho V (\bar x) +\rho V (\bar x) -\rho V(\bar y) -C\left (|\bar x- \bar y|_{-1}\left(1+\frac{|\bar x- \bar y|_{-1}}{\epsilon}\right) \right) -\omega_{\epsilon}(1/N)-\delta C  (1+|\bar y|_X^2)\\
& \geq \rho V_\epsilon(\bar x)-\rho V (\bar x) -\rho K|\bar x-\bar y|_{-1}-C\left (|\bar x- \bar y|_{-1}\left(1+\frac{|\bar x- \bar y|_{-1}}{\epsilon}\right) \right)-\omega_{\epsilon}(1/N) -\delta C  (1+|\bar y|_X^2),
\end{align*}
\end{small}
where $\omega_{\epsilon}$ is a modulus coming from \eqref{eq:H_lambda_BQ_N}.
We now let first $N \to \infty$ and then take $\limsup_{\delta \to 0}$, so that by \eqref{eq:bar_x_=_x_0}, \eqref{eq:delta_bary_squared_to_zero}, \eqref{eq:proof_barx-bary_1} and \eqref{eq:convergence_inf_convolution} we have
\begin{align*}
 \rho V_\epsilon(x_0)& +  \langle \bar x, \tilde A^*(D \phi (\bar x) + D g(\bar x)) \rangle_X + H \left ( x_0, -D \phi (x_0) - D g(x_0), -D^2 \phi (x_0) - D^2 g(x_0) \right)\\
 &  \geq \rho V_\epsilon(x_0)-\rho V (x_0)   - C_1 \epsilon  \geq -\nu(\epsilon) - C_1 \epsilon =: -\gamma(\epsilon),
 \end{align*}
where  $\gamma(\epsilon) \to 0$ as $\epsilon \to 0$. We emphasize that the modulus $\nu$ and the constant $C_1$, and hence $\gamma$, are independent of $\phi, g, x_0$. Thus we proved that $V_\epsilon$ is a viscosity solution of \eqref{pertHJB-gamma-eps} in the stronger sense of Proposition \ref{prop:inf_conv_subsolution_perturbed}.
%By \eqref{eq:gradient_g} we have $\langle \bar x, \tilde A^* D g(\bar x) \rangle_X = \frac{ h_0'(|\bar x|_X) }{ |\bar x|_X} \langle \tilde A^* \bar x,  \bar %x \rangle_X \leq 0$ so that we can drop it and get  
%\begin{align*}
% \rho w_\epsilon(x_0)& +  \langle \bar x, \tilde A^*D \phi (\bar x)  \rangle_X + H \left ( x_0, -D \phi (x_0) - D g(x_0), -D^2 \phi (x_0) - D^2 g(x_0) 
%\right)  \geq  -\gamma(\epsilon).
% \end{align*}
 %The last two inequalities imply the thesis of the theorem.
\end{proof}
\section{Non-smooth Dynkin's Formula}\label{sec:dynkin_formula}
In Sections \ref{sec:dynkin_formula} and \ref{sec:lions_approx} we follow the setup and technique introduced in \cite{lions-infdim1} and modify them to accommodate an equation with an unbounded term. 
We define the space $\mathcal D$ by
$$\mathcal D=\Big \{ \phi \in C^{1,1}(X_{-1}) : \eqref{eq:second_derivative_space_D} \textit{ exists and is uniformly continuous on }  X_{-1}\},$$
where  
\begin{equation}\label{eq:second_derivative_space_D}
\lim_{t \to 0} \frac 1 t \langle D_{-1} \phi(x+tk)-D_{-1} \phi(x),h \rangle_{-1} ,\quad \forall x,h,k \in X_{-1}.
\end{equation}
Functions in the space $\mathcal D$ possess second order derivatives in some sense. 
Note that for $\phi \in \mathcal D$, since $D \phi=BD_{-1}\phi$, we have
\begin{equation}\label{eq:second_derivative_space_D_derivatives_in_X}
\lim_{t \to 0} \frac 1 t \langle D \phi(x+tk)-D \phi(x),h \rangle_X=\lim_{t \to 0} \frac 1 t \langle D_{-1} \phi(x+tk)-D_{-1} \phi(x),h \rangle_{-1}, \quad \forall x,k \in X_{-1}, h \in X
\end{equation}
and it is uniformly continuous with respect to $x \in X_{-1}$.
Moreover, we have
\begin{equation}\label{eq:second_derivative_space_D_derivatives_in_X_-1_bilinear_form}
\lim_{t \to 0} \frac 1 t \langle D_{-1} \phi(x+tk)-D_{-1} \phi(x),h \rangle_{-1} = \langle A(x)k,h \rangle_{-1} ,\quad \forall h,k \in X_{-1}
\end{equation}
and the limit is uniformly continuous with respect to $x \in X_{-1}$. Here $A(x)$ are
bounded, linear, self-adjoint operators on $X_{-1}$ such that $|A(x)|_{\mathcal L(X_{-1})} \leq C_\phi,$ where $C_\phi$ is the Lipschitz constant of $D_{-1} \phi$. We will denote $A(x)=\tilde D_{-1}^2\phi(x)$. We point out that $\tilde D_{-1}^2\phi(x)$ is not the Fr\'echet or the Gateaux derivative. It is a sort of a weak Gateaux second order derivative in $X_{-1}$ in the sense that 
$$ \frac 1 t \left( D_{-1} \phi(x+tk)-D_{-1} \phi(x) \right) \stackrel{X_{-1}}{\rightharpoonup} \tilde D_{-1}^2\phi(x)k, \quad \forall x, k \in X_{-1}, $$
where $\stackrel{X_{-1}}{\rightharpoonup}$ is the weak convergence in $X_{-1}$.
Denoting  $\tilde D^2\phi(x)=B\tilde D_{-1}^2\phi(x)$, by \eqref{eq:second_derivative_space_D_derivatives_in_X} we have
\begin{equation}\label{eq:second_derivative_space_D_derivatives_in_X_bilinear_form}
\lim_{t \to 0} \frac 1 t \langle D \phi(x+tk)-D \phi(x),h \rangle_X= \langle \tilde D^2\phi(x) k,h \rangle_X ,\quad \forall  k,h \in X
\end{equation}
and the right-hand side of \eqref{eq:second_derivative_space_D_derivatives_in_X_bilinear_form} is uniformly continuous in $x\in X$. Here, 
 $\tilde D^2\phi(x)$ is a bounded, linear, self-adjoint operator on $X$. We have $|\tilde D^2\phi(x)|_{\mathcal L(X)} \leq  C_\phi$ and
\begin{equation*}
\frac 1 t \left(D \phi(x+tk)-D \phi(x) \right) \rightharpoonup \tilde D^2\phi(x)k, \quad \forall x, k \in X.
\end{equation*}
Hence the quantities
\begin{equation}\label{eq:second_derivative_space_D_unif_continuous}
\begin{aligned}
& \langle \tilde D_{-1}^2\phi(x)k,h \rangle_{-1} \quad \forall h,k \in X_{-1},\quad \quad \langle \tilde D^2\phi(x) k,h \rangle_X \quad \forall  k,h \in X
\end{aligned}
\end{equation}
are uniformly continuous with respect to $x \in X_{-1}$ and $x \in X$ respectively.
\begin{remark}\label{rem:equivalent_condition_in_D}
We note that \eqref{eq:second_derivative_space_D} can be replaced by the following condition: $\partial_{ij} \phi$ exists and is uniformly continuous on $X_{-1}$ for every $i,j \in \mathbb N$. Here $\partial_i$ indicates the partial derivative with respect to $e_i$ (in $X_{-1}$), where $\{e_i\}_{i \in \mathbb N}$ is the basis of $X_{-1}$ defined in Subsection \ref{sec:operator_B}.
\end{remark}
We denote $\tilde D^2_{x_0^2}\phi(x):=P_{x_0}\tilde D^2 \phi(x)P_{x_0}$, where $P_{x_0}$ is the orthogonal projection in $X$ onto the $\mathbb R^n$ component.
\begin{lemma}\label{lemma:D^2_x_0^2phi_in_D}
Let $\phi \in \mathcal D$ and consider its restriction to $X$. Then, for every $\bar x_1 \in L^2$, we have $\phi^{\bar x_1}:=\phi(\cdot,\bar x_1) \in C^2(\mathbb R^n)$, hence $D^2_{x_0^2}\phi(x)=\tilde D^2_{x_0^2}\phi(x)$. Moreover $D^2_{x_0^2}\phi(x)$ is uniformly continuous on $X$.
\end{lemma}
\begin{proof}
Let $\{v_i \}$ be an orthonormal basis of $\mathbb R^n$. Fixing any $i,j \leq n$ and considering \eqref{eq:second_derivative_space_D_derivatives_in_X_bilinear_form} with $k=(v_i,0), h=(v_j,0)$ we have that $\frac{\partial^2 }{\partial x_i x_j} \phi^{\bar x_1}$ exists and it is uniformly continuous on $X$. Then, by the fact that in finite dimensional spaces the continuity of all second order partial derivatives implies $C^2$, we have $\phi^{\bar x_1}:=\phi(\cdot,\bar x_1) \in C^2(\mathbb R^n)$ and $D^2_{x_0^2}\phi(x)=\tilde D^2_{x_0^2}\phi(x)$. The uniform continuity in $X$ follows. 
\end{proof}\begin{flushleft}
Thanks to this lemma we will denote
\end{flushleft}
$$\tilde D^2\phi(x)k=\begin{bmatrix}
 D^2_{x_0^2}\phi(x) & \tilde D^2_{x_0 x_1}\phi(x)\\
\tilde D^2_{x_1 x_0}\phi(x) &\tilde D^2_{x_1^2}\phi(x)
\end{bmatrix}\begin{bmatrix}k_0\\
k_1\end{bmatrix}, \quad \forall x=(x_0,x_1), k=(k_0,k_1) \in X.$$
We now prove Dynkin's formula for functions of the form $e^{-\rho t} \phi(x),$ where $\phi \in \mathcal D$. The formula could be extended to more general functions, but we restrict ourselves to functions $e^{-\rho t} \phi(x)$ since only such functions will be used in the proof of the Verification Theorem.

\begin{lemma}[Dynkin's formula]\label{lemma:ito}
Let Assumptions \ref{hp:state}, \ref{hp:discount} hold.
 Fix any initial datum $x \in X$, any control $u(\cdot) \in \mathcal {\overline U}$ and denote by $Y(t)$ the solution of the state equation \eqref{eq:abstract_dissipative_operator}. Let $R,T>0$ and define
 $
\chi^R:=\inf \{s \in[0, T]:|Y(s)|_X>R\}.
$
Then, for every $\phi \in \mathcal D$, for every $0 \leq t \leq T$
\begin{small}
\begin{align*}
 \mathbb E \left[  e^{-\rho (t \wedge \chi^R)}   \phi(Y(t \wedge \chi^R)) \right] 
&=\phi( x)+\mathbb E  \int_0^{t \wedge \chi^R} \Big [- \rho e^{-\rho s} \phi(Y(s)) +   e^{-\rho s}  \langle  Y(s), \tilde A^* D \phi(Y(s))  \rangle_X\\
& \quad +e^{-\rho s } \langle \tilde  b(Y(s),u(s)), D \phi(Y(s))  \rangle_X   
+ \frac 1 2 e^{-\rho s} \tr \left ( \sigma(Y(s))\sigma(Y(s))^* \tilde D^2\phi(Y(s)) \right ) \Big ]   ds\\
&=\phi( x)+\mathbb E  \int_0^{t \wedge \chi^R} \Big [- \rho e^{-\rho s} \phi(Y(s)) +   e^{-\rho s}  \langle  Y(s), \tilde A^* D \phi(Y(s))  \rangle_X\\
& \quad  +e^{-\rho s } \tilde  b_0(Y(s),u(s)) \cdot D_{x_0} \phi(Y(s))
+ \frac 1 2 e^{-\rho s} \tr \left ( \sigma_0(Y(s))\sigma_0(Y(s))^T D^2_{x_0^2}\phi(Y(s)) \right ) \Big ]   ds.
\end{align*}
\end{small}
\end{lemma}
\begin{proof}
Let $R_N:=(N I-\tilde A)^{-1}$ be the resolvent operator of $\tilde A$ for $N \in \mathbb N$ 
and let $\tilde A_N=N \tilde AR_N$ be the Yosida approximation of $\tilde A$. Denote by $Y^N$ the solution of the state equation with $\tilde A_N$ in place of $\tilde A$, that is
\begin{equation*}
dY^N(s) = [\tilde A_N Y^N(s)+\tilde b(Y^N(s),u(s))]dt + \sigma(Y^N(s))\,dW(s), \quad Y^N(0)= x.
\end{equation*}
By standard theory, e.g. \cite[Proposition 1.132]{fgs_book},  we have
\begin{align}\label{eq:proof_conv_expect_YN-Y}
\lim_{N \to \infty} \mathbb E \left[ \sup_{s \in [0,T]} \left |Y^N(s)-Y(s) \right|_X^2 \right]=0, \quad \forall T>0.
\end{align}
We define
$
\chi_N^R:=\inf \left\{s \in[0, T]:\left|Y^N(s)\right|_X>R\right\}.
$
By \eqref{eq:proof_conv_expect_YN-Y}, up to a subsequence, we have $\sup_{s \in [0,T]} \left |Y^N(s)-Y(s) \right|_X \to 0$ a.s. so that $\lim_{N \to \infty} \chi_N^R=\chi^R \quad a.s. \quad \forall R>0 .$

Since $\tilde A_N \in \mathcal L(X)$ we can apply the non-smooth Ito's formula from \cite[Lemma III.2]{lions-infdim1} which holds for equations with bounded terms and for $\phi$, $D\phi, \tilde D^2\phi$ bounded. Note that in our case $\Phi, D\phi, D^2\phi$  are only bounded on bounded sets of $X$, however, since we are using the stopping time $\chi_N^R$, this is enough in order to apply \cite[Lemma III.2]{lions-infdim1}. (Observe that in \cite{lions-infdim1} what we call $\tilde D^2 \phi$ is denoted by $D^2 \phi$, see page 246 there.) Therefore we have
\begin{small}
\begin{align}\label{eq:proof_ito_projected}
\mathbb E  \left[  e^{-\rho (t \wedge \chi_N^R)}  \phi(Y^N(t \wedge \chi_N^R)) \right] \nonumber 
& =\phi( x)+\mathbb E  \int_0^{t \wedge \chi_N^R} \Big [- \rho e^{-\rho s} \phi(Y^N(s)) +   e^{-\rho s}  \langle \tilde A_N  Y^N(s), D \phi(Y^N(s))  \rangle_X \nonumber \\
& \quad +e^{-\rho s } \langle \tilde  b(Y^N(s),u(s)), D \phi(Y^N(s))  \rangle_X    
+ \frac 1 2 e^{-\rho s} \tr \left ( \sigma(Y^N(s))\sigma(Y^N(s))^* \tilde D^2\phi(Y^N(s)) \right ) \Big ]   ds \nonumber \\
&=\phi( x)+\mathbb E  \int_0^{t \wedge \chi_N^R} \Big [- \rho e^{-\rho s} \phi(Y^N(s)) +   e^{-\rho s}  \langle  Y^N(s), (\tilde A_N)^* D \phi(Y^N(s))  \rangle_X \nonumber \\
& \quad  +e^{-\rho s } \tilde  b_0(Y^N(s),u(s)) \cdot D_{x_0} \phi(Y^N(s))
+ \frac 1 2 e^{-\rho s} \tr \left ( \sigma_0(Y^N(s))\sigma_0(Y^N(s))^T D^2_{x_0^2}\phi(Y^N(s)) \right ) \Big ]   ds.
\end{align}
\end{small}
We now prove that 
\begin{align*}
\lim_{N \to \infty}|(\tilde A_N)^* D \phi(Y^N(s))-\tilde A^* D \phi(Y(s))|_X =0 \quad a.s. \,\, \forall s \geq 0.
\end{align*}
Indeed we have
\begin{align*}
|(\tilde A_N)^* D \phi(Y^N(s))-\tilde A^* D \phi(Y(s))|_X &\leq |(\tilde A_N)^*[D \phi(Y^N(s))- D \phi(Y(s))]|_X+|(\tilde A_N)^* D \phi(Y(s))-\tilde A^* D \phi(Y(s))|_X.
\end{align*}
Consider the first term.  We first note that  by \eqref{eq:A*B^1/2} and the fact that $ |(R_N)^* |_{\mathcal L(X)} =|R_N |_{\mathcal L(X)}\leq 1/N$ (as $\tilde A$ is maximal dissipative), we have $|(\tilde A_N)^*B^{1/2}|_{\mathcal L(X)}=N| (R_N)^*\tilde A^*  B^{1/2}|_{\mathcal L(X)} \leq N |R_N |_{\mathcal L(X)}  |\tilde A^*  B^{1/2}|_{\mathcal L(X)} \leq C$. Hence, since $D\phi=BD_{-1}\phi$ for  $\phi \in C^{1,1} ( X_{-1})$,
\begin{align*}
|(\tilde A_N)^*[D \phi(Y^N(s))- D \phi(Y(s))]|_X&=|(\tilde A_N)^*B^{1/2}[B^{1/2}D_{-1} \phi(Y^N(s))- B^{1/2}D_{-1} \phi(Y(s))]|_X\\
&\leq |(\tilde A_N)^*B^{1/2} |_{\mathcal L(X)} |D_{-1} \phi(Y^N(s))- D_{-1} \phi(Y(s))|_{-1}\\
& \leq C |D_{-1} \phi(Y^N(s))- D_{-1} \phi(Y(s))|_{-1} \leq C | Y^N(s)- Y(s)|_{-1} \\
& \leq C | Y^N(s)- Y(s)|_X \xrightarrow[]{N \to \infty} 0, \quad a.s.\,\,\forall s \geq 0.
\end{align*}
For the second term, we note again that $D \phi(Y(s))=BD_{-1} \phi(Y(s)) \in R(B^{1/2})= D(\tilde A^*)$ for every $s \leq T$, so that by the fundamental property of Yosida approximations
\begin{align*}
\lim_{N \to \infty}|\tilde A_N^* D \phi(Y(s))-\tilde A^* D \phi(Y(s))|_X=0, \quad a.s.\,\,\forall s \geq 0,
\end{align*}
and we have the claim. 

We also have $|(\tilde A_N)^*D \phi(x)|_X\leq |\tilde A^*D \phi(x)|_X\leq C_R$ for $|x|_X\leq R$.
The lemma now follows by letting $N \to \infty$  in \eqref{eq:proof_ito_projected} and using the dominated convergence theorem (note the stopping time $\chi_N^R$).
\end{proof}

\section{Second approximation in the space \texorpdfstring{$\mathcal D$}{D}}\label{sec:lions_approx}
In this section we use a regularization procedure inspired by \cite{lions-infdim1} to produce functions $V_\epsilon^\eta \in \mathcal D$ approximating $V_\epsilon,$ which are almost classical supersolutions of perturbed HJB equations. Since  $V_\epsilon^\eta \in \mathcal D$, we will then be able to use Dynkin's formula in order to solve the optimal control problem.
\begin{lemma}\label{lemma:z_eta}
Let Assumptions \ref{hp:state} and \ref{hp:cost} hold.
Let $z \in C^{1,1}(X_{-1})$ be such that $D_{-1}z$ is bounded on $X_{-1}$. Suppose $z$ is a viscosity supersolution of
\begin{align*}
\rho z(x)-\langle  \tilde A^*Dz(x),x \rangle_X+\tilde H \left (x,D_{x_0}z(x),D^2_{x^2_0}z(x) \right) =-\gamma \quad \mbox{in}\,\,X.
\end{align*}
Then for every $\eta>0$ there exist $z^\eta \in \mathcal D$ such that 
\begin{equation}\label{eq:convergence_convolution}
|z-z^\eta|\leq C \eta, \quad |D_{-1}z-D_{-1}z^\eta|_{-1}\leq C \eta
\end{equation}
for some $C>0$ (independent of $\eta$)
and such that for every $R>0$, $z^\eta$ is a viscosity supersolution of
\begin{align*}
\rho z^\eta(x)-\langle  \tilde A^*Dz^\eta(x),x \rangle_X+\tilde H \left (x,D_{x_0}z^\eta(x), D^2_{x_0^2}z^\eta(x) \right) \geq -\gamma - \omega_R(\eta) \quad \mbox{in}\,\, B_R
\end{align*}
for some local moduli of continuity $\omega_R$.
\end{lemma}

\begin{proof}
The proof extends the ideas of \cite[Proof of Lemma IV.1]{lions-infdim1} to the case of HJB equations with unbounded operators. We point out that in this section we use notation which is different from the one used elsewhere.

We take the orthonormal basis $\{e_i \}$ of $X_{-1}$ defined in Subsection \ref{sec:operator_B}, where $e_i \in X$ for $i\in\mathbb N$, and we identify $X_{-1}$ with $ l^2(\mathbb N)$ or equivalently with $\mathbb R^k \times X^{k,\perp}$ for  $k \in \mathbb N$, where recall from Subsection \ref{sec:operator_B} $$X^k=\mbox{span}(f_1,...,f_k)=\mbox{span}(e_1,...,e_k).$$
Hence for   $x \in  X_{-1}$  we write $x=(x^1,x^2,...)=(x^1,...,x^k,x')=(x_k,x')$ where $x_k=(x^1,...,x^k) \in X^k \sim \mathbb R^k$, $x'=(x^{k+1},x^{k+2},...) \in X^{k,\perp}$. Here $\{x_i\}$ represent the coordinates of $x \in X_{-1}$ with respect to the orthonormal basis $\{ e_i\}$ of $X_{-1}$. Since $X \subset X_{-1}$, with this notation any element in $x \in X$ will also be denoted by $x=(x_k,x')$. We remark that this notation should not be confused with the notation $x=(x_0,x_1) \in X$ which is used in the rest of the paper, so in general we have $x_k \neq x_0$ and $x' \neq x_1$. We also point out that sometimes we will use $N$ instead of $k$, i.e. $x=(x_N,x')$.

Consider a standard mollifier function $ \rho \in C^\infty(\mathbb R)$ with $\mbox{supp}(\rho) \subset [-1,1]$, $\rho \geq 0$, $\int_\mathbb{R} \rho dx=1$. Let $\eta>0$. We define 

\begin{align*}
&z^\eta (x)=\lim_{k \to \infty} z^{\eta,k}(x), \quad \quad    z^{\eta,k}(x)=\int_{\mathbb R^k} z(y^1,...,y^k,x') \prod_{i=1}^k \rho_{\eta_i}(x^i-y^i)  dy^1...dy^k,
\end{align*}
for every  $x=(x^1,x^2,...)=(x^1,...,x^k,x')=(x_k,x') \in X_{-1}$ with $x_k \in \mathbb R^k \sim X^k, x' \in X^{k,\perp}$, $\eta_i:=\eta \sqrt{\lambda_i}/2^{i}$, $\rho_h(x)=(1/h )  \rho(x/h)$ for every $h >0$. Recall that the $\lambda_i$ are the eigenvalues of the operator $B$, see Subection \ref{sec:operator_B}. Note that
\begin{equation}\label{eq:sum_eta_i}
\sum_{i=1}^\infty \eta_i =c\eta.
\end{equation}
\textbf{Step 1:} We prove that $z^\eta$ is well defined, $z^\eta \in \mathcal D$ and it is close to $z$.
We first claim that 
\begin{align}\label{eq:proof_convolution_cauchy}
 \sup_{X_{-1}} |z-z^{\eta,1}|\leq C \eta_1, \quad \sup_{X_{-1}} |z^{\eta,k+1}-z^{\eta,k} |\leq C \eta_{k+1},\quad k\in\mathbb N.
 \end{align}
 Indeed, since $z$ is Lipschitz in $X_{-1}$, we have
 \begin{align*}
  |z(x)-z^{\eta,1}(x)|& = \Bigg | \int_\mathbb{R} \left  [ z(x^1,x')- z(y^1,x') \right ] \rho_{\eta_1}(x^1-y^1) dy^1 \Bigg | \leq C \int_\mathbb{R}   |x^1-y^1| \rho_{\eta_1}(x^1-y^1) dy^1 \leq C \eta_1
 \end{align*}
 and for $k \in \mathbb N$
 \begin{align*}
  |z^{\eta,k+1}(x)-z^{\eta,k}(x)| = &  \Bigg | \int_{\mathbb{R}^k} \prod_{i=1}^k \rho_{\eta_i}(x^i-y^i) \Bigg [ \int_\mathbb{R} z(y^1,...,y^k,y^{k+1},x') \rho_{\eta_{k+1}}(x^{k+1}-y^{k+1})   dy^{k+1} \\
  & - z(y^1,...,y^k, x^{k+1},x')  \Bigg ] dy^1...dy^k \Bigg |\\
   \leq &   \int_{\mathbb{R}^k} \prod_{i=1}^k \rho_{\eta_i}(x^i-y^i) \int_\mathbb{R} \big |z(y^1,...,y^k,y^{k+1},x') - z(y^1,...,y^k, x^{k+1},x') \big  | \\
 &   \times \rho_{\eta_{k+1}}(x^{k+1}-y^{k+1}) dy^{k+1} dy^1...dy^k \leq C \eta_{k+1}
 \end{align*}
We recall that $D_{-1}z^{\eta,k}=(D_{-1}z)^{\eta,k}$, i.e. the derivative of the convolution is the convolution of the derivative. Then, since $D_{-1}z$ is Lipschitz in $X_{-1}$, with a similar calculation (with $D_{-1}z$ in place of $z$ inside the integrals and these are now meant in the Bochner sense) we obtain
 \begin{align}\label{eq:proof_derivative_convolution_cauchy}
\sup_{X_{-1}} |D_{-1}z-D_{-1}z^{\eta,1}|_{-1} \leq C \eta_1, \quad \sup_{X_{-1}} |D_{-1}z^{\eta,k+1}-D_{-1}z^{\eta,k} |_{-1}\leq C \eta_{k+1},
\quad k\in\mathbb N.
\end{align}
Now, since $D_{-1}z$ is Lipshitz, if $x=(x_k,x')\in X_{-1}$, $\partial_{ij} z(x)$ exists for $i,j\leq k$ for a.e. $x_k$ and there exists $C>0$ such that $|\partial_{ij} z| \leq C$ (and $C>0$ is independent of $i,j,k$). Thus, for every $i,j \leq k$, we have $\partial_{ij}z^{\eta,k}=(\partial_{ij}z)^{\eta,k}$ so that
\begin{align}\label{eq:proof_second_derivative_bounded}
\sup_{X_{-1}} |\partial_{ij} z^{\eta,k}| \leq C \quad \forall i,j \leq k.
\end{align} 
Next we show that
\begin{align}\label{eq:proof_second_derivative_convergence}
\sup_{X_{-1}} |\partial_{ij} z^{\eta,k+1}-\partial_{ij} z^{\eta,k}| \leq C \frac{\eta_{k+1}}{\eta_i \eta_j}, \quad \forall i,j \leq k.
\end{align}
Indeed, note that $\partial_{i } \rho_{\eta_i}(x_i)=1/ \eta_i c_{\eta_i}(x_i)$ where $c_{\eta_i}(x_i)=1/ \eta_i  \rho' (x_i / \eta_i)$.  Assume $i \neq j$ (the case $i=j$ is treated in a similar way). Since $z$ is Lipschitz, we have
\begin{small}
 \begin{align*}
  |\partial_{ij }z^{\eta,k+1}(x)-& \partial_{ij } z^{\eta,k}(x)| =   \Bigg | \int_{\mathbb{R}^k} \partial_{i } \rho_{\eta_i}(x_i-y_i)  \partial_{j }\rho_{\eta_j}(x_j-y_j)   \prod_{h=1, h \neq i,j}^k \rho_{\eta_h}(x_h-y_h) \times \\
  & \qquad\qquad\quad
   \Bigg [ \int_\mathbb{R} z(y^1,...,y^k,y^{k+1},x') \rho_{\eta_{k+1}}(x^{k+1}-y^{k+1})   dy^{k+1}  - z(y^1,...,y^k, x^{k+1},x')  \Bigg ] dy^1...dy^k \Bigg |\\
   \leq & \frac{1}{\eta_i \eta_j}  \int_{\mathbb{R}^k} \prod_{h=1, h \neq i,j}^k \rho_{\eta_h}(x^h-y^h) c_{\eta_i}(x_i-y_i) c_{\eta_j}(x_j-y_j)
    \int_\mathbb{R} \big |z(y^1,...,y^k,y^{k+1},x') - z(y^1,...,y^k, x^{k+1},x') \big  | \\
 &  \qquad\qquad\qquad\qquad \times \rho_{\eta_{k+1}}(x^{k+1}-y^{k+1}) dy^{k+1} dy^1...dy^k  \leq C \frac{\eta_{k+1}}{\eta_i \eta_j}.
 \end{align*}
 \end{small}
 Now observe that $D_{-1} \partial_{ij}z^{\eta,k} =\partial_{ij} D_{-1}z^{\eta,k}$ for every $i,j \leq k$. Then, a similar calculation as that done to prove \eqref{eq:proof_second_derivative_convergence} (with $D_{-1}z$ in place of $z$ inside the integrals), since $D_{-1}z$ is Lipschitz in $X_{-1}$, we obtain
 \begin{align}\label{eq:proof_third_derivative_convergence}
\sup_{X_{-1}} |D_{-1} \partial_{ij} z^{\eta,k+1}-D_{-1} \partial_{ij} z^{\eta,k}|_{-1}=\sup_{X_{-1}} |\partial_{ij} D_{-1} z^{\eta,k+1}-\partial_{ij} D_{-1} z^{\eta,k}|_{-1} \leq C \frac{\eta_{k+1}}{\eta_i \eta_j}, \quad \forall i,j \leq k.
\end{align}
 Finally, since $z$ is Lipschitz in $X_{-1}$ and thus $D_{-1}z$ is bounded, using a similar calculation as that to prove \eqref{eq:proof_third_derivative_convergence}, we get
 \begin{align}\label{eq:proof_third_derivative_bounded}
 \sup_{X_{-1}}| D_{-1} \partial_{ij} z^{\eta,k}|_{-1}= \sup_{X_{-1}}|  \partial_{ij} D_{-1}z^{\eta,k}|_{-1} \leq \frac{C}{\eta_i \eta_j} , \quad \forall i,j \leq k.
 \end{align}
 
Set $g_k=z^{\eta,k}-z^{\eta,1}$. Note that by \eqref{eq:sum_eta_i} and \eqref{eq:proof_convolution_cauchy}, we have $\{g_k\} \subset C^{1,1}_b (X_{-1})$, where $C^{1,1}_b (X_{-1})$ is the subspace of functions in $C^{1,1} (X_{-1})$ which are bounded and have bounded derivatives. Moreover, $g^{k+1}-g^k=z^{\eta,k+1}-z^{\eta,k}$, so that by \eqref{eq:sum_eta_i}, \eqref{eq:proof_convolution_cauchy}, \eqref{eq:proof_derivative_convolution_cauchy}, $\{g_k \}$ is a Cauchy sequence in $C^{1}_b (X_{-1})$, where $C^{1}_b (X_{-1})$ is the subspace of functions in $C^{1} (X_{-1})$ which are bounded and have bounded derivatives. Then $g_k \to g$ in $C^{1}_b (X_{-1})$ as $k\to \infty$ to a function $g \in C^{1}_b (X_{-1})$ of the form $g=z^{\eta}-z^{\eta,1}$ for some $z^{\eta} \in C^{1} (X_{-1})$. This implies that
   \begin{equation}\label{eq:proof_convergence_convolution}
 \lim_{k \to \infty} \sup_{X_{-1}} |z^{\eta,k} - z^{\eta}| =0, \quad  \lim_{k \to \infty} \sup_{X_{-1}} |D_{-1}z^{\eta,k} - D_{-1}z^{\eta}|_{-1} =0.
 \end{equation} 
 Note that since $z, D_{-1}z$ are Lipschitz in $X_{-1}$, $z^{\eta,k},D_{-1}z^{\eta,k}$ are families of  Lipschitz functions with respect to $|\cdot |_{-1}$ with Lipschitz constants of $z, D_{-1}z$ (so independent of $\eta,k$).
Thus, letting $k \to \infty$, we derive that $z^\eta \in C^{1,1}(X_{-1})$ and $z^\eta, D_{-1}z^\eta$ are Lipschitz with respect to $|\cdot |_{-1}$ with Lipschitz constants independent of $\eta$.

Now, by  \eqref{eq:proof_convolution_cauchy}, \eqref{eq:proof_derivative_convolution_cauchy}, \eqref{eq:proof_convergence_convolution}, we have
 \begin{align*}
 |z^{\eta}(x)-z(x)|\leq  |z^{\eta} (x)- z^{\eta,k+1}(x)|+ \sum_{i=1}^{k} |z^{\eta,i+1}(x)-z^{\eta,i}(x)| +  |z^{\eta,1}(x)-z(x)| \leq  \omega_\eta \left (\frac 1 k \right) +C \sum_{i=0}^{k} \eta_{i+1}
 \end{align*}
 and
  \begin{align*}
 |D_{-1}z^{\eta}(x)-D_{-1}z(x)|_{-1} &\leq  |D_{-1}z^{\eta} (x)- D_{-1}z^{\eta,k+1}(x)|_{-1}\\
 &+ \sum_{i=1}^{k} |D_{-1}z^{\eta,i+1}(x)-D_{-1}z^{\eta,i}(x)|_{-1} +  |D_{-1}z^{\eta,1}(x)-D_{-1}z(x)|_{-1}\\
 &  \leq  \omega_\eta \left (\frac 1 k \right) +C \sum_{i=0}^{k} \eta_{i+1}
 \end{align*}
 for some modulus $\omega_\eta$, so that, by letting $k\to \infty$ and recalling \eqref{eq:sum_eta_i}, we obtain \eqref{eq:convergence_convolution}.
 
Proceeding in a similar way, by \eqref{eq:proof_second_derivative_bounded}, \eqref{eq:proof_second_derivative_convergence}, \eqref{eq:proof_third_derivative_convergence}, \eqref{eq:proof_third_derivative_bounded}, we have 
 \begin{equation}\label{eq:proof_second_derivative_convergence_to_z_eta}
 \lim_{k \to \infty} \sup_{X_{-1}} |\partial_{ij}z^{\eta,k} - \partial_{ij} z^\eta| =0, \quad  \lim_{k \to \infty} \sup_{X_{-1}} |D_{-1} \partial_{ij}z^{\eta,k} - D_{-1}\partial_{ij} z^\eta|_{-1} =0,  \quad \forall i,j \in \mathbb N.
 \end{equation}
Note that, since $z, D_{-1}z$ are Lipschitz in $X_{-1}$, by the properties of convolutions we have that for every $i,j$, $\partial_{ij}z^{\eta,k}, D_{-1} \partial_{ij}z^{\eta,k}=\partial_{ij}D_{-1}  z^{\eta,k}$ are  families of Lipschitz functions with respect to the $|\cdot|_{-1}$ norm with Lipschitz constants independent of $\eta,k$. Letting $k \to \infty$, it follows that for every $i,j$ the functions $\partial_{ij}  z^\eta \in C^{1,1}(X_{-1})$ and $\partial_{ij}  z^\eta$ are Lipschitz with respect to the $|\cdot|_{-1}$ norm with a Lipschitz constant independent of $\eta$. Moreover, note that by \eqref{eq:proof_second_derivative_bounded}  we have $|\partial_{ij}z^\eta| \leq C$ for a constant $C$ independent of $i,j$.
 
To conclude that $z^\eta \in \mathcal D$, we have to check that for every $k,h \in X_{-1}$ the limit
 \begin{equation}\label{eq:z_eta_second_order_derivative}
\lim_{t \to 0} \frac {1} {|t|} \langle D_{-1}z^\eta (x+tk)-D_{-1}z^\eta(x) ,h \rangle_{-1}
\end{equation}
exists and is uniformly continuous in $X_{-1}$.
Fix $k,h \in X_{-1}$, $n>0$ and set $k_n=(k^1,...k^n,0,0...), h_n=(h^1,...h^n,0,0...) \in X_{-1}$. 

Let $x\in X_{-1}$. We denote by $A_n(x)$ the operator from $X^n$ to $X^n$ given by the matrix $(\partial_{ij}z^\eta(x))_{i,j}$. We extend it to $X_{-1}$ by setting
$T_n(x)=P_nA_n(x)P_n$. We have $|T_n(x)|_{\mathcal L(X_{-1})}\leq C$ for all $n$. We have for $n>m$
\[
\begin{split}
|\langle T_n(x)h,k\rangle_{-1}&-\langle T_m(x)h,k\rangle_{-1}|=|\langle A_n(x)P_n h,P_nk\rangle_{-1}-\langle A_n(x)P_mh,P_mk\rangle_{-1}|
\\
&\leq
|\langle A_n(x)(P_n-P_m) h,P_nk\rangle_{-1}|+|\langle A_n(x)P_mh,(P_n-P_m)k\rangle_{-1}|\to 0\quad\mbox{as}\,\,m\to\infty.
\end{split}
\]
Therefore the sequence $\{\langle T_n(x)h,k\rangle_{-1}\}$ is a Cauchy sequence and thus $T_n(x)h$ converges weakly in $X_{-1}$ to an element of $X_{-1}$ which we denote by $\tilde D_{-1}^2z^\eta(x)h$. It is easy to see that such defined $\tilde D_{-1}^2z^\eta(x)$ is a linear, bounded and self-adjoint operator on $X_{-1}$.

We fix $\delta>0$ and let $n_0$ be such that for $n\geq n_0$ we have
\begin{equation}\label{eq:A(x)n}
\Bigg | \langle \tilde D_{-1}^2z^\eta(x)h,k\rangle_{-1}-\sum_{i,j=1}^n \partial_{ij}z^\eta(x)h_i k_j \Bigg | \leq \frac{\delta}{3}.
\end{equation}
We now estimate
\begin{small}
\begin{align}\label{eq:conv-n-a}
\Bigg | \frac {1} {|t|} \langle D_{-1}z^\eta (x+tk)&-D_{-1}z^\eta(x)  ,h \rangle_{-1}  - \sum_{i,j=1}^n \partial_{ij}z^\eta(x)h^i k^j  \Bigg | 
\\
&  \leq  \frac {1} {|t|} \Bigg  |  \langle D_{-1} z^\eta (x+tk)-D_{-1} z^\eta(x)  ,h \rangle_{-1} -  \langle D_{-1} z^\eta (x+tk_n)-D_{-1} z^\eta(x)  ,h_n \rangle_{-1}   \Bigg |\nonumber\\
  & \quad + \Bigg | \frac {1} {|t|} \langle D_{-1} z^\eta (x+tk_n)-D_{-1} z^\eta (x) ,h_n \rangle_{-1}   -  \sum_{i,j=1}^n \partial_{ij}z^\eta(x)h^i k^j  \Bigg | 
  \nonumber\\
 &  =:  I_1(x,n,t) +I_2(x,n,t).\nonumber
\end{align}
\end{small}
By the Lipschitzianity of $D_{-1} z$ we have
\begin{align*}
I_1(x,n,t) & \leq \frac {1} {|t|} \Big  |  \langle D_{-1} z^\eta (x+tk)-D_{-1} z^\eta(x)  ,h-h_n \rangle_{-1} \Big | + \frac {1} {|t|} \Big  |  \langle D_{-1} z^\eta (x+tk)-D_{-1} z^\eta(x+tk_n)  ,h_n \rangle_{-1} \Big |\\
& \leq C |k|_{-1} |h-h_n|_{-1}+ C |k-k_n|_{-1} |h|_{-1}.
\end{align*}
Regarding $I_2$, since $s\mapsto \langle D_{-1} z^\eta (x+sk_n),h_n \rangle_{-1}$ is smooth, by the mean value theorem there is $\tilde t \in \mathbb R,
|\tilde t| \leq |t|$ such that
\begin{align*}
I_2(x,n,t)& = \left | \sum_{i,j=1}^n \partial_{ij}z^\eta(x+\tilde t k^n)h^i k^j - \sum_{i,j=1}^n \partial_{ij}z^\eta(x)h^i k^j  \right |  \leq C_n |k|_{-1}^2 |h|_{-1} |t|,
\end{align*}
where the inequality follows by the Lipschitzianity of $\partial_{ij}z^\eta$ for every $i,j  \in \mathbb N.$

We can now find $\overline n>n_0$ such that $I_1(x,\overline n,t)<\delta /3$. Then we choose $|t|$ small enough such that $I_2(x,\overline n,t) <\delta/3$, so that for such $t$ we have
\[
\Bigg |\frac {1} {|t|} \langle D_{-1}z^\eta (x+tk)-D_{-1}z^\eta(x) ,h \rangle_{-1}- \langle \tilde D_{-1}^2z^\eta(x)h,k\rangle_{-1}\Bigg |<\delta.
\]
Therefore the limit in \eqref{eq:z_eta_second_order_derivative} exists and is equal to $\langle \tilde D_{-1}^2z^\eta(x)h,k\rangle_{-1}$. To prove that the latter expression is uniformly continuous in $X_{-1}$, we send $t\to 0$ in \eqref{eq:conv-n-a} to obtain
\[
\Bigg |\langle \tilde D_{-1}^2z^\eta(x)h,k\rangle_{-1}- \sum_{i,j=1}^n \partial_{ij}z^\eta(x)h^i k^j  
\Bigg |\leq C |k|_{-1} |h-h_n|_{-1}+ C |k-k_n|_{-1} |h|_{-1}
\]
which shows that $\langle \tilde D_{-1}^2z^\eta(x)h,k\rangle_{-1}$ is the uniform limit of uniformly continuous functions in $X_{-1}$.

\noindent
\textbf{Step 2:} Let $N \in\mathbb N$. We will prove that for every fixed $\bar x' \in X^{N,\perp}$ such that $(0,\bar x') \in X$, the function $z_{\bar x'}:=z(\cdot,\bar x')$ is a viscosity supersolution of a certain HJB equation on $\mathbb{R}^N$.

We recall that we use the notation $x=(x_N,x') \in X$ defined at the beginning of the proof.
Let $\bar x_N \in X^N$ be a minimum of $z_{\bar x'}+\Psi(\cdot)=z(\cdot, \bar x')+\Psi(\cdot)$ for $\Psi \in C^2(\mathbb R^N)$. (This means that if $x=\sum_i x^i e_i$ then $\Psi(x)=\Psi(x_N)=\Psi(x^1,...,x^N)$.)
Then, for every $x=(x_N,x') \in X$, using $z \in C^{1,1}(X_{-1})$ and Young's inequality, for any $\delta>0$
\begin{align*}
z(x_N,x')+\Psi(x_N) =& z(x_N,x')-z(x_N,\bar x')- \langle D_{-1}z(x_N,\bar x'),(0, x'-\bar x') \rangle_{-1} \\
& \quad +z(x_N,\bar x')+\Psi(x_N)+\langle D_{-1}z(x_N,\bar x'),(0, x'-\bar x') \rangle_{-1} \\
& \geq -\frac C 2 |(0,x'-\bar x')|_{-1}^2+z(x_N,\bar x')+\Psi(x_N)+\langle D_{-1}z(\bar x_N,\bar x'),(0, x'-\bar x') \rangle_{-1}\\
 & \quad  -C  |(x_N-\bar x_N,0)|_{-1}|(0,x'-\bar x')|_{-1} \\
& \geq z(\bar x_N,\bar x')+\Psi(\bar x_N)+\langle D_{-1}z(\bar x_N,\bar x'),(0, x'-\bar x') \rangle_{-1}\\
&\quad - \frac \delta 2 |(x_N-\bar x_N,0)|_{-1}^2- \frac C 2 \left (1+ \frac 1 \delta \right )|(0,x'-\bar x')|_{-1}^2.
\end{align*}
This implies that $z+\overline \Psi$ has a minimum at $\bar x=(\bar x_N,\bar x')$, where 
$$\overline \Psi(x):=\Psi(x_N) -\langle D_{-1}z(\bar x_N,\bar x'),(0, x'-\bar x') \rangle_{-1}+ \frac \delta 2 |(x_N-\bar x_N,0)|_{-1}^2+ \frac C 2 \left (1+ \frac 1 \delta \right )|(0,x'-\bar x')|_{-1}^2.$$
We notice that since $x=(x_N,x')\in X$, we can write
$$\overline \Psi(x)=\Psi(x_N) -\langle BD_{-1}z(\bar x),(0, x'-\bar x') \rangle_X+ \frac \delta 2 |(x_N-\bar x_N,0)|_{-1}^2+ \frac C 2 \left (1+ \frac 1 \delta \right )|(0,x'-\bar x')|_{-1}^2.$$
Hence, we have
$D\overline \Psi( x)=D\Psi(x_N)-Q_NBD_{-1}z(x) + \delta BP_N (x-\bar x) +C \left (1+ \frac 1 \delta \right ) BQ_N(x-\bar x),
$
so that
\begin{align*}
&D\overline \Psi(\bar x)=D\Psi(\bar x_N)-Q_NBD_{-1}z(\bar x), \quad D^2 \overline \Psi(\bar x)=D^2\Psi(\bar x_N)+\delta BP_N + C \left (1+ \frac 1 \delta \right )BQ_N.
\end{align*}
Here $D\overline \Psi(x), D^2\Psi(x), D\Psi( x_N), D^2\Psi( x_N)$ denote the standard Fr\'echet derivatives in $X$.
% but seen as a function in $X_{-1}$ so that we can compute them as follows:
%$$D\Psi( x_N)=(D_{x_N}\Psi(x_N),0), \quad D^2\Psi(x_N)=\begin{bmatrix}
%D_{x_N^2}^2\Psi( x_N) & 0 \\
%0 & 0
%\end{bmatrix},$$
Recalling that $D\Psi( x_N)=P_N D\Psi( x_N) ,D^2\Psi( x_N)=P_N D^2\Psi( x_N) P_N$, $D\Psi( x_N)$ and $D^2\Psi( x_N)$ are also first and second order derivatives of $\Psi$ as a function on $X^N \cong \mathbb R^N$, where $X^N$ is considered as a subspace of $X$.

Since $z$ is a viscosity supersolution, we now have
\begin{align*}
\rho z(\bar x)& + \langle \tilde A^*D \Psi(\bar x_N),\bar x \rangle_X -\langle \tilde A^*Q_NBD_{-1}z(\bar x),\bar x\rangle_X \\   & +  H \Big (\bar x,-D \Psi(\bar x_N) +Q_NBD_{-1}z(\bar x),-D^2\Psi(\bar x_N)-\delta BP_N - C \left (1+ \frac 1 \delta \right )BQ_N \Big ) \   \geq - \gamma,
\end{align*}
i.e.
\begin{align*}
\rho z(\bar x)& + \langle \tilde A^*D \Psi(\bar x_N),\bar x \rangle_X-\langle \tilde  A^* Q_NBD_{-1}z(\bar x),\bar x\rangle_X\\
&  +  \tilde H \Big (\bar x,-D_{x_0} \Psi(\bar x_N) +(Q_NBD_{-1}z(\bar x))_0,-D_{x_0^2}^2\Psi(\bar x_N)-\delta (BP_N)_{00} - C \left (1+ \frac 1 \delta \right )(BQ_N)_{00}  \Big) \geq  - \gamma.
\end{align*}
By \eqref{eq:A*B^1/2} we can write $\langle \tilde  A^* Q_NBD_{-1}z(\bar x),\bar x\rangle_X=\langle \tilde  A^*B^{1/2} Q_NB^{1/2}D_{-1}z(\bar x), \bar x\rangle_X$ so that
\begin{align*}
\rho z(\bar x)& + \langle \tilde A^*D \Psi(\bar x_N),\bar x \rangle_X\\
&  \quad +  \tilde H \Big (\bar x,-D_{x_0} \Psi(\bar x_N) +(Q_NBD_{-1}z(\bar x))_0,-D_{x_0^2}^2\Psi(\bar x_N)-\delta (BP_N)_{00} - C \left (1+ \frac 1 \delta \right )(BQ_N)_{00}  \Big)\\
& \geq  - \gamma - C \left  |Q_NB^{1/2}D_{-1}z(\bar x)\right | |\bar x|.
\end{align*}
By definition of $\tilde H$ we now have
\begin{small}
\begin{align*}
\rho z(\bar x)&+ \langle \tilde A^*D \Psi(\bar x_N),\bar x \rangle_X     +  \tilde H(\bar x,-D_{x_0} \Psi(\bar x_N) ,-D_{x_0^2}^2\Psi(\bar x_N) )\\
& \geq   - \gamma - C \left  |Q_NB^{1/2}D_{-1}z(\bar x)\right | |\bar x|  + \bar x_0 \cdot (Q_NBD_{-1}z(\bar x))_0  \\
&\quad  - \sup_{u\in U} \Bigg \{ - b_0\left (\bar x_0,\int_{-d}^0 a_1(\xi)\bar x_1(\xi)\,d\xi ,u \right) \cdot (Q_NBD_{-1}z(\bar x))_0  \Bigg \} \\
&\quad  +  \frac{1}{2} \tr \left [ \sigma_0 \left (\bar x_0,\int_{-d}^0 a_2(\xi)\bar x_1(\xi)\,d\xi \right)\sigma_0 \left ( \bar x_0,\int_{-d}^0 a_2(\xi)\bar x_1(\xi)\,d\xi\right)^T  \left (-\delta BP_N -C \left (1+ \frac 1 \delta \right )BQ_N \right)_{00}  \right]\\
& \geq   - \gamma - C \left  |Q_NB^{1/2}D_{-1}z(\bar x)\right | |\bar x|  + \bar x_0 \cdot (Q_NBD_{-1}z(\bar x))_0  \\
& \quad  -  C \left (1+ \frac 1 \delta \right ) \tr \left [ \sigma_0 \left (\bar x_0,\int_{-d}^0 a_2(\xi)\bar x_1(\xi)\,d\xi \right)\sigma_0 \left ( \bar x_0,\int_{-d}^0 a_2(\xi)\bar x_1(\xi)\,d\xi\right)^T  (BQ_N)_{00}  \right] - C_R \delta  \\
&=: -\gamma +f_N(\bar x)- C_R \delta.
\end{align*}\end{small}
We point out that for every fixed $\delta>0$, the functions $f_N$ are continuous, locally uniformly (in $N$) bounded in $X$ and for every $x\in X, f_N(x)\to 0$ as $N\to \infty$.

Thus we have shown that for every fixed $x' \in X^{N,\perp}$ such that $(0, x') \in X$ the $x_N$-function $z_{x'}:=z(\cdot,x')$ is a viscosity supersolution of
\begin{align*}
\rho z_{x'}(x_N  )&- \langle \tilde A^*D z_{x'}(x_N ),(x_N ,x') \rangle_X     +  \tilde H( (x_N ,x') ,D_{x_0} z_{x'}(x_N ) ,D_{x_0^2}^2z_{x'}(x_N ) ) = -\gamma - C_R \delta +f_N(x_N , x'), 
\end{align*}
that is
\begin{align}\label{eq:proof_eq_z(,y)}
\rho z_{x'}(\cdot  )&- \langle \tilde A^*D z_{x'}(x_N ),(x_N ,x') \rangle_X     +  \tilde H( (x_N ,x') ,D_{x_0} z_{x'}(x_N ) )  -\frac{1}{2} \tr  \Bigg [ \sigma_0 \left ((x_N ,x')_0,\int_{-d}^0 a_2(\xi)(x_N ,x')_1(\xi)\,d\xi \right) \nonumber \\
& \qquad\quad
\times  \sigma_0 \left ( (x_N ,x')_0,\int_{-d}^0 a_2(\xi)(x_N ,x')_1(\xi)\,d\xi\right)^T  D_{x_0^2}^{2}  z_{x'}(x_N )    \Bigg]  = -\gamma - C_R \delta +f_N(x_N , x'). 
\end{align}
%where as before $Dz_{x'}=(D_{x_N}z,0)$ and $D_{x_0}z_{x'}=(D_{x_N}z,0)_0$ is the $x_0$-component of  $Dz_{x'}$.\\
Since this is an equation on $\mathbb R^N$, $z_{x'} \in C^{1,1}(\mathbb R^N)$, $\tilde A^*D z_{x'}$ is continuous and all the terms above are well defined, the left-hand side of \eqref{eq:proof_eq_z(,y)} is greater than or equal to the right-hand side for a.e. $x_N \in \mathbb R^N$.

\noindent
\textbf{Step 3:} Let $R>0$ and let $x' \in X_{N}^\perp$ be such that $(0,x') \in X, |x'|_X\leq R$ and consider the $x_N$-function $z_{x'}^{\eta,k}:=z^{\eta,k}(\cdot,x')$, $k\leq N$. We will show that this function satisfies a perturbed HJB equation.

Applying the $X_k$-convolution to both sides of \eqref{eq:proof_eq_z(,y)} we have
\begin{align*}
\rho z_{x'}^{\eta,k}(x_N)&+ \int_{\mathbb R^k} \Bigg \{ - \langle \tilde A^*D z_{x'}(\hat x^1, ...,\hat x^k, x^{k+1},...,x^N),(\hat x^1, ...,\hat x^k, x^{k+1},...,x^N,x') \rangle_X \nonumber \\
&   \quad   +  \tilde H( (\hat x^1, ...,\hat x^k, x^{k+1},...,x^N,x') ,D_{x_0} z_{x'}(\hat x^1, ...,\hat x^k, x^{k+1},...,x^N) ) \nonumber \\
& \quad  -1/2 \tr  \Big [ \sigma_0 \left ((\hat x^1, ...,\hat x^k, x^{k+1},...,x^N,x')_0,\int_{-d}^0 a_2(\xi)(\hat x^1, ...,\hat x^k, x^{k+1},...,x^N,x')_1(\xi)\,d\xi \right) \nonumber \\
& \quad  \times  \sigma_0 \left ( (\hat x^1, ...,\hat x^k, x^{k+1},...,x^N,x')_0,\int_{-d}^0 a_2(\xi)(\hat x^1, ...,\hat x^k, x^{k+1},...,x^N,x')_1(\xi)\,d\xi\right)^* \nonumber \\
&\quad  \times   D_{x_0^2}^{2}  z_{x'}(\hat x^1, ...,\hat x^k, x^{k+1},...,x^N)    \Big ] \Bigg \} \prod_{i=1}^k \rho_{\eta_i}(x^i-\hat x^i)    d\hat x^1... d\hat x^k \nonumber \\
&   \geq -\gamma -C_R \delta   +\int_{\mathbb R^k} f_N(\hat x^1, ...,\hat x^k, x^{k+1},...,x^N, x') \prod_{i=1}^k \rho_{\eta_i}(x^i-\hat x^i) d\hat x^1... d\hat x^k 
\end{align*}
for a.e. $x_N \in \mathbb R^N$.
Note that, since ${\rm supp}( \rho_{\eta_i} )\subset [-\eta_i,\eta_i]$,  we are effectively integrating only with respect to $(\hat {x}^1,...\hat x^k) \in X_k \cong \mathbb R^k$ such that $|\hat x^i-x^i| \leq \eta_i=\eta \sqrt{\lambda_i}/2^{i}$. Then for such $\hat x_i$, recalling that $e_i=f_i/ \sqrt{\lambda_i}$ with $\{f_i\}$ being an orthonormal basis of $X$ and by setting $\hat x=(\hat x^1, ...,\hat x^k, x^{k+1},...,x^N,x')$, $ x=(x_N,x')=(x^1, ...,x^k, x^{k+1},...,x^N,x'),$ we have $|\hat x-x|_X^2=\sum_{i \in \mathbb N} |\hat x^i-x^i|^2 /\lambda_i   \leq \eta^2.$
Moreover, since $|\tilde A^*D z_{x'}(x)|_X=|\tilde A^*B^{1/2}B^{1/2}P_ND_{-1} z(x)|_X\leq C|D_{-1} z(x)|_{-1}\leq C$. It then follows that $|\langle \tilde A^*D z_{x'}(\hat x^1, ...,\hat x^k, x^{k+1},...,x^N),x  \rangle_X - \langle \tilde A^*D z_{x'}(\hat x^1, ...,\hat x^k, x^{k+1},...,x^N),\hat x \rangle_X|\leq C_R \eta$ for some constant $C_R$ if in addition $|x_N|_X\leq R$. Arguing similarly for the other terms, by Lemma \ref{lemma:properties_H} and the Lipschitzianity of $Dz$ as a map from $X_{-1}$ to $X$, we have that there is a modulus $\omega_R$ such that
\begin{align*}
\rho z_{x'}^{\eta,k}(x_N)&+ \int_{\mathbb R^k} \Bigg \{ - \langle \tilde A^*D z_{x'}(\hat x^1, ...,\hat x^k, x^{k+1},...,x^N),(x_N,x') \rangle_{\red{X}}  +  \tilde H( (x_N,x') ,D_{x_0} z_{x'}(x_N) ) \nonumber \\
&\quad   -1/2 \tr  \Big [ \sigma_0 \left ((x_N,x')_0,\int_{-d}^0 a_2(\xi)(x_N,x')_1(\xi)\,d\xi \right) \sigma_0 \left ( (x_N,x')_0,\int_{-d}^0 a_2(\xi)(x_N,x')_1(\xi)\,d\xi\right)^T \nonumber \\
& \quad  \times   D_{x_0^2}^{2}  z_{x'}(\hat x^1, ...,\hat x^k, x^{k+1},...,x^N)    \Big ] \Bigg \} \prod_{i=1}^k \rho_{\eta_i}(x^i-\hat x^i)    d\hat x^1... d\hat x^k \nonumber \\
&   \geq -\gamma -C_R \delta   -\int_{\mathbb R^k} f_N(\hat x^1, ...,\hat x^k, x^{k+1},...,x^N, x') \prod_{i=1}^k \rho_{\eta_i}(x^i-\hat x^i) d\hat x^1... d\hat x^k - \omega_R( \eta)
\end{align*}
for a.e. $x_N \in \mathbb R^N$ such that $|x_N|\leq R$.

We now notice that by  \eqref{eq:proof_derivative_convolution_cauchy}, for every $\in X$
\begin{align*}
|D z^{\eta,k}(x)-D z(x)|& =|BD_{-1} z^{\eta,k}(x)-BD_{-1} z(x)|\leq C |D_{-1} z^{\eta,k}(x)-D_{-1} z(x)|_{-1}\\
& \leq C \sum_{i=2}^{k} |D_{-1}z^{\eta,i}(x)-D_{-1}z^{\eta,i-1}(x)|_{-1} + C_R  |D_{-1}z^{\eta,1}(x)-D_{-1}z(x)|_{-1}
\leq C \eta.
\end{align*}
Using this inequality and the fact that the convolutions of derivatives are the derivatives of the convolutions, we now have
\begin{align*}
\rho & z_{x'}^{\eta,k}(x_N)-  \langle \tilde A^*D z^{\eta,k}_{x'}(x_N),(x_N,x') \rangle_X  +  \tilde H( (x_N,x') ,D_{x_0} z^{\eta,k}_{x'}(x_N) ) \nonumber \\
& \quad  -1/2 \tr  \Big [ \sigma_0 \left ((x_N,x')_0,\int_{-d}^0 a_2(\xi)(x_N,x')_1(\xi)\,d\xi \right) \sigma_0 \left ( (x_N,x')_0,\int_{-d}^0 a_2(\xi)(x_N,x')_1(\xi)\,d\xi\right)^T  D_{x_0^2}^{2}  z_{x'}^{\eta,k}(x_N)    \Big ]  \nonumber \\
&   \geq -\gamma -C_R \delta   -\int_{\mathbb R^k} f_N(\hat x^1, ...,\hat x^k, x^{k+1},...,x^N, x') \prod_{i=1}^k \rho_{\eta_i}(x^i-\hat x^i) d\hat x^1... d\hat x^k - \omega_R (\eta)
\end{align*}
for some modulus $\omega_R$. Therefore we proved that if $x' \in X_{N}^\perp$ is such that $(0,x') \in X, |x'|_X\leq R$, then
\begin{align}\label{eq:proof_eq_z_eta_k(,y)}
\rho z_{x'}^{\eta,k}(x_N )&- \langle \tilde A^*Dz_{x'}^{\eta,k}(x_N),(x_N,x') \rangle_X     +  \tilde H( (x_N,x'),D_{x_0} z_{x'}^{\eta,k}(x_N) ,D^{2}_{x_0^2} z_{x'}^{\eta,k}(x_N) ) \nonumber \\
&  \geq  -\gamma -C_R \delta - \omega_R(\eta)  -\int_{\mathbb R^k} f_N(\hat x^1, ...,\hat x^k, x^{k+1},...,x^N, x') \prod_{i=1}^k \rho_{\eta_i}(x^i-\hat x^i) d\hat x^1... d\hat x^k 
\end{align}
for a.e. $x_N \in \mathbb R^N$ such that $|x_N|\leq R$.
Observe that since $z_{x'}^{\eta,k} \in C^{1,1}(\mathbb R^N)$, it is well known that the inequality holds in the viscosity sense in $\mathbb R^N$ (see e.g. \cite[Theorem I.2]{lions_1983}), that is $z_{x'}^{\eta,k}$ is a viscosity supersolution of \eqref{eq:proof_eq_z_eta_k(,y)}.

\noindent
\textbf{Step 4:} We let $N \to \infty$ and prove that $z^{\eta,k}$ is a viscosity supersolution of a perturbed HJB equation on $X$.\\
Let $\bar x=(\bar x_N,\bar x') \in X$ be a minimum of $z^{\eta,k}+\Psi$ for a test function $\Psi=\phi+g$ (defined on $X$)  with $\phi \in \Phi$ and $g \in \mathcal G$. Then $\bar x_N$ is a minimum for $z_{\bar x'}^{\eta,k}(\cdot)+\Psi(\cdot,\bar x')=z^{\eta,k}(\cdot,\bar x')+\Psi(\cdot,\bar x')$. 

Since $z^{\eta,k}\in C^{1,1}(X_{-1}), D\Psi(\bar x)=Dz^{\eta,k}(\bar x)=BD_{-1}z^{\eta,k}(\bar x)\in D(\tilde A^*)$. Moreover, we have 
\[
D\Psi(\bar x_N,\bar x')=P_ND\Psi(\bar x)=P_NDz^{\eta,k}(\bar x)=BP_ND_{-1}z^{\eta,k}(\bar x)\quad\mbox{and}\quad D^2\Psi(\bar x_N,\bar x')=P_ND^2\Psi(\bar x)P_N.
\] 
Thus, as $|P_ND_{-1}z^{\eta,k}(\bar x)|_{-1}\to 0$,
\begin{equation}\label{eq:convatilde}
\tilde A^*D\Psi(\bar x_N,\bar x')=\tilde A^*P_ND\Psi(\bar x)\tilde A^*B^{\frac{1}{2}}B^{\frac{1}{2}}P_ND_{-1}z^{\eta,k}(\bar x)\to \tilde A^*BD_{-1}z^{\eta,k}(\bar x)=\tilde A^*D\Psi(\bar x).
\end{equation}
Since
$z_{\bar x'}^{\eta,k}$ is a viscosity supersolution of \eqref{eq:proof_eq_z_eta_k(,y)}, we have
\begin{align*}
\rho z^{\eta,k}(\bar x_N, \bar x' )&+ \langle \tilde A^*D \Psi (\bar x_N,\bar x'),\bar x \rangle_X     +  \tilde H( \bar x,-D_{x_0} \Psi (\bar x_N,\bar x') ,-D_{x_0^2}^2 \Psi(\bar x_N,\bar x') ) \nonumber \\
&  \geq  -\gamma -C_R \delta - \omega_R(\eta)  -\int_{\mathbb R^k} f_N(\hat x_1, ...,\hat x_k, \bar x_{k+1},...,\bar x_N, \bar x') \prod_{i=1}^k \rho_{\eta_i}(\bar x_i-\hat x_i) d\hat x_1... d\hat x_k.
\end{align*}
We recall that $D_{x_0}\Psi(\bar x_N,\bar x')=P_{x_0}P_ND\Psi(\bar x)$ and $D_{x_0^2}^2 \Psi(\bar x_N,\bar x')=P_{x_0}P_ND^2\Psi(\bar x)P_NP_{x_0}$.
Therefore, letting $N \to \infty$, using \eqref{eq:convatilde}, the fact that $f_N(x) \to 0$ and the dominated convergence theorem, we obtain
\begin{align*}
\rho z^{\eta,k}(\bar x )&+ \langle \tilde A^*D \Psi (\bar x),\bar x \rangle_X     +  \tilde H( \bar x,-D_{x_0} \Psi (\bar x) ,-D_{x_0^2}^2 \Psi(\bar x) )   \geq  -\gamma -C_R \delta - \omega_R(\eta)  .
\end{align*}
We can now let $\delta \to 0$ to get
\begin{align*}
\rho z^{\eta,k}(\bar x )&+ \langle \tilde A^*D \Psi (\bar x),\bar x \rangle_X     +  \tilde H( \bar x,-D_{x_0} \Psi (\bar x) ,-D_{x_0^2}^2 \Psi(\bar x) )   \geq  -\gamma - \omega_R(\eta).  
\end{align*}
In particular, we proved that for every $k$, the function $z^{\eta,k}$ is a viscosity supersolution in $X$ of
\begin{align}\label{eq:proof_z_eta_k_viscosity_supersolution}
\rho z^{\eta,k}(x )&- \langle \tilde A^*D z^{\eta,k}, x \rangle_X     +  \tilde H( x,D_{x_0} z^{\eta,k} (x) ,D_{x_0^2}^2 z^{\eta,k}(x) )  = -\gamma - \omega_R(\eta)  .
\end{align}
\textbf{Step 5:} We use consistency of viscosity solutions to obtain that $z^\eta$ is a viscosity supersolution of a perturbed HJB equation.
Since for every $k$, the function $z^{\eta,k}$ is a viscosity supersolution of \eqref{eq:proof_z_eta_k_viscosity_supersolution} and, by  \eqref{eq:proof_convergence_convolution}, $z^{\eta,k}$ converges uniformly to $z^\eta$, it follows from consistency of viscosity solutions, \cite[Theorem 3.41]{fgs_book}, that $z^{\eta}$ is a viscosity supersolution of
\begin{align*}
\rho z^{\eta}(x )&- \langle \tilde A^*D z^{\eta}, x \rangle_X     +  \tilde H( x,D_{x_0} z^{\eta} (x) ,D_{x_0^2}^2 z^{\eta}(x) )  = -\gamma - \omega_R(\eta).
\end{align*}
This completes the proof of the lemma.
\end{proof}
Applying Lemma \ref{lemma:z_eta} to $z=\tilde V_\epsilon$, we obtain that for every $\eta>0$ there exist $\tilde V_\epsilon^\eta \in \mathcal D$ such that 
\begin{equation}\label{eq:convergence_convolution_w_eps_eta}
|\tilde V_\epsilon-\tilde V_\epsilon^\eta|\leq C_\epsilon \eta ,  \quad |D_{-1}\tilde V_\epsilon-D_{-1}\tilde V_\epsilon^\eta|_{-1}\leq C_\epsilon \eta
\end{equation}
for some $C_\epsilon>0$ (independent of $\eta$) and such that $\tilde V_\epsilon^\eta$ is a viscosity supersolution of
\begin{align}\label{eq:w_eps_eta_supersolution}
\rho \tilde V_\epsilon^\eta(x)-\langle  \tilde A^*D\tilde V_\epsilon^\eta(x),x \rangle_X +\tilde H \left (x,D_{x_0}\tilde V_\epsilon^\eta(x),D^2_{x_0^2}\tilde V_\epsilon^\eta(x) \right) = -\gamma(\epsilon) - \omega_{R,\epsilon}(\eta), \quad   \forall x \in B_R.
\end{align}
Since $\tilde V_\epsilon^\eta \in \mathcal D$ all terms appearing in \eqref{eq:w_eps_eta_supersolution} are well defined. Thus we will prove that $\tilde V_\epsilon^\eta$ satisfies \eqref{eq:w_eps_eta_supersolution} pointwise as inequality.

\begin{lemma}\label{lemma:w_eta_eps_classical_supersol_perturbed_HJB}
Let Assumptions \ref{hp:state}, \ref{hp:cost}, \ref{hp:discount} and \ref{hp:convexity} hold.
For every $R>0$ we have
\begin{align*}
\rho \tilde V_\epsilon^\eta(x)-\langle  \tilde A^*D\tilde V_\epsilon^\eta(x),x \rangle_X +\tilde H \left (x,D_{x_0}\tilde V_\epsilon^\eta(x), D^2_{x_0^2}\tilde V_\epsilon^\eta(x) \right) \geq -\gamma(\epsilon) - \omega_{R,\epsilon}(\eta) \quad   \forall x \in B_R.
\end{align*}
\end{lemma}
\begin{proof}
We go back to the standard notation from Subsection \ref{subsec:infinite_dimensional_framework}, that is $x=(x_0,x_1)$ means $x_0\in\mathbb R^n,x_1\in L^2$.
Fix $\bar x=(\bar x_0, \bar x_1) \in B_R \subset X$. Since $\tilde V_\epsilon^\eta \in C^{1,1}(X_{-1})$, using Young's inequality, for every $0 <\delta < 1$ we have
\begin{align*}
\tilde V_\epsilon^\eta(x_0,x_1)-\tilde V_\epsilon^\eta(x_0,\bar x_1) & \geq \langle D_{-1}\tilde V_\epsilon^\eta(x_0,\bar x_1),(0,x_1-\bar x_1) \rangle_{-1} - C |(0,x_1-\bar x_1)|_{-1}^2 \\
& \geq \langle D_{-1} \tilde V_\epsilon^\eta(\bar x),(0,x_1-\bar x_1) \rangle_{-1} -C |(x_0-\bar x_0,0)|_{-1}|(0,x_1-\bar x_1)|_{-1} - C |(0,x_1-\bar x_1)|_{-1}^2 \\
& \geq \langle D_{-1} \tilde V_\epsilon^\eta(\bar x),(0,x_1-\bar x_1) \rangle_{-1} -\delta |(x_0-\bar x_0,0)|_{-1}^2-\frac{C}{\delta}|(0,x_1-\bar x_1)|_{-1}^2 \\
& = \langle D\tilde V_\epsilon^\eta(\bar x),(0,x_1-\bar x_1) \rangle_X -\delta |(x_0-\bar x_0,0)|_{-1}^2-\frac{C}{\delta}|(0,x_1-\bar x_1)|_{-1}^2\\
& = \langle D_{x_1}\tilde V_\epsilon^\eta(\bar x),x_1-\bar x_1 \rangle_{L^2} -\delta |(x_0-\bar x_0,0)|_{-1}^2 -\frac{C}{\delta}|(0,x_1-\bar x_1)|_{-1}^2
\end{align*}
for every $x=(x_0,x_1) \in X$. Since $\tilde w_\epsilon^\eta \in \mathcal D$, by Lemma \ref{lemma:D^2_x_0^2phi_in_D}, using the second order Taylor expansion with respect to $x_0$, we have
\begin{align*} 
\tilde V_\epsilon^\eta(x_0,\bar x_1)- \tilde V_\epsilon^\eta(\bar x_0,\bar x_1) \geq D_{x_0}\tilde V_\epsilon^\eta(\bar x)\cdot(x_0-\bar x_0)+\frac{1}{2}D^2_{x_0^2}\tilde V_\epsilon^\eta(\bar x)(x_0-\bar x_0) \cdot (x_0-\bar x_0)  - \delta |(x_0-\bar x_0,0)|_{-1}^2
\end{align*}
when $|(x_0-\bar x_0,0)|_{-1}^2$ is small (and we used that the norms $|(\cdot,0)|$ and $|(\cdot,0)|_{-1}$ are equivalent on $\mathbb{R}^n$).
Adding the last two inequalities we now get
\begin{align*}
\tilde V_\epsilon^\eta(x)&-\tilde V_\epsilon^\eta(\bar x)  \geq D_{x_0}\tilde V_\epsilon^\eta(\bar x)\cdot(x_0-\bar x_0) +\langle D_{x_1}\tilde V_\epsilon^\eta(\bar x),x_1-\bar x_1 \rangle_{L^2} +\frac{1}{2}D^2_{x_0^2}\tilde V_\epsilon^\eta(\bar x)(x_0-\bar x_0)\cdot (x_0-\bar x_0) \\ 
& \qquad\qquad\qquad\qquad\qquad\qquad  -2\delta |(x_0-\bar x_0,0)|_{-1}^2-\frac{C}{\delta}|(0,x_1-\bar x_1)|_{-1}^2  \\
& =  \langle D\tilde V_\epsilon^\eta(\bar x), x-\bar x \rangle_X +\frac{1}{2}D^2_{x_0^2}\tilde V_\epsilon^\eta(\bar x)(x_0-\bar x_0)\cdot (x_0-\bar x_0)
-2\delta |(x_0-\bar x_0,0)|_{-1}^2-\frac{C}{\delta}|(0,x_1-\bar x_1)|_{-1}^2.
\end{align*}
Defining  $\phi $ by
\[
\phi(x)=-\tilde V_\epsilon^\eta(\bar x) - \langle D\tilde V_\epsilon^\eta(\bar x), x-\bar x \rangle_X -\frac{1}{2}D^2_{x^2_0}\tilde V_\epsilon^\eta(\bar x)(x_0-\bar x_0)\cdot (x_0-\bar x_0) +2\delta |(x_0-\bar x_0,0)|_{-1}^2+\frac{C}{\delta}|(0,x_1-\bar x_1)|_{-1}^2
\]
we have $\tilde V_\epsilon^\eta(x) +\phi(x) \geq 0$ when $|(x_0-\bar x_0,0)|_{-1}^2$ is small so that $\bar x$ is a local minimum for $\tilde V_\epsilon^\eta+\phi$.
We notice that $\phi \in \Phi$ since $D\tilde V_\epsilon^\eta(\bar x) \in D(\tilde A^*)$
(recall that $\tilde V_\epsilon^\eta \in C^{1}(X_{-1})$).
We also observe that $D |(x_0-\bar x_0,0)|_{-1}^2=2B(x_0-\bar x_0,0)$, $\frac{\partial^2}{\partial^2 x_0^2} |(x_0-\bar x_0,0)|_{-1}^2=2B_{00}$, so that 
\[
-D \phi(\bar x)=D\tilde V_\epsilon^\eta(\bar x), \quad -D^2_{x^2_0}\phi (\bar x)=D^2_{x_0^2}\tilde V_\epsilon^\eta(\bar x)-4\delta B_{00},
\]
where $B_{00}$ was defined in \eqref{eq:representation_B}.\\
Since $\bar x$ is a local minimum of $\tilde V_\epsilon^\eta+\phi$ and $w_\epsilon$ is a viscosity supersolution of \eqref{eq:w_eps_eta_supersolution}, we thus have
\begin{align*}
 \rho \tilde V_\epsilon^\eta(\bar x)& + \langle \bar x, \tilde A^*D \phi (\bar x)  \rangle_X + \tilde H \left ( \bar x, -D_{x_0}\phi (\bar x) , -D^2_{x^2_0}\phi(\bar x)  \right) \geq -\gamma(\epsilon)-\omega_{R,\epsilon}(\eta),
 \end{align*}
 from which by \eqref{eq:hamiltonian_synthesis}, we obtain
 \begin{align*}
\rho \tilde V_\epsilon^\eta(\bar x) & -\langle \bar x, \tilde A^*D\tilde V_\epsilon^\eta(\bar x) \rangle_X+\tilde H \left (\bar x,D_{x_0}\tilde V_\epsilon^\eta(\bar x),D^2_{x^2_0}\tilde V_\epsilon^\eta(\bar x) \right)\\ & \geq -\gamma(\epsilon) -\omega_{R,\epsilon}(\eta) -   4\delta\tr \left [ \sigma_0 \left ( \bar x_0,\int_{-d}^0 a_1(\xi)\bar x_1(\xi)\,d\xi \right)\sigma_0 \left ( \bar x_0,\int_{-d}^0 a_1(\xi)\bar x_1(\xi)\,d\xi\right)^T  B_{00} \right].
\end{align*}
The result follows by letting $\delta\to 0$.
\end{proof}

\section{Verification Theorem and Optimal Synthesis}\label{sec:verthm-optfeed}
In this section we prove a Verification Theorem and construct an optimal feedback control for our problem.
We start by proving the following proposition.
\begin{prop}\label{prop:V_bar_equal_V} 
Let Assumptions \ref{hp:state}, \ref{hp:cost}, \ref{hp:discount} and \ref{hp:convexity} hold. There exists $\overline \rho\geq \rho_0$ such that if $\rho> \overline \rho$ then
$\overline V(x)=V(x) \quad \forall x \in X.$
\end{prop}
\begin{proof}
Since $ \mathcal{\overline U } \supset \mathcal U$ we immediately have $\overline V \leq V$. Hence we are left to prove 
\begin{equation}\label{eq:tilde_V_leq}
V \leq \overline  V.
\end{equation}
We divide the proof of this fact into several parts.

%(i) We approximate $b,\sigma, L$ by their inf-sup convolutions $b^\xi, \sigma^\xi, L^\xi \in C^{1,1}(X_{-1})$ and we consider approximating control %problems with these drift, diffusion and cost functions. Denoting by $V^\xi$ the value function of the approximating problems, we show that
%\begin{equation}\label{eq:uniform_convergence_V_xi}
%V^\xi \xrightarrow{\xi \to 0} V \quad \textit{uniformly in } X.
%\end{equation} 

(i) For every $\xi >0$ we define $L^\xi\colon X \times U \to \mathbb R, \tilde b_0^\xi \colon X \times U \to \mathbb R^{n}, \sigma_0^\xi\colon X \to M^{n \times q }$ by
\begin{align*}
&L^\xi(x,u)=\sup_{z \in X} \inf_{y \in X} \left[ L(y,u) + \frac{1}{2 \xi} |z-y|_{-1}^2- \frac 1 \xi |z-x|_{-1}^2 \right]\\
&(\tilde b_{0}^\xi)_i(x,u)=\sup_{z \in X} \inf_{y \in X} \left[ \tilde (b_{0})_i(y,u) + \frac{1}{2 \xi} |z-y|_{-1}^2- \frac 1 \xi |z-x|_{-1}^2 \right], \quad i=1,...,n,\\
&(\sigma_{0}^\xi)_{ij}(x)=\sup_{z \in X} \inf_{y \in X} \left[ (\sigma_{0})_{ij}(y) + \frac{1}{2 \xi} |z-y|_{-1}^2- \frac 1 \xi |z-x|_{-1}^2 \right] \quad i=1,...,n, j=1,...,q.
\end{align*}
%where for the vector valued function $\tilde b_0 $ and the matrix-valued function $\sigma_0$ the inf-sup operations are meant component-wise. 
The functions $L(\cdot,u),(\tilde b_0)_i(\cdot,u), (\sigma_0)_{ij}$ are Lipschitz in the $|\cdot|_{-1}$ norm. By \cite{lasry} we have $ L^\xi(\cdot,u), (\tilde b_{0}^\xi)_i(\cdot,u), (\sigma_{0}^\xi)_{ij} \in C^{1,1}(X_{-1})$ (where by density of $X \subset X_{-1}$ we have extended $ L^\xi, \tilde b_{0}^\xi, \sigma_{0}^\xi$ to $X_{-1}$). Moreover $L^\xi(\cdot,u),(\tilde b_0^\xi )_i (\cdot,u) , (\sigma_0^\xi)_{ij}$ are Lipschitz in the $|\cdot|_{-1}$ norm (with  Lipschitz constants  independent of $\xi,u$). In fact the Lipschitz constants of $L^\xi(\cdot,u),(\tilde b_0^\xi )_i (\cdot,u) , (\sigma_0^\xi)_{ij}$ are the same as those of $L(\cdot,u),(\tilde b_0)_i(\cdot,u), (\sigma_0)_{ij}$.  Finally
\begin{equation}\label{eq:unif_convercenge_L_xi,b_xi,sigma_xi}
L^\xi \xrightarrow{\xi \to 0} L, \quad b_{0}^\xi \xrightarrow{\xi \to 0} \tilde b_{0} \quad \textit{uniformly in } X \times U, \quad  \sigma_{0}^\xi \xrightarrow{\xi \to 0} \sigma_{0}\quad \textit{uniformly in } X.
\end{equation} 
Define
$$\tilde b^\xi(x,u)=[\tilde b_0^\xi(x,u),0]^T, \quad \sigma^\xi(x)w=\begin{bmatrix}
\sigma_0^\xi(x) w,
0
\end{bmatrix}^T \quad \forall x \in X, \ w \in \mathbb{R}^q. $$  
Now, for every $x \in X, u(\cdot) \in \mathcal U$ we consider approximating optimal control problems (in the reference probability space formulation) with state equations
\begin{equation}\label{eq:state_eq_approximated_Y_xi}
dY^\xi(t)=\tilde AY^\xi(t) dt + \tilde b^\xi(Y^\xi(t),u(t))dt+\sigma^\xi(Y^\xi(t))dW_t, \quad Y^\xi(0)=x,
\end{equation}
and cost functionals and value functions
$$J^\xi(x;u(\cdot))=\mathbb E \int_0^\infty e^{-\rho t} L^\xi(Y^\xi(t),u(t)) dt,\quad V^\xi(x)=\inf_{u(\cdot) \in \mathcal{U}} J^\xi(x;u(\cdot)).$$
Moreover, $V^\xi$ satisfies \eqref{Va-1} and Theorem \ref{th:existence_uniqueness_viscosity_infinite} holds for $V^\xi$, that is $V^\xi$ is the unique viscosity solution of the HJB equation \eqref{eq:HJB} with $b,\sigma, L$ replaced by $b^\xi, \sigma^\xi, L^\xi$.

(ii) Denoting by $Y(t)$ the solution of \eqref{eq:abstract_dissipative_operator}, we prove that there exist a modulus of continuity $\omega$  and a constant $\lambda >0$  (both independent of $x,u(\cdot)$) such that
\begin{equation}\label{eq:E_norm_-1_Y_xi-Y}
\mathbb E |Y^\xi(t)-Y(t)|_{-1}^2 \leq t \omega(\xi)e^{\lambda t} \quad \forall t \geq 0, \xi >0.
\end{equation}
Indeed, since for $\xi>0$ 
$$d(Y^\xi-Y)(t)=\tilde A(Y^\xi-Y)(t) dt +\left  [\tilde b^\xi(Y^\xi(t),u(t))-\tilde b(Y(t),u(t)) \right] dt+\left [ \sigma^\xi(Y^\xi(t))-\sigma(Y(t)) \right ]dW(t), \quad (Y^\xi-Y)(0)=x,$$
by Ito's formula \cite[Proposition 1.165]{fgs_book} we have
\begin{small}
\begin{align*}
\mathbb E |Y^\xi(t)-Y(t)|_{-1}^2 &  = 2 \int_0^t \mathbb E  \langle   \tilde A^*B (Y^\xi(s)-Y(s)), Y^\xi(s)-Y(s) \rangle_X +  \mathbb E  \langle B(Y^\xi(s)-Y(s)), \tilde b^\xi(Y^\xi(s),u(s))-\tilde b(Y(s),u(s))  \rangle_X ds\\
& \quad + \int_0^t \mathbb E \operatorname{Tr}\left[ \left ( \sigma^\xi(Y^\xi(s))-\tilde \sigma (Y(s)) \right ) \left ( \sigma^\xi(Y^\xi(s))-\tilde \sigma (Y(s)) \right )^* B  \right ]ds \\
& \leq 2 \int_0^t\mathbb E  \langle B(Y^\xi(s)-Y(s)), \tilde b^\xi(Y^\xi(s),u(s))-\tilde b(Y(s),u(s))  \rangle_X ds\\
& \quad + \int_0^t \mathbb E \operatorname{Tr}\left[ \left ( \sigma^\xi(Y^\xi(s),u(s))-\tilde \sigma (Y(s),u(s)) \right ) \left ( \sigma^\xi(Y^\xi(s))-\tilde \sigma (Y(s)) \right )^* B  \right ]ds,
\end{align*}
\end{small}
where the  inequality follows using the weak $B$-condition with $C_0=0$ (i.e. Proposition \ref{prop:weak_b}). Consider the first term on the right-hand-side. Using \eqref{eq:b_B_property}, \eqref{eq:unif_convercenge_L_xi,b_xi,sigma_xi} and the uniform convergence of $\tilde b_0^\xi$, we have
\begin{small}
\begin{align*}
\int_0^t & \mathbb E  \langle B(Y^\xi(s)-Y(s)), \tilde b^\xi(Y^\xi(s),u(s))-\tilde b(Y(s),u(s))  \rangle_X ds\\ &  = \int_0^t\mathbb E  \langle B(Y^\xi(s)-Y(s)), \tilde b^\xi(Y^\xi(s),u(s))-\tilde b(Y^\xi(s),u(s))  \rangle_X ds +\int_0^t\mathbb E  \langle B(Y^\xi(s)-Y(s)), \tilde b(Y^\xi(s),u(s))-\tilde b(Y(s),u(s))  \rangle_X ds\\
& \leq C \int_0^t \left [   \mathbb E |Y^\xi(s)-Y(s)|_{-1}^2 + \left  |\tilde b_0^\xi(Y^\xi(s),u(s))-\tilde b_0(Y^\xi(s),u(s))  \right |^2   \right]ds + C \int_0^t  \mathbb E |Y^\xi(s)-Y(s)|_{-1}^2ds\\
& \leq C \int_0^t   \mathbb E |Y^\xi(s)-Y(s)|_{-1}^2ds + t\omega (\xi)
\end{align*}
\end{small}
for some $C>0$ and a modulus of continuity $\omega$ independent of $x,u(\cdot).$
An analogous inequality is obtained similarly for the second term on the right-hand side. Hence we have
\begin{align*}
\mathbb E |Y^\xi(t)-Y(t)|_{-1}^2  \leq C  \int_0^t   \mathbb E |Y^\xi(s)-Y(s)|_{-1}^2ds + t\omega (\xi)
\end{align*}
and by Gronwall's lemma we obtain \eqref{eq:E_norm_-1_Y_xi-Y} for some $\lambda>0$.\\
(iii) We can now prove that
\begin{equation}\label{eq:uniform_convergence_V_xi}
V^\xi \xrightarrow{\xi \to 0} V \quad \textit{uniformly in } X.
\end{equation} 
Indeed, fix $\rho> \max(\rho_0,\frac{\lambda}{2})$ and let $x \in X, u(\cdot) \in \mathcal U$.  By \eqref{eq:unif_convercenge_L_xi,b_xi,sigma_xi}, \eqref{eq:L_unif_cont_norm_B} and \eqref{eq:E_norm_-1_Y_xi-Y} we have
\begin{small}
\begin{align*}
|J^\xi(x;u(\cdot))-J(x;u(\cdot))|
& \leq \int_0^\infty  e^{-\rho t} \mathbb E |L^\xi(Y^\xi(t),u(t))-L(Y^\xi(t),u(t))| dt + \int_0^\infty  e^{-\rho t} \mathbb E |L(Y^\xi(t),u(t))-L(Y(t),u(t))| dt\\
&  \leq \int_0^\infty  e^{-\rho t} \omega_1(\xi) dt +C \int_0^\infty  e^{-\rho t} \mathbb E |Y^\xi(t)-Y(t)|_{-1}dt \leq  \omega_1(\xi)  +\omega_2(\xi) \int_0^\infty  t e^{-(\rho -\frac{\lambda}{2})t} dt  \leq \omega(\xi).
\end{align*}
\end{small}
for some modulus of continuity $\omega$ independent of $x,u(\cdot).$ This implies \eqref{eq:uniform_convergence_V_xi}.

(iv) We prove that there exists $\bar \rho\geq \max(\rho_0,\frac{\lambda}{2})$ such that for every $\rho> \bar \rho$,  $V^\xi$ is $|\cdot|_{-1}$-semiconcave for every $\xi >0$.

Fix $\xi >0$. It is enough to show that there exists $\bar \rho>0$ such that $\forall \rho> \bar\rho$ $ J^\xi(\cdot, u(\cdot))$   is $|\cdot|_{-1}$-semiconcave  with a semiconcavity constant independent of $u(\cdot)$.
Indeed let $x, \bar x \in X$, $u(\cdot) \in \mathcal U$, $\lambda \in [0,1]$. Denote by $Y(t), \bar Y(t)$ the solutions of \eqref{eq:state_eq_approximated_Y_xi} with initial state $x, \bar x$ respectively and control $u(\cdot).$ Moreover, set $x_\lambda=\lambda x + (1-\lambda) \bar x$ and let $Y_\lambda(t)$ be the solution of \eqref{eq:state_eq_approximated_Y_xi} with initial state $x_\lambda$ and control $u(\cdot)$. Finally, set $Y^\lambda(t)=\lambda Y(t)+ (1-\lambda) \bar Y(t)$. Then, by the $|\cdot|_{-1}$-Lipschitzianity (uniformly in $u, \xi$) and the $|\cdot|_{-1}$-semiconcavity of $L^\xi(\cdot,u)$ (with a semiconcavity constant uniform in $u$), we have
\begin{small}
\begin{align*}
\lambda J^\xi(x; u(\cdot))&+ (1-\lambda)J^\xi(\bar x;u(\cdot)) - J^\xi(x_\lambda;u(\cdot))\\
&=\int_0^\infty e^{-\rho t}\mathbb E \left [ \lambda L^\xi(Y(t),u(t))+(1-\lambda ) L^\xi(\bar Y(t),u(t)) - L^\xi(Y^\lambda(t), u(t))\right ] dt\\
&\quad + \int_0^\infty e^{-\rho t}\mathbb E \left [L^\xi(Y^\lambda(t), u(t))-  L^\xi(Y_\lambda(t), u(t)\right ] dt\\
& \leq C_\xi\lambda (1-\lambda) \int_0^\infty e^{-\rho t}  \mathbb E |Y(t)-\bar Y(t)|^2_{-1}dt + C\int_0^\infty e^{-\rho t} \mathbb E |Y^\lambda(t) - Y_\lambda(t))|_{-1} dt.
\end{align*}
\end{small}
for some $C_\xi,C>0$ (independent of $u(\cdot)$).\\
By \cite[Lemmas 5.3, 5.8]{defeo_swiech_wessels} there exist constants $C,\tilde \rho>0$ (both independent of $\xi,u(\cdot)$), and $C_\xi$ (independent of $u(\cdot)$) such that 
$$\mathbb E |Y(t)-\bar Y(t)|^2_{-1} \leq C  e^{\tilde \rho t } |\bar x- x|_{-1}^2 , \quad  \mathbb E |Y^\lambda(t) - Y_\lambda(t))|_{-1} \leq C_\xi  \lambda (1-\lambda)  e^{\tilde \rho t } |\bar x- x|_{-1}^2.$$
By inserting these inequalities in the previous one, for every $\rho> \bar \rho:=\tilde \rho=\max(\tilde\rho,\rho_0,\frac{\lambda}{2})$ 
\begin{align*}
\lambda J^\xi(x, u(\cdot))&+ (1-\lambda)J^\xi(\bar x, u(\cdot)) - J^\xi(x_\lambda,u(\cdot)) \leq C_\xi \lambda (1-\lambda) |\bar x- x|_{-1}^2,
\end{align*}
which yields the claim. Hence $(iv)$ follows.

(v)
We now show \eqref{eq:tilde_V_leq}. Following Section \ref{sec:approximation via inf-convolutions} we extend $V^\xi$ to the function $\tilde V^\xi$ on $X_{-1}$, which then satisfies \eqref{Va-1} on $X_{-1}$ and is semiconcave in $X_{-1}$.
%Indeed, by standard results we have that $V^\xi$ is semiconcave with respect to $|\cdot|_{-1}$ for every $\xi >0$. In particular the proof is similar to the one of \cite[Theorem 3.9]{defeo_swiech_wessels}. We only point out that in our case, as we have an infinite horizon problem, we need  $\rho > 2C+|B|_{\mathcal L(X)}C^2$. Indeed, looking at \cite[Eq. (47)]{defeo_swiech_wessels} and \cite[Eq. (41), proof of Lemma 3.7]{defeo_swiech_wessels}, we need to replace the use of \cite[Lemma 3.2]{defeo_swiech_wessels}   with \cite[Lemma 3.20]{fgs_book}. In \cite[Lemma 3.20]{fgs_book}, from the proof, we have $C(T)=e^{2C_0+2C+|B|_{\mathcal L(X)}C^2}$, where $C_0$ is the constant from the weak-B condition and for us $C_0=0$ and $C$ is the constant from \eqref{eq:b_lip}, \eqref{eq:G_lipschitz_norm_B}. Hence we must have $\rho > 2C+|B|_{\mathcal L(X)}C^2$ in order for the argument to work.
We now fix $\xi>0$ and for every $\epsilon>0$ consider the sup-convolution $(\tilde V^\xi)^\epsilon$ of $\tilde V^\xi$, that is 
%(here we keep using the notation for the inf-convolution instead of the usual one for sup-convolutions for consistency with Section 4), i.e 
$$(\tilde V^\xi)^{\epsilon}(x):=\sup_{y \in X_{-1}}\left [ \tilde V^\xi(y)-\frac{1}{2\epsilon}|x-y|^2_{-1}\right].$$
We denote the restriction of $(\tilde V^\xi)^{\epsilon}$ to $X$ by $(V^\xi)^{\epsilon}$.
Similarly to Lemma \ref{lem:c11}, if $\epsilon$ is small enough, $(\tilde V^\xi)^{\epsilon}\in C^{1,1}(X_{-1})$ and
\begin{equation}
(V^\xi)^{\epsilon} \xrightarrow{\epsilon \to 0} V^\xi \quad \textit{uniformly.}
\end{equation}
By repeating the procedure from Sections \ref{sec:approximation via inf-convolutions} and \ref{sec:lions_approx} (Proposition \ref{prop:inf_conv_subsolution_perturbed}, Lemma \ref{lemma:z_eta}, Lemma \ref{lemma:w_eta_eps_classical_supersol_perturbed_HJB} with  adjustments since we are now dealing with sup-convolutions and viscosity subsolutions, but the proofs are the same), we obtain that for every $\eta>0$ there exist $(\tilde V^\xi)^\epsilon_\eta \in \mathcal D$ such that 
\begin{equation*}
|(\tilde V^\xi)^\epsilon-(\tilde V^\xi)^\epsilon_\eta|\leq C^\xi_\epsilon \eta ,  \quad |D_{-1}(\tilde V^\xi)^\epsilon-D_{-1}(\tilde V^\xi)^\epsilon_\eta|_{-1}\leq C^\xi_\epsilon \eta
\end{equation*}
for some $C^\xi_\epsilon>0$ (independent of $\eta$) and such that for every $R>0$ 
\begin{align}\label{eq:V_xi_eta_eps_subsolution_perturbed_HJB}
\rho (\tilde V^\xi)^\epsilon_\eta(x)-\langle  \tilde A^*D(\tilde V^\xi)^\epsilon_\eta(x),x \rangle_X +\tilde H \left (x,D_{x_0}(\tilde V^\xi)^\epsilon_\eta(x), D^2_{x_0^2}(\tilde V^\xi)^\epsilon_\eta(x) \right) \leq \gamma^\xi(\epsilon) + \omega^\xi_{R,\epsilon}(\eta) \quad   \forall x \in B_R.
\end{align}

We now fix $x \in X,u(\cdot) \in \overline U$ and denote by $Y(t)$ the solution of the state equation with initial state $x$ and control $u(\cdot)$. Let $R,t>0$ and define
 $
\chi^R = \inf \{s \in[0, t]:|Y(s)|_X>R\}.
$  Since $ (\tilde V^\xi)_\eta^\epsilon \in \mathcal D$, we can apply Lemma \ref{lemma:ito} to $\phi=(\tilde V^\xi)_\eta^\epsilon$ to get
  \begin{small}
\begin{align*}
(\tilde V^\xi)_\eta^\epsilon( x)
& =\mathbb E \left[ e^{-\rho (t\wedge \chi^R)}  (\tilde V^\xi)_\eta^\epsilon(Y(t \wedge \chi^R )) \right]
 +\mathbb E  \int_0^{t \wedge \chi^R} e^{-\rho s}  \Big [ \rho (\tilde V^\xi)_\eta^\epsilon(Y(s)) -    \langle  Y(s), \tilde A^* D (\tilde V^\xi)_\eta^\epsilon(Y(s))  \rangle_X \\
& \quad  - \tilde  b_0(Y(s),u(s)) \cdot D_{x_0} (\tilde V^\xi)_\eta^\epsilon(Y(s))      
- \frac 1 2  \tr \left ( \sigma_0(Y(s))\sigma_0(Y(s))^T D^2_{x_0^2}(\tilde V^\xi)_\eta^\epsilon(Y(s)) \right ) \Big ]   ds\\
& =\mathbb E \left[ e^{-\rho (t\wedge \chi^R)}  (\tilde V^\xi)_\eta^\epsilon(Y(t \wedge \chi^R )) \right] +\mathbb E  \int_0^{t \wedge \chi^R} e^{-\rho s} l(Y(s),u(s)) ds\\
 &\,\, +\mathbb E  \int_0^{t \wedge \chi^R} e^{-\rho s}  \Big [ \rho (\tilde V^\xi)_\eta^\epsilon(Y(s))  -    \langle  Y(s), \tilde A^* D (\tilde V^\xi)_\eta^\epsilon(Y(s))  \rangle_X
  - \tilde  b_0(Y(s),u(s)) \cdot D_{x_0} (\tilde V^\xi)_\eta^\epsilon(Y(s))        -l(Y(s),u(s))  \\
&\quad \quad \quad \quad \quad \quad - \frac 1 2  \tr \left ( \sigma_0(Y(s))\sigma_0(Y(s))^T D^2_{x_0^2}(\tilde V^\xi)_\eta^\epsilon(Y(s)) \right ) \Big ]   ds\\
& \leq \mathbb E \left[ e^{-\rho (t\wedge \chi^R)}  (\tilde V^\xi)_\eta^\epsilon(Y(t \wedge \chi^R )) \right] +\mathbb E  \int_0^{t \wedge \chi^R} e^{-\rho s} l(Y(s),u(s)) ds\\
 &\,\, +\mathbb E  \int_0^{t \wedge \chi^R} e^{-\rho s}  \Big [ \rho (\tilde V^\xi)_\eta^\epsilon(Y(s))  -    \langle  Y(s), \tilde A^* D (\tilde V^\xi)_\eta^\epsilon(Y(s))  \rangle_X
 +\tilde H\left (Y(s), D_{x_0} (\tilde V^\xi)_\eta^\epsilon(Y(s))  , D^2_{x_0^2}(\tilde V^\xi)_\eta^\epsilon(Y(s)) \right )   ds\\
 & \leq \mathbb E \left[ e^{-\rho (t\wedge \chi^R)}  (\tilde V^\xi)_\eta^\epsilon(Y(t \wedge \chi^R )) \right] +\mathbb E  \int_0^{t \wedge \chi^R} e^{-\rho s} l(Y(s),u(s)) ds+\gamma^\xi(\epsilon) + \omega^\xi_{R,\epsilon}(\eta),
\end{align*}
\end{small}
where the first and the second inequalities follow by definition of $\tilde H$ and \eqref{eq:V_xi_eta_eps_subsolution_perturbed_HJB} respectively. Letting (in this order) $\eta \to 0, \epsilon \to 0$, $R \to \infty$ and $t \to \infty$ we obtain $V^\xi(x) \leq J(x;u(\cdot))$. Finally, letting $\xi \to 0$, we get $V(x)
\leq  J(x;u(\cdot))$ and hence
\[  
V( x)\leq \inf_{u \in \overline U} J(x;u(\cdot)) = \overline V(x) \quad \forall x \in X.
\]
This concludes the proof of \eqref{eq:tilde_V_leq} so that $\overline V=V$. 
\end{proof}

We now strengthen Assumption \ref{hp:discount} by the following assumption.
\begin{hypothesis}\label{hp:discount2}
$\rho >\bar \rho$, where $\bar\rho$ is from Proposition \ref{prop:V_bar_equal_V}.
\end{hypothesis}
 \begin{theorem}[Verification]\label{th:verification} Let Assumptions \ref{hp:state}, \ref{hp:cost},  \ref{hp:uniform_ellipticity}, \ref{hp:convexity}, \ref{hp:discount2}  hold.
Let $x \in X$ and $u^*(\cdot) \in \mathcal {\overline U}$ be an admissible control. Denote by $Y^*(s)$ the solution of \eqref{eq:abstract_dissipative_operator} with $u(\cdot)=u^*(\cdot)$. Assume that
  $$u^*(s)\in {\rm argmax}_{u\in U} \Bigg \{ - b_0\left ( Y_0^*(s),\int_{-d}^0 a_1(\xi)Y_1^*(s)(\xi)\,d\xi ,u \right) \cdot D_{x_0}V(Y^*(s))  - l(Y_0^*(s),u)  \Bigg \},$$
$\mathbb P -$a.s. for a.e. $s \geq 0.$ Then, the pair $(Y^*(\cdot),u^*(\cdot))$ is optimal.
 \end{theorem}
 \begin{proof}
Let $R,t>0$ and define
 $
\chi^R := \inf \{s \in[0, t]:|Y^*(s)|_X>R\}.
$
 Since $\tilde V^\eta_\epsilon \in \mathcal D$ we can apply Lemma \ref{lemma:ito} to $\phi=\tilde V^\eta_\epsilon$ to get
  \begin{small}
\begin{align*}
\tilde V^\eta_\epsilon( x)
& =\mathbb E \left[ e^{-\rho (t\wedge \chi^R)}  \tilde V^\eta_\epsilon(Y^*(t \wedge \chi^R )) \right]
 +\mathbb E  \int_0^{t \wedge \chi^R} e^{-\rho s}  \Big [ \rho \tilde V^\eta_\epsilon(Y^*(s)) -    \langle  Y^*(s), \tilde A^* D \tilde V^\eta_\epsilon(Y^*(s))  \rangle_X \\
& \quad  - \tilde  b_0(Y^*(s),u^*(s)) \cdot D_{x_0} \tilde V^\eta_\epsilon(Y^*(s))      
- \frac 1 2  \tr \left ( \sigma_0(Y^*(s))\sigma_0(Y^*(s))^T D^2_{x_0^2}\tilde V^\eta_\epsilon(Y^*(s)) \right ) \Big ]   ds\\
& =\mathbb E \left[ e^{-\rho (t\wedge \chi^R)}  \tilde V^\eta_\epsilon(Y^*(t \wedge \chi^R )) \right] +\mathbb E  \int_0^{t \wedge \chi^R} e^{-\rho s} l(Y^*(s),u^*(s)) ds\\
 &\quad +\mathbb E  \int_0^{t \wedge \chi^R} e^{-\rho s}  \Big [ \rho \tilde V^\eta_\epsilon(Y^*(s))  -    \langle  Y^*(s), \tilde A^* D \tilde V^\eta_\epsilon(Y^*(s))  \rangle_X
  - \tilde  b_0(Y^*(s),u^*(s)) \cdot D_{x_0} \tilde V^\eta_\epsilon(Y^*(s))        -l(Y^*(s),u^*(s))  \\
&\quad \quad \quad \quad \quad \quad - \frac 1 2  \tr \left ( \sigma_0(Y^*(s))\sigma_0(Y^*(s))^T D^2_{x_0^2}\tilde V^\eta_\epsilon(Y^*(s)) \right ) \Big ]   ds.
\end{align*}
\end{small}
Note that, by \eqref{eq:convergence_inf_convolution}, \eqref{eq:convergence_convolution_w_eps_eta}, we have $|D_{x_0} \tilde V^\eta_\epsilon(Y^*(s))       - D_{x_0} V(Y^*(s))  |_X \leq |D_{x_0} \tilde V^\eta_\epsilon(Y^*(s))       - D_{x_0} \tilde V_\epsilon(Y^*(s))     |_X + |D_{x_0} \tilde V_\epsilon(Y^*(s))       - D_{x_0} V(Y^*(s))  |_X \leq C_\epsilon \eta + \omega_R(\epsilon) $. Note also that thanks to $\tau_R$, we have $|\tilde  b_0(Y^*(s),u^*(s))| \leq C_R $. From these facts,  the  definition of  $u^*(\cdot)$, \eqref{eq:Hamiltonian_local_lip} and Lemma \ref{lemma:w_eta_eps_classical_supersol_perturbed_HJB}, we have
  \begin{small}
\begin{align*}
\tilde V^\eta_\epsilon( x) 
& \geq \mathbb E \left[ e^{-\rho (t\wedge \chi^R)}  \tilde V^\eta_\epsilon(Y^*(t \wedge \chi^R )) \right] +\mathbb E  \int_0^{t \wedge \chi^R} e^{-\rho s} l(Y^*(s),u^*(s)) ds\\
 &\quad +\mathbb E  \int_0^{t \wedge \chi^R} e^{-\rho s}  \Big [ \rho \tilde V^\eta_\epsilon(Y^*(s))  -    \langle  Y^*(s), \tilde A^* D \tilde V^\eta_\epsilon(Y^*(s))  \rangle_X
  - \tilde  b_0(Y^*(s),u^*(s)) \cdot D_{x_0} V(Y^*(s))      -l(Y^*(s),u^*(s))  \\
&\quad \quad \quad \quad \quad \quad - \frac 1 2  \tr \left ( \sigma_0(Y^*(s),u^*(s))\sigma_0(Y^*(s))^T D^2_{x_0^2}\tilde V^\eta_\epsilon(Y^*(s)) \right )-C_{R,\epsilon} \eta - \omega_R(\epsilon) \Big ]   ds \\
& = \mathbb E \left[ e^{-\rho (t\wedge \chi^R)}  \tilde V^\eta_\epsilon(Y^*(t \wedge \chi^R )) \right] +\mathbb E  \int_0^{t \wedge \chi^R} e^{-\rho s} l(Y^*(s),u^*(s)) ds\\
 &\quad +\mathbb E  \int_0^{t \wedge \chi^R} e^{-\rho s}  \Big [ \rho \tilde V^\eta_\epsilon(Y^*(s))  -    \langle  Y^*(s), \tilde A^* D \tilde V^\eta_\epsilon(Y^*(s))\rangle_X  +\tilde H\left (Y^*(s),D_{x_0}V(Y^*(s) ) \right) \\
 &\quad \quad \quad \quad \quad \quad - \frac 1 2  \tr \left ( \sigma_0(Y^*(s))\sigma_0(Y^*(s))^T D^2_{x_0^2}\tilde V^\eta_\epsilon(Y^*(s)) \right )-C_{R,\epsilon} \eta - \omega_R(\epsilon) \Big ]   ds \\
 & \geq  \mathbb E \left[ e^{-\rho (t\wedge \chi^R)}  \tilde V^\eta_\epsilon(Y^*(t \wedge \chi^R )) \right] +\mathbb E  \int_0^{t \wedge \chi^R} e^{-\rho s} l(Y^*(s),u^*(s)) ds\\
 &\quad +\mathbb E  \int_0^{t \wedge \chi^R} e^{-\rho s}  \Big [ \rho \tilde V^\eta_\epsilon(Y^*(s))  -    \langle  Y^*(s), \tilde A^* D \tilde V^\eta_\epsilon(Y^*(s))\rangle_X  +\tilde H\left (Y^*(s),D_{x_0}\tilde V^\eta_\epsilon(Y^*(s) ) \right) \\&\quad \quad \quad \quad \quad \quad - \frac 1 2  \tr \left ( \sigma_0(Y^*(s))\sigma_0(Y^*(s))^T D^2_{x_0^2}\tilde V^\eta_\epsilon(Y^*(s)) \right )-C_{R,\epsilon} \eta - \omega_R(\epsilon) \Big ]   ds \\
  & \geq  \mathbb E \left[ e^{-\rho (t\wedge \chi^R)}  \tilde V^\eta_\epsilon(Y^*(t \wedge \chi^R )) \right]+\mathbb E  \int_0^{t \wedge \chi^R} e^{-\rho s} l(Y^*(s),u^*(s)) ds  -\omega_{R,\epsilon}( \eta) - \omega_R(\epsilon),
\end{align*}
\end{small}
where the constants $C_{R,\epsilon}$ and moduli $\omega_R(\epsilon)$ might have changed from line to line.

Letting first $\eta \to 0$ and then $\epsilon \to 0$,  by \eqref{eq:convergence_inf_convolution}, \eqref{eq:convergence_convolution_w_eps_eta}, we get
\begin{align*}
V(x)
\geq  \mathbb E \left[ e^{-\rho (t\wedge \chi^R)}  V(Y^*(t \wedge \chi^R )) \right]+\mathbb E  \int_0^{t\wedge \chi^R}  e^{-\rho s} l(Y^*(s),u^*(s)) ds.
\end{align*}
We now send first $R \to \infty$ and then $t \to \infty$. Recalling Proposition \ref{prop:V_bar_equal_V}, we have 
\begin{align*}
\overline V(x)=V( x)
\geq \mathbb E  \int_0^\infty e^{-\rho s} l(Y^*(s),u^*(s)) ds= J(x;u^*(\cdot)),
\end{align*}
from which we obtain the optimality of $u^*(\cdot).$
 \end{proof}

We now  construct an optimal feedback control. To do this, we make an additional assumption about the Hamiltonian. 
\begin{hypothesis}\label{hp:feedback_map_well-defined}
We assume that the supremum in \eqref{eq:hamiltonian_synthesis_no_sigma} is a maximum, i.e. 
\begin{align*} \tilde H \left (x,p_0 \right)&= - x_0 \cdot p_0 + \max_{u\in U} \Bigg \{ - b_0\left ( x_0,\int_{-d}^0 a_1(\xi)x_1(\xi)\,d\xi ,u \right) \cdot p_0
 - l(x_0,u)  \Bigg \}. 
 \end{align*}
  \end{hypothesis}
We define the multivalued map $\Psi \colon X \to  \mathcal  P (U)$ by
  \begin{align}\label{eq:multivalued_function_PSI}
   \Psi(x):&={\rm argmax}_{u \in U}   \Bigg \{ - b_0\left ( x_0,\int_{-d}^0 a_1(\xi)x_1(\xi)\,d\xi ,u \right) \cdot D_{x_0}V(x) 
 - l(x_0,u)  \Bigg \}   \neq \emptyset, \quad \forall x \in X.
  \end{align}
   \begin{corollary}\label{cor:closed_loop} Let the assumptions of Theorem \ref{th:verification} be satisfied. In addition, let Assumption \ref{hp:feedback_map_well-defined} hold.
 Assume that $\Psi$ has a measurable selection $\psi$ such that the closed loop equation
  \begin{equation}
\label{eq:closed_loop}
dY(t) = [\tilde A Y(t)+\tilde b(Y(t),\psi(Y(t)) ] dt + \sigma(Y(t)) \,dW(t), \quad Y(0) = x \in X,
\end{equation}
admits a weak mild solution $Y^\psi(t)$ (e.g. see \cite[Definition 1.121]{fgs_book}) in some generalized reference probability space $\tau$. If we set $u^\psi(\cdot) :=  \psi(Y^\psi (\cdot)),$ then the pair $(u^\psi(\cdot),Y^\psi(\cdot))$ is optimal.
 \end{corollary}
\begin{proof}
First note that $u^\psi(\cdot) \in \mathcal{ \overline U}$. Hence $Y^\psi(\cdot)$ is the unique mild solution to \eqref{eq:abstract_dissipative_operator} (in the strong probabilistic sense now in the generalized reference probability space $\tau$). 
Now note that by construction, $u^\psi$ satisfies for every $s \geq 0$
  $$u^\psi(Y^\psi(s))\in {\rm argmax}_{u\in U} \Bigg \{ - b_0\left ( Y^\psi_0(s),\int_{-d}^0 a_1(\xi)Y_1^\psi(s)(\xi)\,d\xi ,u \right) \cdot D_{x_0}V(Y^\psi(s)) 
 - l(Y_0^\psi(s),u)  \Bigg \} .$$
 Hence, by the Verification Theorem \ref{th:verification}, we conclude that the pair $(u^\psi(\cdot),Y^\psi(\cdot))$ is optimal. 
\end{proof}

 We remark that we had to relax the class of admissible controls since the feedback map $\psi$ is not regular enough to guarantee existence of mild solutions of the closed loop equation in a reference probability space.
 \begin{remark}\label{rem:comparison_feedback_with_mild_and_bsde}
As mentioned in the introduction, using the dynamic programming approach in Hilbert spaces (remembering, as recalled there, that other approaches can also be used to tackle problems with delays), verification theorems and optimal feedback laws for stochastic optimal control problems with delays in the state have been studied using mild solutions in $L^2$ spaces or BSDE. Indeed, in the former a linear structure of the state equation and appropriate conditions ensuring the existence of an invariant measure are assumed, e.g. see \cite[Section 5.6]{fgs_book}. In the latter some regularity of the coefficients is assumed (e.g. differentiability and $\sigma_0$ having a bounded inverse), e.g. see \cite[Section 6.6]{fgs_book}. 
 Moreover, both approaches can handle pointwise delays when these appear in a linear way in the state equation (e.g. when the state equation is of the form $dy(t)=a y(t-d)+ \cdots$  for some $a \in M^{n \times n}$). 
We also refer to  \cite{masiero_2022}, for an approach using partial smoothing of the stochastic semigroup for a special class of problems. Our approach here, based on viscosity solutions of HJB equations in Hilbert spaces, allows us to work under different assumptions than those of the other approaches. Indeed, 
 apart from standard conditions, we assume the $|\cdot |_{-1}$-semiconvexity of the value function $V$ and a local non-degeneracy of $\sigma_0$.
 \end{remark}

  \section{Application to stochastic optimal advertising}\label{sec:example}
The following problem is taken from \cite[Section 7]{defeo_federico_swiech}, see also the seminal work \cite{gozzi_marinelli_2004} in the finite horizon case. We remark that in \cite[Section 4]{gozzi_marinelli_2004}, using mild solutions in $L^2$ spaces, optimal feedback laws were constructed, under appropriate assumptions guaranteeing the existence of an invariant measure for the stochastic semigroup.  Moreover, the problem can be treated using BSDE (e.g. see \cite[Chapter 6]{fgs_book}), under differentiability assumptions on the coefficients (including the running cost).
 As remarked in Remark \ref{rem:comparison_feedback_with_mild_and_bsde}, our approach here, based on viscosity solutions of HJB equations in Hilbert spaces, allows us to relax these assumptions. Here, instead, we assume the $|\cdot |_{-1}$-semiconvexity of the value function $V$ and a local non-degeneracy of $\sigma_0$. However, in the stochastic framework of the advertising problem (i.e. $\sigma_0 \in \mathbb R, \sigma_0 \neq 0$), these conditions are naturally satisfied by the problem. On the other hand, contrary to other approaches, we cannot treat point-wise delays. We finally refer to the introduction for different methods that, under suitable conditions, could be used to tackle the problem.

The model for the dynamics of the stock of advertising goodwill $y(t)$ of a product is given by the following controlled $1$-dimensional SDDE ($n=h=k=1$)
\begin{equation*}
\begin{cases}
dy(t) = \left[ a_0 y(t)+\int_{-d}^0 a_1(\xi)y(t+\xi)\,d\xi +c_0 u(t) \right]
         dt   + \sigma_0 \, dW(t),\\
y(0)=x_0, \quad y(\xi)=x_1(\xi)\; \quad \forall \xi\in[-d,0),
\end{cases}
\end{equation*}
where  $d>0$, the control process $u(s)$ models the intensity of advertising spending and $W$ is a real-valued Brownian motion, and
\begin{enumerate}[(i)]
\item $a_0 \leq 0$ is a constant factor of image deterioration in absence of advertising;
\item   $c_0 > 0$ is a constant advertising effectiveness factor; 
\item $a_1 \leq 0$ is a given deterministic function satisfying the assumptions used in the previous sections which represents the distribution of the forgetting time; 
\item $\sigma_0>0$ represents  the uncertainty in the model;
\item $x_0 \in \mathbb R$ is the level of goodwill at the beginning of the advertising campaign;
\item $x_1 \in L^2([-d,0];\mathbb R)$ is the history of the goodwill level.
\end{enumerate}
We use the same setup of the stochastic optimal control problem as in Section \ref{sec:formul}. The control set  is 
$U= [0,\bar u]$
for some $\bar u>0$.
The optimization problem is 
$$
\inf_{u(\cdot) \in \tilde{\mathcal U}}\mathbb{E} \left[\int_0^\infty e^{-\rho s} l(y(s),u(s)) d s\right],
$$
where  $\rho >0$ is a discount factor, $l(x,u)=h(u)-g(x)$, with a continuous and strictly convex cost function $h \colon [0,\bar u] \rightarrow \mathbb R$ and a continuous and concave utility function $g \colon \mathbb R \rightarrow \mathbb R$, which satisfy Assumption \ref{hp:cost}. Moreover we assume that $g$ is strictly increasing; $h \in C^0([0,\bar u]) \cap C^1([0,\bar u))$; $h$ is strictly increasing and $h(0)=0$; $h'(0)=0$, $\lim_{u \to \bar u} h'(u)=\infty.$

Setting
$$b_0 \left ( x_0,\int_{-d}^0 a_1(\xi)x_1(\xi)\,d\xi ,u\right):=a_0 x_0+\int_{-d}^0 a_1(\xi)x_1(\xi)\,d\xi +c_0 u,$$ 
we are then  in the setting of Section \ref{sec:formul}. Therefore, using the infinite dimensional framework of Section \ref{sec:prelim}, we can use Theorem \ref{th:existence_uniqueness_viscosity_infinite} to characterize the value function $V$ as the unique viscosity solution to \eqref{eq:HJB}, and Theorem \ref{th:C1alpha} to obtain partial regularity of $V$. Moreover $V$ is convex as the assumptions of Example \ref{ex:convex2} are satisfied.

We want to construct an optimal feedback for the optimization problem. Note that
$$
\tilde b(x,u)=\begin{bmatrix} b_0 \left ( x_0,\int_{-d}^0 a_1(\xi)x_1(\xi)\,d\xi ,u\right) +x_0 \\ 0   \end{bmatrix} = \begin{bmatrix} (a_0+1) x_0+\int_{-d}^0 a_1(\xi)x_1(\xi)\,d\xi \\ 0   \end{bmatrix}+ \begin{bmatrix}  c_0 u\\ 0   \end{bmatrix}
=: D x + E u,
$$
for every $x=(x_0,x_1) \in X, u \in [0,\overline u].$ Hence \eqref{eq:abstract_dissipative_operator} becomes
 \begin{equation}\label{eq:abstract_eq_example}
d Y(t) = [(\tilde A+D) Y(t)+E u(t) ] dt + \sigma \,dW(t), \quad Y(0) = x \in X.
\end{equation}
Notice that, since $h$ is strictly convex, then its continuous derivative $h' \colon [0,\bar u) \to [0,\infty)$ is strictly increasing. Hence it is invertible ant its inverse $(h')^{-1} \colon [0,\infty) \to [0,\bar u)$ is continuous. Then  $\Psi$ defined by \eqref{eq:multivalued_function_PSI} becomes
\begin{small}
\begin{align*}
   \Psi(x)&=argmax_{u \in U}   \Bigg \{ - \left (a_0 x_0+\int_{-d}^0 a_1(\xi)x_1(\xi)\,d\xi +c_0 u \right) D_{x_0}V(x)
 - h(u)+g(x)  \Bigg \} \\
 & =  argmax_{u \in U}   \Bigg \{ - c_0 u D_{x_0}V(x) 
 - h(u)  \Bigg \}= \begin{cases}
 & (h')^{-1} \left( c_0 D_{x_0}V(x) \right) \quad  \textit{if } D_{x_0}V(x)<0,\\
  &0 \quad  \quad  \quad  \quad  \quad  \quad  \quad  \quad  \quad   \textit{if } D_{x_0}V(x) \geq 0,
 \end{cases}
  \end{align*}
  \end{small}
where we have also used the fact that $c_0>0$ and that for fixed $x$ the argument of the argmax is a linear term $\theta(u)=-c_0u DV_{x_0}(x)$ perturbed by a strictly concave term $\nu(u)=-h(u)$ which is strictly decreasing and such that $\nu(0)=0, \nu'(0)=0$
Hence we can see $\Psi$ as a (single-valued) map, i.e. $\Psi \colon X \rightarrow U$ (so that the measurable selection is trivially $\psi=\Psi$). 
Moreover, since $D_{x_0}V$ is continuous, we have that $\Psi$ is continuous on $X.$

Next we study the solutions of the closed loop equation 
  \begin{equation}\label{eq:closed_loop_eq_example}
d Y(t) = [(\tilde A+D) Y(t)+ E\Psi(Y(t)) ] dt + \sigma \,dW(t), \quad Y(0) = x \in X,
\end{equation}
Fix any generalized probability space $ \nu=( \Omega,  {\mathcal F}, {\mathcal F}_t, \mathbb P,  W_t)$ and denote by $ Y(t)$ the unique mild solution of the (uncontrolled) equation
  \begin{equation*}
d Y(t) = (\tilde A + D) Y(t) dt + \sigma \,dW(t), \quad  Y(0) = x \in X.
\end{equation*}
Denote by $\sigma^{-1} \colon \sigma (\mathbb R) \subset X \to \mathbb R$ the inverse of the operator $\sigma \in \mathcal L(\mathbb R,X)$ and set 
\begin{small}
$$\phi \colon [0,\infty) \rightarrow \mathbb R, \quad \phi(t):=\sigma^{-1} E\Psi(Y(t)) =
\frac {c_0} {\sigma_0}\Psi(Y(t))  .$$
\end{small}
Since $|\Psi| \leq \overline u$, we have that $\phi$ is bounded so
 that \cite[Proposition 10.17 (i)]{DZ14} holds. This means that we can apply Girsanov theorem \cite[Theorem 10.14]{DZ14} to get  the existence of a probability $ \overline {\mathbb P}$ on $ \Omega$ under which 
$$\overline W_t:=-\int_0^t \phi(s) ds + W_t=-\sigma^{-1} \int_0^t E\Psi(Y(s)) ds + W_t$$
is a Wiener process. It follows that $ Y(t)$  is a  mild solution of \eqref{eq:closed_loop_eq_example} in the generalized reference probability space $\overline \nu:=( \Omega,  {\mathcal F}, {\mathcal F}_t, \overline  {\mathbb P},  \overline W_t)$. Hence $ Y(t)$  is a weak mild solution of \eqref{eq:closed_loop_eq_example}. 
Then we can apply Corollary \ref{cor:closed_loop} to get the optimality of $(u^\Psi(\cdot),Y^\Psi(\cdot))$ with $Y^\Psi(t)=Y(t)$, $u^\Psi(\cdot) :=  \Psi(Y^\Psi (\cdot))$.

 \appendix

\section{Comparison for SDDE}\label{sec:comparison}
We prove a comparison result for a class of SDDE. In particular we generalize the deterministic $1$-dimensional result \cite[Lemma 2.8]{goldys_1} to the multidimensional stochastic case with additive noise, under a more general drift $b_0$. 

In this appendix, we denote the positive part of $x \in \mathbb R$ by $x^+=\max(x,0)$; if $x \in \mathbb R^n$ this operation is understood to be component-wise. The inequalities $x < y, x  \leq y$ for $x, y \in \mathbb R^n$ are also understood component-wise, i.e. denoting by $x^i, y^i$ the $i$-th component of $x,y$, we have $x^i < y^i, x^i \leq y^i $ for every $i \leq n.$  

Fix a reference probability space $\tau=(\Omega,\mathcal{F},(\mathcal{F}_t)_{t\geq 0},
\mathbb{P}, W)$.  Let $x=(x_0,x_1) \in X, u(\cdot )  \in {\mathcal U}_\tau$ and consider the following SDDE with additive noise
\begin{small}
\begin{equation}\label{eq:ap1}
\begin{cases}
dy(t) = \ds b_0 \left ( y(t),\int_{-d}^0 a_1(\xi)y(t+\xi)\,d\xi ,u(t) \right)
         dt 
\ds  + \sigma_0  dW(t),\\
y(0)=x_0, \quad y(\xi)=x_1(\xi)\; \quad \forall \xi\in[-d,0).
\end{cases}
\end{equation}
\end{small}
 \begin{lemma}\label{lemma:comparison_sdde}
Let the standing assumptions of Section \ref{sec:formul} be satisfied  and let $a_1(\xi) \geq 0$ for $\xi\in [-d,0]$. Let $\sigma_0(x,z,u)=\sigma_0 \in M^{n \times q}$ and assume  that $b_0$ satisfies Assumption \ref{hp:state},
\begin{equation}\label{eq:assumption_b0_comparison_monotonone}
b_0^i(x,z_1,u) \leq b_0^i(x,z_2,u) \quad \forall i \leq n, z_1, z_2 \in \mathbb R^h, z_1 \leq z_2
\end{equation}
and 
\begin{equation}\label{eq:assumption_b0_comparison}
b_0^i(x,z,u)-b_0^i(y,z,u) \leq C [ | (x-y)^+| + |x^i-y^i| ] \quad \forall i \leq n, x,y \in \mathbb R^n, z \in \mathbb R^h, u \in U.
\end{equation}
  Let $y(t)$  be the solution of \eqref{eq:ap1} and let $x(t)$ satisfy 
 \begin{small}
\begin{equation}\label{eq:comparison_differential_ineq}
\begin{cases}
dx(t)=F(t)dt+\sigma_0 dW(t) \leq \ds b_0 \left ( x(t),\int_{-d}^0 a_1(\xi)x(t+\xi)\,d\xi ,u(t) \right)
         dt 
\ds  + \sigma_0  dW(t),\\
x(0)\leq x_0, \quad x(\xi)\leq x_1(\xi)\; \quad \forall \xi\in[-d,0),
\end{cases}
\end{equation}
\end{small}
where $F$ is an $\mathcal{F}_t$-progressively measurable process such that $F\in L^1_{\rm loc}([0,+\infty);\mathbb{R}^n)$, $\mathbb{P}$-a.s.
 Then $x(t)\leq y(t)$  for a.e. $\omega \in \Omega,$ $\forall t \geq 0.$
 \end{lemma}
 \begin{flushleft}
 Notice that, for instance when $n=1,$ \eqref{eq:assumption_b0_comparison} is satisfied by non decreasing or Lipschitz functions $b_0(\cdot, z,u)$.
 \end{flushleft}
\begin{proof}
By \eqref{eq:ap1} and \eqref{eq:comparison_differential_ineq}, for a.e. $\omega \in \Omega$, $\forall t \geq 0$
 \begin{small}
\begin{align*}x(t)-y(t)& =x(0)-x_0+ \int_0^t \left[ F(s) - b_0 \left ( y(s),\int_{-d}^0 a_1(\xi)y(s+\xi)\,d\xi ,u(s) \right) \right] ds\\
& \leq x(0)-x_0 + \int_0^t\left[  b_0 \left ( x(s),\int_{-d}^0 a_1(\xi)x(s+\xi)\,d\xi ,u(s) \right)-b_0 \left ( y(s),\int_{-d}^0 a_1(\xi)y(s+\xi)\,d\xi ,u(s) \right) \right] ds.
\end{align*}
\end{small}
Hence  the process $(x-y)(t)=x(t)-y(t)$ is differentiable for a.e.  $\omega \in \Omega$, for a.e. $t \geq 0$ with 
 \begin{small}
\begin{align}\label{eq:comparison_inequality_(x-y)'}
(x-y)'(t)& =F(t)- b_0 \left ( y(t),\int_{-d}^0 a_1(\xi)y(t+\xi)\,d\xi ,u(t) \right)  \nonumber\\
& \leq  b_0 \left ( x(t),\int_{-d}^0 a_1(\xi)x(t+\xi)\,d\xi ,u(t) \right)-b_0 \left ( y(t),\int_{-d}^0 a_1(\xi)y(t+\xi)\,d\xi ,u(t) \right).
\end{align}
\end{small}

We define the  $\mathbb R^n$-valued process $h(t):=(x(t)-y(t))^+ \geq 0$  and   $\overline a:=\sup_{\xi \in [-d,0]}a_1(\xi) \in M^{h \times n}$ (these operations are understood component-wise).  We show by contradiction that $h(t)=0$ for a.e. $\omega \in \Omega$, $ \forall t \geq 0.$  Hence let $r>0$ and define the $\mathbb R^n$-valued random variable $M=\sup_{t\in [0,r]}h(t) \geq 0$.  By contradiction suppose $\exists G\in \mathcal F$ with $\mathbb P(G)>0$ such that $|M(\omega)| >0$ $\forall \omega \in G $.   Note that, since $a_1 \geq 0$ (component-wise), it holds
$$\int_{-d}^0 a_1(\xi) \left [ x(t+\xi) - y(t+\xi)  \right] \,d\xi \leq \int_{-d}^0 a_1(\xi) h(t)   \,d\xi \leq d \overline a  M \quad \forall t \in [0,r] .$$
By \eqref{eq:assumption_b0_comparison_monotonone}, we have for every $i \leq n,$ for a.e. $ \omega \in G$, $\forall t \in[0, r]$  
\begin{align}\label{eq:monotonocity_comparison}
b_0^i \left ( x(t),\int_{-d}^0 a_1(\xi)x(t+\xi)\,d\xi ,u(t) \right) \leq b_0^i \left ( x(t),\int_{-d}^0 a_1(\xi)y(t+\xi)\,d\xi + \overline a  M d ,u(t) \right).
\end{align}
Define, for $N \in \mathbb{N}$,
 \begin{small}
\begin{align}\label{eq:properties_phi_N_comparison_sdde}
\varphi_N(x):= \begin{cases}0, & \quad \forall  x \leq 0, \\ N x^2, & \quad \forall x \in(0,1 / 2 N] \\ x-1 / 4 N, & \quad \forall  x>1 / 2 N .\end{cases}
\end{align}
\end{small}
The sequence $\left \{ \varphi_N\right\} _{N \in \mathbb{N}} \subset C^1( \mathbb{R})$ is such that
$$
 \begin{aligned}
& \varphi_N(x)=\varphi_N^{\prime}(x)=0 \quad  \forall x \in(-\infty, 0], N \in \mathbb{N}, \quad \quad  0 \leq \varphi_N^{\prime}(x) \leq 1 \quad  \forall x \in \mathbb{R}, N \in \mathbb{N}, \\
& \varphi_N(x) \rightarrow x^{+} \quad \text { uniformly on } x \in \mathbb{R}, \quad \quad \quad \quad  \quad   \varphi_N^{\prime}(x) \rightarrow 1\quad  \forall x \in(0,+\infty) .
\end{aligned}
$$
Fix $i \leq n.$ Noticing that $\varphi_N\left(x^i(0)-x^i_0\right) = 0 $ (as $x(0) \leq x_0$) and using \eqref{eq:comparison_inequality_(x-y)'}, \eqref{eq:monotonocity_comparison}, \eqref{eq:assumption_b0_comparison} and Assumption \ref{hp:state}, we have for a.e. $ \omega \in G$, $\forall t \in[0, r]$  
\begin{small}
$$
\begin{aligned} \varphi_N(x^i(t)&-y(^it)) =\varphi_N\left(x^i(0)-x_0^i\right)+\int_0^t \varphi_N^{\prime}(x^i(s)-y^i(s))\left(x^i-y^i\right)'(s) d s \\
& \leq  \int_0^t \varphi_N^{\prime}(x^i(s)-y^i(s)) \left[ b_0^i\left(x(s), \int_{-d}^0 a_1(\xi) x(s+\xi) d \xi,u(s)\right)-b_0^i\left(y(s), \int_{-d}^0 a_1(\xi) y(s+\xi) d \xi,u(s)\right) \right ] d s \\
& \leq  \int_0^t  \varphi_N^{\prime}(x^i(s)-y^i(s)) \left[ b_0^i \left(x(s), \int_{-d}^0 a_1(\xi) y(s+\xi) d \xi + \overline a  M d ,u(s)\right)-b_0^i \left(y(s), \int_{-d}^0 a_1(\xi) y(s+\xi) d \xi,u(s)\right) \right ] d s \\
& \leq C \int_0^t \varphi_N^{\prime}(x^i(s)-y^i(s)) \left[|h(s)|+ |x^i(s)-y^i(s)|+ |\bar{a}| |M| d\right] ds. 
\end{aligned}
$$
\end{small}
Since $\varphi_N^{\prime}(z) =0 $ when $z\leq 0$, we have $\varphi_N^{\prime}(x^i(s)-y^i(s))  |x^i(s)-y^i(s)| \leq h^i(s) $, where we have used also the fact that  $\varphi_N' \leq 1$. Hence, using again $\varphi_N' \leq 1$, we have for a.e. $ \omega \in G$, $\forall t \in[0, r]$
$$
\begin{aligned}
 \varphi_N(x^i(t)-y^i(t))& \leq C \int_0^t (|h(s)| +  h^i(s)) d s +t C |\bar{a}| |M|  d  . 
\end{aligned}
$$
Letting $N \to \infty$,  for a.e. $ \omega \in G$, $\forall t \in[0, r]$
$$
\begin{aligned}
 h^i(t)& \leq C \int_0^t (|h(s)| +  h^i(s))d s +t C |\bar{a}| |M|  d  . 
\end{aligned}
$$
By summing over $i$, it follows $\forall t \in[0, r]$ (recall that $h^i(t) \geq 0$ for every $i$)
$$
\begin{aligned} 
& |h(t)| \leq C \sum_{i=1}^n h^i(t)\leq C \int_0^t \left(|h (s) |+ \sum_{i=1}^n  h^i(s)\right)  d s +t C |\bar{a}| |M|  d   \leq C \int_0^t |h(s)|d s +r C |\bar{a}| |M| d  . 
\end{aligned}
$$
Then, using Gronwall's lemma,  we have for a.e. $ \omega \in G$, $\forall t \in[0, r]$
$$
|h(t)| \leq r C |\bar{a}| |M|  d \cdot e^{C t} \leq  r C |\bar{a}| |M| d \cdot e^{C r}   \quad \forall t \leq r.
$$
Choosing $r$ such that $r C |\bar{a}| d \cdot e^{C r}  \leq 1/2$, we have for a.e. $ \omega \in G$, $\forall t \in[0, r]$
 $$ |h(t)| \leq  |M| /2.$$
Since we assumed $|M( \omega)|=|\sup_{t\in [0,r]}h(t)(\omega)|>0$ $\forall \omega \in G$, this is a contradiction by the definition of $M$ (recall also that $h(t) \geq 0$). Hence $M=0$ for a.e. $\omega \in \Omega$, so that $h=0$  for a.e. $\omega \in \Omega$, $\forall t \in [0, r]$. Iterating the argument on intervals of the form $[i r, (i+1) r]$ for every $i \in \mathbb N$ we obtain $h = 0$ for a.e. $\omega \in \Omega$,  $\forall t \geq 0$. 
\end{proof}
\paragraph{\textbf{Acknowledgments:}} The authors are grateful to Fausto Gozzi for useful remarks related to the content of the manuscript and to the referees for their careful reading and their helpful comments.
\paragraph{\textbf{Funding:}} Filippo de Feo acknowledges support from DFG CRC/TRR 388 "Rough
Analysis, Stochastic Dynamics and Related Fields", Project B05, by INdAM (Instituto Nazionale di Alta Matematica F. Severi) - GNAMPA (Gruppo Nazionale per l'Analisi Matematica, la Probabilità e le loro Applicazioni), and by the Italian Ministry of University and Research (MUR), in the framework of PRIN projects 2017FKHBA8 001 (The Time-Space Evolution of Economic Activities: Mathematical Models and Empirical Applications) and 20223PNJ8K (Impact of the Human Activities on the Environment and Economic Decision Making in a Heterogeneous Setting: Mathematical Models and Policy Implications).

\bibliographystyle{amsplain}

\begin{thebibliography}{99}

\bibitem{bambi_fabbri_gozzi (2012)}
M. Bambi, G. Fabbri, F. Gozzi, \textit{Optimal policy and consumption smoothing effects in the time-to-build AK model.} Econom. Theory 50 (2012), no. 3, 635–669.

\bibitem{bambi_digirolami_federico_gozzi (2017)}
M. Bambi, C. Di Girolami, S. Federico, F. Gozzi, \textit{Generically distributed investments on
 flexible projects and endogenous growth.} Econom. Theory 63 (2017), no. 2, 521–558.




\bibitem{bayraktar_keller_2018}
E. Bayraktar, C.  Keller, \textit{Path-dependent Hamilton-Jacobi equations in infinite dimensions.}  J. Funct. Anal. 275 (2018), no. 8, 2096–2161.

\bibitem{bayraktar_keller_2022}
E. Bayraktar, C. Keller, \textit{Path-dependent Hamilton-Jacobi equations with super-quadratic growth in the gradient and the vanishing viscosity method.} SIAM J. Control Optim. 60 (2022), no. 3, 1690–1711.







\bibitem{delfour}
 A. Bensoussan, G. Da Prato, M. C. Delfour, S. K. Mitter, \textit{Representation and Control of Infinite Dimensional Systems.} 2nd edn., Systems and Control: Foundations and Applications, Birkh{\"a}user, Boston, 2007.


\bibitem{biffis_gozzi_prosdocini_2020}
E. Biffis, F. Gozzi, C. Prosdocimi, \textit{Optimal portfolio choice with path dependent labor income: the infinite horizon case.} SIAM J. Control Optim. 58 (2020), no. 4, 1906–1938.


\bibitem{biagini_gozzi_zanella_2022}
S. Biagini, F. Gozzi, M. Zanella, \textit{Robust portfolio choice with sticky wages.} SIAM J. Financial Math. 13 (2022), no. 3, 1004–1039.



\bibitem{can-fr}  
P. Cannarsa, H. Frankowska, \textit{Value function and optimality conditions for semilinear control problems}. Appl. Math. Optim. 26 (1992), no. 2, 139–169.

\bibitem{carlier_tahraoui_2010}
G. Carlier, R. Tahraoui, \textit{Hamilton-Jacobi-Bellman equations for the optimal control of a state equation with memory.} ESAIM Control Optim. Calc. Var. 16 (2010), no. 3, 744–763.

\bibitem{Chen_wu_2010}
L. Chen, Z. Wu, \textit{Maximum principle for the stochastic optimal control problem with delay and application.} Automatica J. IFAC 46 (2010), no. 6, 1074–1080.

\bibitem{Chen_wu_2011}
L. Chen, Z. Wu, \textit{A type of general forward-backward stochastic differential equations and applications.} Chin. Ann. Math. Ser. B 32 (2011), no. 2, 279–292.



\bibitem{ChenLu}
L. Chen, Q. L\"u, \textit{Stochastic verification theorem for infinite dimensional stochastic control systems.} arXiv preprint,
arXiv:2209.09576v1 (2022).


\bibitem{choj78}
A. Chojnowska-Michalik, \textit{Representation theorem for general stochastic delay equations.} Bull. Acad. Polon. Sci. Sér. Sci. Math. Astronom. Phys. 26 (1978), no. 7, 635–642.


\bibitem{cosso_federico_gozzi_rosestolato_touzi}
A. Cosso, S. Federico, F. Gozzi, M. Rosestolato, N. Touzi, \textit{ Path-dependent equations and viscosity solutions in infinite dimension.} Ann. Probab. 46 (2018), no. 1, 126–174.





\bibitem{DZ14}
G. Da Prato, J. Zabczyk, \textit{Stochastic Equations in Infinite Dimensions}. Encyclopedia of Mathematics and its Applications, vol. 152, Cambridge University Press, Cambridge, 2014.



\bibitem{defeo_federico_swiech}
F. de Feo, S. Federico, A. \'{S}wi\k{e}ch, \textit{Optimal control of stochastic delay differential equations and applications to path-dependent financial and economic models.} SIAM J. Control Optim. 62 (2024), no. 3, 1490–1520.

\bibitem{deFeo_2023}
F. de Feo, \textit{Stochastic optimal control problems with delays in the state and in the control via viscosity solutions and applications to optimal advertising and optimal investment problems}. Decis. Econ. Finance (2024) 31 pp.



\bibitem{defeo_swiech_wessels}
F. de Feo, A. \'{S}wi\k{e}ch, L. Wessels, \textit{Stochastic optimal control in Hilbert spaces: $C^{1,1}$ regularity of the value function and optimal synthesis via viscosity solutions.} 
arXiv preprint, arXiv:2310.03181 (2023).





\bibitem{djehiche_gozzi_zanco_zanella_2022}
B. Djehiche, F. Gozzi, G. Zanco, M. Zanella, \textit{ Optimal portfolio choice with path dependent benchmarked labor income: a mean field model.} Stochastic Process. Appl. 145 (2022), 48–85.

 \bibitem{ekren_2014}
 I. Ekren, C. Keller, N. Touzi, J. Zhang, \textit{On viscosity solutions of path dependent PDEs.} Ann. Probab. 42 (2014), no. 1, 204–236. 
  
  \bibitem{ekren_2016}
I. Ekren, N. Touzi, J. Zhang, \textit{Viscosity solutions of fully nonlinear parabolic path dependent PDEs: Part I.} Ann. Probab. 44 (2016), no. 2, 1212–1253.

  \bibitem{ekren_2016b}
I. Ekren, N. Touzi, J. Zhang, \textit{Viscosity solutions of fully nonlinear parabolic path dependent PDEs: Part II.} Ann. Probab. 44 (2016), no. 4, 2507–2553.



    \bibitem{fabbri_gozzi_2008}
G. Fabbri, F. Gozzi, \textit{Solving optimal growth models with vintage capital: the dynamic programming approach.} J. Econom. Theory 143 (2008), no. 1, 331–373.

  
    \bibitem{fgs_2008}
G. Fabbri, F. Gozzi, A. \'{S}wi\k{e}ch, \textit{Verification theorem and construction of $\varepsilon $-optimal controls for control of abstract evolution equations.} J. Convex Anal. 17 (2010), no. 2, 611–642.

  
  \bibitem{fgs_book}
  G. Fabbri, F. Gozzi, A. \'{S}wi\k{e}ch, \textit{Stochastic optimal control in infinite dimension. Dynamic programming and HJB equations. With a contribution by Marco Fuhrman and Gianmario Tessitore.} Probability Theory and Stochastic Modelling, 82, Springer, Cham, 2017.
  
  

  

  \bibitem{goldys_1}
S. Federico, B. Goldys, F. Gozzi, \textit{HJB equations for the optimal control of differential equations with delays and state constraints, I: regularity of viscosity solutions.} SIAM J. Control Optim. 48 (2010), no. 8, 4910–4937.

 \bibitem{goldys_2}
S. Federico, B. Goldys, F. Gozzi, \textit{HJB equations for the optimal control of differential equations with delays and state constraints, II: verification and optimal feedbacks.} SIAM J. Control Optim. 49 (2011), no. 6, 2378–2414.

\bibitem{federico_gozzi_2018}
S. Federico, F. Gozzi, \textit{Verification theorems for stochastic optimal control problems in Hilbert spaces by means of a generalized Dynkin formula.} Ann. Appl. Probab. 28 (2018), no. 6, 3558–3599.

\bibitem{tacconi}
S. Federico, E. Tacconi, \textit{Dynamic programming for optimal control problems with delays in the control variable.} SIAM J. Control Optim. 52 (2014), no. 2, 1203–1236.


\bibitem{tankov_federico}
 S. Federico, P. Tankov, \textit{Finite-dimensional representations for controlled diffusions with delay.} Appl. Math. Optim. 71 (2015), no. 1, 165–194.



\bibitem{FS}
W. H.~Fleming, H. M.~Soner, \textit{Controlled Markov processes and viscosity solutions.} 2nd edn., Stochastic Modelling and Applied Probability, 25, Springer, New York, 2006.


\bibitem{fuhrman_tessitore_masiero_2010}
M. Fuhrman, F. Masiero, G. Tessitore, \textit{Stochastic equations with delay: optimal control via BSDEs and regular solutions of Hamilton-Jacobi-Bellman equations.} SIAM J. Control Optim. 48 (2010), no. 7, 4624–4651.




  
\bibitem{gozzi_marinelli_2004}
F. Gozzi, C. Marinelli, \textit{Stochastic optimal control of delay equations arising in advertising models.} Stochastic partial differential equations and applications—VII, 133–148, Lect. Notes Pure Appl. Math., 245, Chapman $\&$ Hall/CRC, Boca Raton, FL, 2006.



\bibitem{gozzi_masiero_2017}
F. Gozzi, F. Masiero, \textit{Stochastic optimal control with delay in the control I: Solving the HJB equation through partial smoothing.} SIAM J. Control Optim. 55 (2017), no. 5, 2981–3012.


\bibitem{gozzi_masiero_2017_b}
F. Gozzi, F. Masiero, \textit{Stochastic optimal control with delay in the control II: Verification theorem and optimal feedbacks.} SIAM J. Control Optim. 55 (2017), no. 5, 3013–3038.

\bibitem{gozzi_masiero_2022}
F. Gozzi, F. Masiero. \textit{Errata: Stochastic Optimal Control with Delay in the Control I: Solving the HJB Equation through Partial Smoothing, and Stochastic Optimal Control with Delay in the Control II: Verification Theorem and Optimal Feedbacks.} SIAM J. Control Optim. 59 (2021), no. 4 , 3096--3101. 


\bibitem{gozzi_swiech_zhou_2005}
F. Gozzi, A. \'{S}wi\k{e}ch, X.Y. Zhou, \textit{A corrected proof of the stochastic verification theorem within the framework of viscosity solutions.} SIAM J. Control Optim. 43 (2005), no. 6, 2009–2019.


\bibitem{gozzi_swiech_zhou_2010}
F. Gozzi, A. \'{S}wi\k{e}ch, X. Y. Zhou, \textit{Erratum: "A corrected proof of the stochastic verification theorem within the framework of viscosity solutions''.} SIAM J. Control Optim. 48 (2010), no. 6, 4177–4179.



\bibitem{guatteri_masiero_2021}
G. Guatteri, F. Masiero, \textit{Stochastic maximum principle for problems with delay with dependence on the past through general measures}. Math. Control Relat. Fields 11 (2021), no. 4, 829-855.


\bibitem{guatteri_masiero_2023}
G. Guatteri, F. Masiero, \textit{Stochastic maximum principle for equations with delay: going to infinite dimensions to solve the non-convex case.} arXiv preprint arXiv:2306.07422 (2023).


\bibitem{larssen2}
B. Larssen, N. H. Risebro, \textit{When are HJB-equations in stochastic control of delay systems finite dimensional?} Stochastic Anal. Appl. 21 (2003), no. 3, 643–671.

\bibitem{lasry}
J. M. Lasry, P. L. Lions, \textit{A remark on regularization in Hilbert spaces.} Israel J. Math. 55 (1986), no. 3, 257--266.


  \bibitem{LY}
X. J. Li, J. M. Yong, \textit{Optimal Control Theory for Infinite-Dimensional Systems.} Systems and Control: Foundations and Applications (Birkh{\"a}user, Boston, 1995).

\bibitem{lions_1983}
P. L. Lions, \textit{Optimal control of diffusion processes and Hamilton–Jacobi–Bellman equations. II. Viscosity solutions and uniqueness.} Comm. Partial Differential Equations 8 (1983), no. 11, 1229–1276.






\bibitem{lions-infdim1}
P. L. Lions, \textit{Viscosity solutions of fully nonlinear second-order equations and optimal stochastic control in infinite dimensions. I. The case of bounded stochastic evolutions}. Acta Math. 161 (1988), no. 3-4, 243–278.

  


  



  
    \bibitem{masiero_2022}
F.  Masiero, G. Tessitore, \textit{Partial smoothing of delay transition semigroups acting on special functions.} J. Differential Equations 316 (2022), 599–640.




  
      \bibitem{mayorga_swiech_2022}
  S. Mayorga, A. \'{S}wi\k{e}ch, \textit{Finite dimensional approximations of Hamilton-Jacobi-Bellman equations for stochastic particle systems with common noise}. SIAM J. Control Optim. 61 (2023), no. 2, 820–851.
  
  
    \bibitem{mendelson}
  E. Mendelson,  \textit{Introduction to mathematical logic.} Chapman and Hall/CRC, 2009.
  
  
     \bibitem{meng_shi_2021}
  W. Meng, J. Shi, \textit{A global maximum principle for stochastic optimal control problems with
delay and applications.} Systems Control Lett. 150 (2021).
  
  \bibitem{oksendal_sulem}
B. Oksendal, A. Sulem, \textit{A maximum principle for optimal control of stochastic systems with delay, with applications to finance.} Optimal Control and Partial Differential Equations. J. M. Menaldi, E. Rofman, A. Sulem (eds.), ISO Press, Amsterdam (2000), 64–79.
  
  \bibitem{oksendal_sulem_zhang}
B. Oksendal, A. Sulem, T. Zhang, \textit{Optimal control of stochastic delay equations and time-advanced backward stochastic differential equations.} Adv. Appl. Probab. 43 (2011), no. 2, 572 - 596


\bibitem{peng_yang} 
S. Peng, Z.  Yang, \textit{Anticipated backward stochastic differential equations}. Ann.  Prob. 37 (2009), no. (3), 877-902.

  
  
  \bibitem{ren_touzi_zhang} Z. Ren, N. Touzi, J. Zhang,  \textit{Comparison of viscosity solutions of fully nonlinear degenerate parabolic path-dependent PDEs.} SIAM J. Math. Anal. 49 (2017), no. 5, 4093--4116.

\bibitem{ren_rosestolato}
Z. Ren, M. Rosestolato, \textit{Viscosity solutions of path-dependent PDEs with randomized time.}  SIAM J. Math. Anal. 52 (2020), no. 2, 1943--1979.


\bibitem{revuz}
D. Revuz, M. Yor, \textit{Continuous Martingales and Brownian Motion}. 3rd edition, Grundlehren der Mathematischen Wissenschaften, vol. 293, Springer, Berlin, 1999.


\bibitem{rosestolato}
M. Rosestolato, A.  \'{S}wi\k{e}ch, \textit{Partial regularity of viscosity solutions for a class of Kolmogorov equations arising from mathematical finance.} J. Differential Equations 262 (2017), no. 3, 1897–1930.


 
  \bibitem{stannat_wessels_2021}
W. Stannat, L. Wessels, \textit{Necessary and Sufficient Conditions for Optimal Control of Semilinear Stochastic Partial Differential Equations.} Ann. Appl. Probab. 34 (2024), no. 3, 3251–3287.
 


\bibitem{yong_zhou}
J. Yong, X. Y. Zhou, \textit{Stochastic Controls, Hamiltonian Systems and HJB Equations}. Applications of Mathematics, vol. 43, Springer, New York, 1999.

\bibitem{zhou_2018}
J. Zhou, \textit{A class of infinite-horizon stochastic delay optimal control problems and a viscosity solution to the associated HJB equation.} ESAIM Control Optim. Calc. Var. 24 (2018), no. 2, 639–676.

\bibitem{zhou_2019}
J. Zhou,  \textit{Delay optimal control and viscosity solutions to associated Hamilton–Jacobi–Bellman equations.} Internat. J. Control 92 (2019), no. 10, 2263–2273.

\bibitem{zhou_2021}
J. Zhou, \textit{A notion of viscosity solutions to second-order Hamilton–Jacobi–Bellman equations with delays.} Internat. J. Control 95 (2022), no. 10, 2611–2631.










\end{thebibliography}

\end{document}